%% file: root.tex
\documentclass[journal,twoside,web]{ieeecolor}
\usepackage{generic}

\usepackage{graphicx}          

\usepackage{textcomp}
\def\BibTeX{{\rm B\kern-.05em{\sc i\kern-.025em b}\kern-.08em
		T\kern-.1667em\lower.7ex\hbox{E}\kern-.125emX}}
\usepackage{hyperref}	
\usepackage{xcolor}		

\usepackage{amsfonts}   
\usepackage{amsmath}    
\usepackage{amssymb}
\usepackage{bm}         
\usepackage{algorithm}	
\usepackage{algpseudocode}
\usepackage{svg}	
\usepackage{subfig}	
\usepackage{setspace}
\usepackage{amsthm}     
\newtheorem{theorem}{Theorem}
\newtheorem{proposition}{Proposition}
\newtheorem{lemma}{Lemma}
\newtheorem{definition}{Definition}
\theoremstyle{definition}

\theoremstyle{remark}

\newtheorem{remark}{Remark}
\newtheorem{assumption}{Assumption}

\usepackage{tikz}
\usetikzlibrary{spy}
\usepackage[utf8]{inputenc}
\usepackage{pgfplots}
\usepgfplotslibrary{groupplots,dateplot}
\usetikzlibrary{patterns,shapes.arrows}
\pgfplotsset{compat=newest}
\def\axisdefaultheight{110pt}

\DeclareMathOperator*{\argmin}{arg\,min}

\begin{document}
	
\title{Distributed Model Predictive Control for Piecewise Affine Systems Based on Switching ADMM}
\newcommand{\mytitle}{Distributed MPC for PWA Systems Based on Switching ADMM}
\hypersetup{hidelinks,pdfauthor=Samuel Mallick, pdfcreator=Samuel Mallick, pdftitle=\mytitle}

\author{Samuel Mallick, Azita Dabiri, and Bart De Schutter, \IEEEmembership{Fellow, IEEE}
	\thanks{This paper is part of a project that has received funding from the European Research Council (ERC) under the European Union’s Horizon
		2020 research and innovation programme (Grant agreement No. 101018826
		- CLariNet).}
	\thanks{All authors are affiliated with Delft Center for Systems and Control, Delft
		University of Technology, Delft, The Netherlands (e-mail: \{s.h.mallick,a.dabiri,b.deschutter\}@tudelft.nl).}}



\maketitle

\begin{abstract}
	This paper presents a novel approach for distributed model predictive control (MPC) for piecewise affine (PWA) systems.
	Existing approaches rely on solving mixed-integer optimization problems, requiring significant computation power or time.
	We propose a distributed MPC scheme that requires solving only convex optimization problems.
	The key contribution is a novel method, based on the alternating direction method of multipliers, for solving the non-convex optimal control problem that arises due to the PWA dynamics.
	We present a distributed MPC scheme, leveraging this method, that explicitly accounts for the coupling between subsystems by reaching agreement on the values of coupled states.
	Stability and recursive feasibility are shown under additional assumptions on the underlying system.
	Two numerical examples are provided, in which the proposed controller is shown to significantly improve the CPU time and closed-loop performance over existing state-of-the-art approaches.
\end{abstract}
\begin{IEEEkeywords}
	Distributed model predictive control, networked systems, piecewise affine systems, alternating direction method of multipliers (ADMM)
\end{IEEEkeywords}
\section{Introduction}

Distributed control of large-scale hybrid systems is an important challenge in the field of systems and control \cite{lamnabhi-lagarrigueSystemsControlFuture2017}.
Hybrid systems are dynamical systems exhibiting both continuous and discrete dynamics.
They are powerful models for many real-world large-scale systems such as transportation networks \cite{luanDecompositionDistributedOptimization2020}, power networks \cite{mendesPracticalApproachHybrid2017}, and water networks \cite{vanekerenTimeInstantOptimizationHybrid2013}.
In particular, piecewise affine (PWA) models are a popular class of hybrid system model due their suitability for analysis \cite{sontagNonlinearRegulationPiecewise1981}, proven equivalence to other hybrid model classes \cite{heemelsEquivalenceHybridDynamical2001a}, and ability to approximate non-linear functions to arbitrary accuracy.

Distributed model predictive control (MPC) \cite{maestreDistributedModelPredictive2014} is a promising candidate for control of large-scale PWA systems.
MPC is an optimization-based paradigm where control inputs are generated online by minimizing a sum of stage costs over a prediction horizon, subject to the system dynamics, and physical and operational constraints \cite{mayneConstrainedModelPredictive2000}.
Distributed MPC applies this idea to spatially separated or decomposed systems, where subsystem controllers solve, sometimes iteratively, local optimization problems, and where the coupling between connected subsystems is accounted for by communication \cite{maestreDistributedModelPredictive2014}.
For linear systems, this approach can reconstruct the solution to a globally defined convex MPC controller through the use of iterative distributed optimization algorithms, e.g., the alternating direction method of multipliers (ADMM) \cite{boydDistributedOptimizationStatistical2010}.
In this case, as the globally defined controller considers the global system, the effect of coupling between subsystems is accounted for and the stability and performance of the globally defined MPC controller carries through to the distributed implementation.
For PWA systems, however, the PWA dynamics result in a non-convex global optimization problem for which distributed optimization algorithms, such as ADMM, are not guaranteed to recover the global optimum, nor even to find a feasible solution.
As such, the effects of coupling between subsystems may cause instability or negatively affect the performance of the resulting distributed controller.
Furthermore, the local optimization problems are also non-convex, and the computational burden of solving them iteratively can become prohibitive, particularly as their computational complexity grows exponentially with the size of the prediction horizon \cite{bemporadControlSystemsIntegrating1999}.
Therefore, two key challenges for distributed MPC for PWA systems are: how to reduce the online computational burden, and how to account for the effects of coupling between subsystems.

The few existing approaches for distributed MPC for PWA systems rely on converting the PWA dynamics into mixed logical dynamical (MLD) form, converting the MPC optimization problems into mixed-integer quadratic problems (MIQPs) \cite{bemporadControlSystemsIntegrating1999}.
Distributed control is then conducted by the subsystems solving local MIQPs, iteratively in parallel \cite{grossDistributedPredictiveControl2013}, or in a sequential order \cite{kuwataDistributedRobustReceding2007}, with communication of the local solutions between subsystems handling the coupling.
However, these approaches suffer from the computational complexity of solving multiple MIQPs.
To reduce the online computational burden, \cite{mendesPracticalApproachHybrid2017} proposes to select values for the integers in the MIQPs with heuristics, using convex distributed MPC techniques to then solve for the values of the continuous variables.
However, the heuristic choice of integer solutions is in general highly suboptimal.
Perhaps the most theoretically complete approach was proposed recently in \cite{maDistributedMPCBased2023}, where distributed MPC for PWA systems is presented using robustly controllable sets.
The computational complexity of the local MIQPs is reduced by considering a prediction horizon of length one, and convergence to a terminal set is proven.
However, the coupling between subsystems is handled using robustness techniques, where subsystems consider coupling effects as disturbances, rather than agreeing on the values of the shared states.
This adds conservativeness, reducing performance and limiting the approach to systems with weak coupling.

In light of these challenges, this paper presents a novel approach for distributed MPC for PWA systems that systematically handles the coupling between subsystems and relies only on convex optimization.
The key contribution, enabling the approach, is a novel distributed method for solving the global non-convex optimization problem arising in the globally formulated MPC controller.
This method is based on ADMM, and leverages the underlying PWA dynamics of the system to solve the non-convex optimization problem distributively using only convex optimization.
Furthermore, the coupling between subsystems is handled explicitly, with the distributed solutions provably converging to points where the values of all shared states are agreed upon.
As such, the proposed method can improve upon the performance of existing controllers that heuristically handle coupling between subsystems, and presents a significantly lower computational burden than approaches based on mixed-integer optimization.

This paper is organized as follows.
Section \ref{sec:pre} gives the preliminaries, formalizing the systems considered and giving background on centralized and distributed MPC.
Section \ref{sec:Sw} presents the switching ADMM procedure and the distributed MPC controller, with the stabilizing properties proved in Section \ref{sec:stab_feas}.
In Section \ref{sec:ex} two numerical examples demonstrate the effectiveness of the approach, and, finally, Section \ref{sec:conclusion} concludes the paper.

\section{Preliminaries}
\label{sec:pre}
\subsection{Notation}
Define the following sets, all subsets of $\mathbb{N}$, as $\mathcal{M} = \{1,\dots,M\}$, $\mathcal{L}_i = \{1, \dots, L_i\}$, $\mathcal{K} = \{0,\dots,N-1\}$, and $\mathcal{K}' = \{0,\dots,N\}$.
We use the counter $t$ to represent closed-loop time steps and $k$ for time steps within an MPC prediction horizon.
A vector that stacks the vectors $x_i$, $i \in \mathcal{M}$, in one column vector is denoted $\text{col}_{i\in \mathcal{M}}(x_i)$.
Define the closure of set $\mathcal{P}$ as $\overline{\mathcal{P}}$.
For brevity we write $(A)^\top B (A)$ as $(\star)^\top B (A)$.

\subsection{Distributed Piecewise Affine Systems}
We consider a global system composed of $M$ subsystems where each subsystem $i \in \mathcal{M}$ has a state $x_i \in \mathbb{R}^{n_i}$ and an input $u_i \in \mathbb{R}^{m_i}$. 
Both states and inputs are subject to convex constraints
\begin{equation}
	x_i \in \mathcal{X}_i, \quad u_i \in \mathcal{U}_i, \: \forall i \in \mathcal{M}.
\end{equation}
The graph $\mathcal{G} = (\mathcal{M}, \mathcal{E})$ defines a coupling topology where the edges $\mathcal{E}$ are ordered pairs $(i, j)$ indicating that subsystem $i$ effects subsystem $j$ through the dynamics, stage cost, constraints, or some combination thereof. 
Define the neighborhood of subsystem $i$ as $\mathcal{N}_i = \{j \in {M} | (j, i) \in \mathcal{E}, i \neq j\}$.
A subsystem is not in its own neighborhood, i.e., $(i,i) \notin \mathcal{E}$.
It is assumed that subsystem $i$ and subsystem $j$ can communicate in a bidirectional manner if $i \in \mathcal{N}_j$ or $j \in \mathcal{N}_i$.

The subsystems are governed by discrete-time piecewise affine (PWA) dynamics\footnote{\begin{color}{black}As only discrete-time systems are considered, hybrid phenomena such as arbitrarily fast switching and Zeno behavior are not relevant \cite{bemporad2003modeling}.\end{color}}, with the state update equations $f_i$ affine on a finite number of convex polytopes, $\{P_i^{(l)}\}_{l \in \mathcal{L}_i}$, referred to as regions:
\begin{equation}
	\begin{aligned}
		\label{eq:dynamics}
		x_i^+ &= f_i\big(x_i, u_i, \{x_j\}_{j \in \mathcal{N}_i}\big) \\
		 &= A_{i}^{(l)} x_i + B_i^{(l)} u_i + c_i^{(l)} \\
		 &\quad + \sum_{j \in \mathcal{N}_i} A_{ij}^{(l)} x_j, \: x_i \in P_i^{(l)}, \forall l \in \mathcal{L}_i.
	\end{aligned}
\end{equation}
\begin{color}{black}The regions have non-overlapping interiors, i.e., $\text{int}(P_i^{(l)}) \cap \text{int}(P_i^{(l^\prime)}) = \emptyset$ for $l,l' \in \mathcal{L}_i, l \neq l'$, and form a partition of the state space, i.e., $\bigcup_{l \in \mathcal{L}_i} P_i^{(l)} = \mathcal{X}_i$.\end{color}
Define the global state and input, the aggregation of states and inputs for all subsystems, as $x = \text{col}_{i \in \mathcal{M}}(x_i)$ and $u = \text{col}_{i \in \mathcal{M}}(u_i)$.
Furthermore, define variables that gather states and inputs over the prediction horizon: $\textbf{x}_i = \big(x_i^\top(0),...,x_i^\top(N)\big)^\top$, $\textbf{u}_i = \big(u_i^\top(0),...,u_i^\top(N-1)\big)^\top$, $\textbf{x} = \big(x^\top(0),...,x^\top(N)\big)^\top$, $\textbf{u} = \big(u^\top(0),...,u^\top(N-1)\big)^\top$.

\subsection{Centralized MPC}

Consider the global MPC controller defined by the following finite-horizon optimal control problem parametrized by $x$:
\begin{subequations}
	\begin{align}
		\begin{split}
			\mathcal{P}(x): \: \min_{\textbf{x},\textbf{u}} \: \sum_{i \in \mathcal{M}} F_i\big( \textbf{x}_i, \textbf{u}_i, \{\textbf{x}_j\}_{j \in \mathcal{N}_i} \big)
		\end{split}\\
		\text{s.t.}& \quad \forall  i \in \mathcal{M}: \nonumber \\
		&x_i(k+1) = f_i\big(x_i(k), u_i(k), \{x_j(k)\}_{j \in \mathcal{N}_i}\big), \: \forall k \in \mathcal{K} \\
		&h_i\big(x_i(k), \{x_j(k)\}_{j \in \mathcal{N}_i}\big) \leq 0, \: \forall k \in \mathcal{K}'\\
		&\big(x_i(k), u_i(k)\big) \in \mathcal{X}_i \times \mathcal{U}_i, \: \forall k \in \mathcal{K}' \label{eq:cent_loc_cnstrs}\\
		&x_i(0) = x_{i}, \label{eq:cent_IC}
	\end{align}
	\label{eq:MPC_opt_prob}\end{subequations}
with
\begin{equation}
	\begin{aligned}
		F_i&\big( \textbf{x}_i, \textbf{u}_i, \{\textbf{x}_j\}_{j \in \mathcal{N}_i} \big) = \\
		&\sum_{k \in \mathcal{K}} \ell_i\big(x_i(k), u_i(k), \{x_j(k)\}_{j \in \mathcal{N}_i}\big) + V_{\text{f}, i}\big(x_i(N)\big).
	\end{aligned}
\end{equation}
The convex functions $\ell_i$ and $V_{\text{f}, i}$ are stage and terminal costs, and the linear functions $h_i$ define coupled convex inequality constraints.
The first entries in the optimal input sequences define feedback control laws $u_{\text{MPC}, i}(x) = u^\star_i(0)$.

Define the set of states $x$ for which $\mathcal{P}$ is feasible as $\mathcal{X}_0$.
For a global state $x \in \mathcal{X}_0$ and control sequence $\textbf{u}$ define the value as
\begin{equation}
	\label{eq:cost}
	V(x, \textbf{u}) = \sum_{i \in \mathcal{M}} F_i\big( \textbf{x}_i, \textbf{u}_i, \{\textbf{x}_j\}_{j \in \mathcal{N}_i} \big),
\end{equation}
if, for all $i \in \mathcal{M}$ and $k \in \mathcal{K}'$, $h_i\big(x_i(k), \{x_j(k)\}_{j \in \mathcal{N}_i}\big) \leq 0$ and $\big(x_i(k), u_i(k)\big) \in \mathcal{X}_i \times \mathcal{U}_i$, and $V(x, \textbf{u}) = \infty$ otherwise, with the state sequences $\textbf{x}_i$ generated by applying $\textbf{u}$ to the dynamics \eqref{eq:dynamics} starting from $x$.
\begin{color}{black}Define a globally feasible set of local control sequences as follows: 
\begin{definition}[Globally feasible local control sequences]
	The set of local control sequences $\{\textbf{u}_i\}_{i \in \mathcal{M}}$ is \textit{globally feasible} for $x$ if the corresponding global control sequence $\textbf{u}$ is feasible for $\mathcal{P}(x)$, i.e., $V(x, \textbf{u}) < \infty$.
\end{definition}\end{color}

The MPC scheme defined by $\mathcal{P}$ considers general coupling between subsystems, i.e., in the costs, dynamics, and constraints.
Solving $\mathcal{P}$ directly is undesirable as the complexity of the optimization problem grows with the number of subsystems.
Additionally, a centralized solution requires centralized coordination, which cannot preserve privacy of the local functions of subsystems, and presents a single point of failure.
It is therefore desirable to solve $\mathcal{P}$ distributively, where each subsystem computes a local control input using only local information and communication with its neighbors.
\begin{remark}
	The approach presented in this paper extends easily to the case where the cost functions $F_i$ also switch depending on the subsystem's PWA region.
For clarity of exposition we do not explicitly include this case in $\mathcal{P}$, as the notation would then become unwieldy.
\end{remark}

\subsection{Distributed MPC}
By introducing copies of coupled states the problem $\mathcal{P}$ can be reformulated for distributed optimization \cite{summersDistributedModelPredictive2012, giselssonFeasibilityStabilityPerformance2014, kohlerDistributedModelPredictive2019}.
Define an augmented state trajectory for each subsystem that includes copies of the state trajectories of neighboring subsystems
\begin{equation}
	\Tilde{\textbf{x}}_i = \big(\textbf{x}_i^\top, \text{col}^\top_{j \in \mathcal{N}_i}(\textbf{x}_j^{(i)})\big)^\top,
\end{equation}
where $\textbf{x}_j^{(i)}$ is agent $i$'s local copy of agent $j$'s state trajectory.
Furthermore, define a global augmented state as $\Tilde{\textbf{x}} = \text{col}_{i \in \mathcal{M}}(\Tilde{\textbf{x}}_i)$.
An equivalent problem to $\mathcal{P}$ is then
\begin{subequations}
	\begin{align}
		\begin{split}
			\mathcal{P}_\text{d}(x): \: \min_{\Tilde{\textbf{x}}, \textbf{u}} \: \sum_{i \in \mathcal{M}} F_i\big( \textbf{x}_i, \textbf{u}_i, \{\textbf{x}_j^{(i)}\}_{j \in \mathcal{N}_i} \big)
		\end{split}\\		
		\text{s.t.}&\quad \forall i \in \mathcal{M}: \nonumber\\
		&x_i(k+1) = f_i\big(x_i(k), u_i(k), \{x_j^{(i)}(k)\}_{j \in \mathcal{N}_i}\big),  \forall k \in \mathcal{K} \label{eq:dynam_cnstrs_dist} \\
		&h_i\big(x_i(k), \{x_j^{(i)}(k)\}_{j \in \mathcal{N}_i}\big) \leq 0, \: \forall k \in \mathcal{K}' \label{eq:couple_cnstr_dist}\\
		& \eqref{eq:cent_loc_cnstrs}, \eqref{eq:cent_IC} \label{eq:dist_local_cnstr}\\\
		&x_j^{(i)}(k) = x_j(k), \: j \in \mathcal{N}_i, \: \forall k \in \mathcal{K}'. \label{eq:consistency_constraint}
	\end{align}
\label{eq:MPC_opt_prob_dist}\end{subequations}
The equality constraint \eqref{eq:consistency_constraint} gives consistency between the true state trajectories and their copies and therefore equivalence between $\mathcal{P}$ and $\mathcal{P}_\text{d}$. 
In $\mathcal{P}_\text{d}$ the stage costs, dynamics, and inequality constraints can be separated across the local variables $\Tilde{\textbf{x}}_i$ and $\textbf{u}_i$.
The alternating direction method of multipliers (ADMM) \cite{boydDistributedOptimizationStatistical2010} can then be applied to solve $\mathcal{P}_\text{d}$ distributively by dualizing the consistency constraint \eqref{eq:consistency_constraint}, and having systems iteratively solve and communicate the solutions to local sub-problems derived from $\mathcal{P}_\text{d}$.
However, due to the PWA dynamics \eqref{eq:dynam_cnstrs_dist} $\mathcal{P}_\text{d}$ is non-convex, and ADMM is hence not guaranteed to converge to an optimal, or even feasible, solution.
Additionally, systems must solve non-convex local optimization problems iteratively, requiring significant computation time or power.
To address this, in the following section we introduce distributed MPC based on a switching ADMM procedure that, leveraging the structure of the PWA dynamics, can guarantee convergence to a feasible solution, and requires solving only convex optimization problems.

\section{Switching ADMM-based Distributed MPC}\label{sec:Sw}
In this section we introduce the proposed distributed MPC algorithm.
First, we show that the structure of $\mathcal{P}_\text{d}$ is piecewise-convex over a finite collection of polytopes.
This structure is then leveraged to formulate the switching ADMM procedure upon which the distributed MPC algorithm is based.

\subsection{Structure of $\mathcal{P}_\text{d}$}

The following analysis is inspired by the \textit{reverse transformation} presented in \cite{mayneModelPredictiveControl2003}, where MPC for a single PWA system was considered.
Define concatenated decision variables as $\textbf{y}_i = (\Tilde{\textbf{x}}_i^\top, \textbf{u}_i^\top)^\top$ and $\textbf{y} = (\Tilde{\textbf{x}}^\top, \textbf{u}^\top)^\top$.
The feasible set for $\mathcal{P}_\text{d}(x)$ is then
\begin{equation}
	\label{eq:feasible_P}
	\mathcal{Y}(x) = \big\{\textbf{y} = (\Tilde{\textbf{x}}^\top, \textbf{u}^\top)^\top \big| \text{\eqref{eq:dynam_cnstrs_dist}\--\eqref{eq:consistency_constraint}}, \forall i \in \mathcal{M}\big\}.
\end{equation}
We now introduce the notion of \textit{switching sequences} that specify a sequence of PWA regions over the prediction horizon.
Consider the switching sequence $s_i \in \mathcal{S}_i = \{1,...,L_i\}^{N+1}$ that specifies the PWA regions for subsystem $i$ as:
\begin{equation}
	\begin{aligned}
		&s_i = s_i(0), s_i(1), ..., s_i(N) \\
		 &\implies x_i(k) \in P_i^{\big(s_i(k)\big)}, \: \forall k \in \mathcal{K}'.
	\end{aligned}
\end{equation}
and the corresponding time-varying linear dynamics
\begin{equation}
	\label{eq:dynams_seq}
	\begin{aligned}
		x_i(k+1) &= {f}_{i}^{(s_i)}\big(x_i(k), u_i(k), \{x_j^{(i)}(k)\}_{j \in \mathcal{N}_i}\big) \\
		 &= A_{i}^{\big(s_i(k)\big)} x_i(k) + B_i^{\big(s_i(k)\big)} u_i(k) \\
		 &\quad + c_i^{\big(s_i(k)\big)} + \sum_{j \in \mathcal{N}_i} A_{ij}^{\big(s_i(k)\big)} x_j^{(i)}(k), \: \forall k \in \mathcal{K}.
	\end{aligned}
\end{equation}
A particular $\textbf{x}_i$ is said to \textit{generate} a switching sequence $s_i$ if 
\begin{equation}
	x_i(k) \in \overline{P}_i^{\big(s_i(k)\big)}, \: \forall k \in \mathcal{K}'.
\end{equation}
Note that more than one $s_i$ can be generated by $\textbf{x}_i$ if $x_i(k)$ lies on the boundary of at least two PWA regions for some $k$, i.e., if $x_i(k) \in \overline{P}_i^{(a)} \cap \overline{P}_i^{(b)}$ then two sequences are generated, one containing $s_i(k) = a$ and the other $s_i(k) = b$.
For a given $x_i$, $\textbf{u}_i$, and $\{\textbf{x}_j^{(i)}\}_{j \in \mathcal{N}_i}$ subsystem $i$ can evaluate the generated switching sequences by rolling out and branching its dynamics as in Algorithm \ref{alg:eval_switching}.
\begin{algorithm}
	\caption{$\text{Eval-switching}$}\label{alg:eval_switching}
	\begin{algorithmic}[1]
		\State \textbf{Inputs}: State $x_i$, input sequence $\textbf{u}_i$, and coupled state copies $\{\textbf{x}_j^{(i)})\}_{j \in \mathcal{N}_i}$
		\State \textbf{Initialize}: $\Phi \gets \Big\{\big(s_i=(1,\dots,1), x_i \big)\Big\}$, set of tuples with sequences (1's as dummy variables) and states
		\For{$k = 0,\dots,N$}
		\State $\Phi_\text{branch} \gets \{\}$
		\For{$\phi = (s_i, x_i) \in \Phi$}
		\State $\text{first-flag} \gets 1$
		\For{$l \in \mathcal{L}_i \: \text{s.t.} \: x_i \in \overline{P}_i^{(l)}$}
		\If{$\text{first-flag}=1$}
		\State $\text{first-flag} \gets 0$
		\State $s_i(k) \gets l$
		\If{$k<N$}
		\State $x_i \gets f_{i}^{(l)}\big(x_i, u_i(k), \{x_j^{(i)}(k)\}_{j \in \mathcal{N}_i}\big)$
		\EndIf
		\Else
		\State $(s_i', x_i') \gets \phi$
		\State $s_i'(k) \gets l$
		\If{$k<N$}
		\State $x_i' \gets f_{i}^{(l)}\big(x_i', u_i(k), \{x_j^{(i)}(k)\}_{j \in \mathcal{N}_i}\big)$
		\EndIf
		\State $\Phi_\text{branch} \gets \Phi_\text{branch} \cup \big\{(s_i', x_i')\big\}$
		\EndIf
		\EndFor
		\EndFor
		\State $\Phi \gets \Phi \cup \Phi_\text{branch}$
		\EndFor
		\State \textbf{Outputs}: $\{s_i\}_{\phi \in \Phi}$
	\end{algorithmic}
\end{algorithm}

Let $s = (s_1, \dots, s_M) \in \mathcal{S}$, $\mathcal{S} = \mathcal{S}_1 \times \dots \times \mathcal{S}_M$, define a global switching sequence specifying local switching sequences for each subsystem.
This global $s$ defines a new optimal control problem
\begin{subequations}
	\begin{align}
		\begin{split}
			\mathcal{P}_s(x): \:  &\min_{\textbf{y}} \: \sum_{i \in \mathcal{M}} F_i\big( \textbf{x}_i, \textbf{u}_i, \{\textbf{x}_j^{(i)}\}_{j \in \mathcal{N}_i} \big)
		\end{split}\\
		\text{s.t.} &\quad \forall  i \in \mathcal{M}: \nonumber\\
		&\eqref{eq:couple_cnstr_dist}\--\eqref{eq:consistency_constraint}, \eqref{eq:dynams_seq} \\
		&x_i(k) \in P_i^{\big(s_i(k)\big)}, \: \forall k \in \mathcal{K}', \label{eq:seq_constr}
	\end{align}
\label{eq:MPC_opt_prob_seq}\end{subequations}
where the time-varying dynamics \eqref{eq:dynams_seq} replace the PWA dynamics and the PWA regions selected by $s$ are enforced with the additional constraints \eqref{eq:seq_constr}.
The feasible set for $\mathcal{P}_s(x)$ is
\begin{equation}
	\mathcal{Y}_s(x) = \big\{\textbf{y} = (\Tilde{\textbf{x}}^\top, \textbf{u}^\top)^\top \big| \text{\eqref{eq:couple_cnstr_dist} - \eqref{eq:consistency_constraint}}, \eqref{eq:dynams_seq}, \eqref{eq:seq_constr}, \forall i \in \mathcal{M}\big\},
\end{equation}
and it is convex as the constraints \text{\eqref{eq:couple_cnstr_dist}$\--$\eqref{eq:consistency_constraint}} are convex, the dynamics \eqref{eq:dynams_seq} are linear, and the PWA regions in \eqref{eq:seq_constr} are convex.
Note that $\mathcal{Y}_s$ is empty for many choices of $s$.
We now show that $\mathcal{Y}$ is the union of the sets $\mathcal{Y}_s$.
\begin{lemma}\label{lemma:1}
	Suppose $x$ and $s$ are given such that $\mathcal{Y}_s(x) \neq \emptyset$. 
	Then for all $\textbf{y} = (\Tilde{\textbf{x}}^\top, \textbf{u}^\top)^\top \in \mathcal{Y}_s(x)$:
	\begin{equation}
		\begin{aligned}
			&f_i\big(x_i(k), u_i(k), \{x_j^{(i)}(k)\}_{j \in \mathcal{N}_i}\big) \\
			&\quad = f_{i, s}\big(x_i(k), u_i(k), \{x_j^{(i)}(k)\}_{j \in \mathcal{N}_i}\big),
		\end{aligned}
	\end{equation}	
	$\forall k \in \mathcal{K},\forall i \in \mathcal{M}$.
\end{lemma}
\begin{proof}
	Let $x$ and $s$ be given, and let $\textbf{y}$ be an arbitrary element of $\mathcal{Y}_s(x)$. 
	From \eqref{eq:seq_constr}, $x_i(k) \in P_i^{\big(s_i(k)\big)}, \forall k \in \mathcal{K}, \forall i \in \mathcal{M}$, so that, from \eqref{eq:dynamics},
	\begin{equation}
		\begin{aligned}
			&f_i\big(x_i(k), u_i(k), \{x_j^{(i)}(k)\}_{j \in \mathcal{N}_i}\big) \\
			 &= \Big(A_{i}^{(l)} x_i + B_i^{(l)} u_i + \sum_{j \in \mathcal{N}_i} A_{ij}^{(l)}\Big)\Big|_{l = s_i(k)} \\
			 & = f_{i, s}\big(x_i(k), u_i(k), \{x_j^{(i)}(k)\}_{j \in \mathcal{N}_i}\big),
		\end{aligned}
	\end{equation}
	for all $k \in \mathcal{K}$ and $i \in \mathcal{M}$.
\end{proof}
\begin{proposition}\label{prop:1}
	The feasible set for problem $\mathcal{P}_\text{d}(x)$ is the union of a finite number of convex sets, that are the feasible sets for the problems $\mathcal{P}_s(x)$, i.e.,
	\begin{equation}
		\label{eq:prop_eq}
		\mathcal{Y}(x) = \bigcup_{s \in \mathcal{S}} \mathcal{Y}_s(x).
	\end{equation}
\end{proposition}
\begin{proof}
	i) Suppose $\textbf{y} \in \mathcal{Y}(x)$.
	Let $s$ be a switching sequence generated by $\textbf{y}$, so that $x_i(k) \in P_i^{\big(s_i(k)\big)}, \forall k \in \mathcal{K}, \forall i \in \mathcal{M}$.
	Hence, $\textbf{y}$ satisfies \eqref{eq:seq_constr} in $\mathcal{P}_s(x)$.
	Furthermore, the dynamics in \eqref{eq:dynamics} reduce to \eqref{eq:dynams_seq}.
	Finally, $\textbf{y}$ satisfies \text{\eqref{eq:couple_cnstr_dist}$\--$\eqref{eq:consistency_constraint}} by definition.
	Thus, $\textbf{y} \in \mathcal{Y}_s(x)$ and $\mathcal{Y}(x) \subseteq \mathcal{Y}_s(x) \subseteq \bigcup_{s \in \mathcal{S}} \mathcal{Y}_s(x)$.
	ii) Now suppose $\textbf{y} \in \bigcup_{s \in \mathcal{S}} \mathcal{Y}_s(x)$.
	Then there exists $s \in \mathcal{S}^M$ such that $\textbf{y} \in \mathcal{Y}_s(x)$.
	By Lemma \ref{lemma:1}, for all $\textbf{y} \in \mathcal{Y}_s(x)$, \eqref{eq:dynams_seq} is equivalent to the dynamics in \eqref{eq:dynamics}, and $\textbf{y}$ satisfies \eqref{eq:dynam_cnstrs_dist}.
	Additionally, $\textbf{y}$ satisfies \text{\eqref{eq:couple_cnstr_dist}$\--$\eqref{eq:consistency_constraint}} by definition.
	Thus,  $\textbf{y} \in \mathcal{Y}(x)$ so that $\bigcup_{s \in \mathcal{S}} \mathcal{Y}_s(x) \subseteq \mathcal{Y}(x)$.
	Equation \eqref{eq:prop_eq} follows.
\end{proof}

Proposition \ref{prop:1} reveals the structure of $\mathcal{P}_\text{d}$. 
It is piecewise convex, with each convex region associated with a global switching sequence $s$ and being defined over the set $\mathcal{Y}_s$.
It follows that, in theory, $\mathcal{P}_\text{d}$ can be solved by solving $\mathcal{P}_s$ for each $s \in \mathcal{S}$, and taking the solution that gives the lowest cost.
However, in practice, this is not feasible as $|S|$ grows exponentially with $M$ and $N$.

\subsection{Switching ADMM Procedure}
Rather than solving every convex piece $\mathcal{P}_s$, we propose to consider one piece at a time.
Applying ADMM to one $\mathcal{P}_s$ will converge to the optimum of that convex problem.
However, during the iterations it can be beneficial to switch to a new $\mathcal{P}_s$ that gives a lower cost.
The proposed idea is then to allow subsystems, during the ADMM iterations, to locally change switching sequences $s_i$, changing the global $s$ and $\mathcal{P}_s$, thereby moving from one convex piece to another convex piece of $\mathcal{P}_\text{d}$, until a local minimum is found.
In this section we formalize this idea.

First, we make the ADMM procedure for solving $\mathcal{P}_s$ explicit.
Define the local constraint set containing the constraints independent of $s_i$ as
\begin{equation}
	\label{eq:local_feas_set}
	\begin{aligned}
		\mathcal{Y}_i(x_i) = &\big\{\textbf{y}_i \big| \eqref{eq:couple_cnstr_dist}, \eqref{eq:dist_local_cnstr} \big),
	\end{aligned}
\end{equation}
and the local constraint set containing the constraints dependent on $s_i$ as
\begin{equation}
	\begin{aligned}
		\mathcal{Y}_{i}^{(s_i)} = &\big\{\textbf{y}_i \big|\eqref{eq:dynams_seq}, \eqref{eq:seq_constr}\big\}.
	\end{aligned}
\end{equation}
Introduce additional copies of each state trajectory $\textbf{z} = (\textbf{z}_1^\top, \textbf{z}_2^\top, ...., \textbf{z}_M^\top)^\top$, with the relevant components of $\textbf{z}$ for subsystem $i$ denoted $\Tilde{\textbf{z}}_i = (\textbf{z}_i^\top, \text{col}^\top_{j \in \mathcal{N}_i}(\textbf{z}_j))^\top$.
The problem $\mathcal{P}_s(x)$ can then be reformulated as
\begin{equation}
	\label{eq:ADMM_form}
	\begin{aligned}
		\min_{\{\textbf{y}_i\}_{i \in \mathcal{M}}} \sum_{i \in \mathcal{M}}  &F_i\big( \textbf{x}_i, \textbf{u}_i, \{\textbf{x}_j^{(i)}\}_{j \in \mathcal{N}_i} \big) \\
		\text{s.t.} &\quad \forall  i \in \mathcal{M}:\\
		\quad &\textbf{y}_i \in \mathcal{Y}_i(x_i) \cup \mathcal{Y}_{i}^{(s_i)} \\
		&\Tilde{\textbf{x}}_i = \Tilde{\textbf{z}}_i,
	\end{aligned}
\end{equation}
with the final constraint forcing agreement across all coupled states.
Recall that in \eqref{eq:ADMM_form} the switching sequences $s_i$ are not optimization variables and have been prespecified.

ADMM solves \eqref{eq:ADMM_form} with the following distributed iterations \cite{boydDistributedOptimizationStatistical2010}:
\begin{subequations}
	\label{eq:ADMM_iters}
	\begin{align}
		\begin{split}
			\textbf{y}_i^{\tau+1} &= \argmin_{\textbf{y}_i \in \mathcal{Y}_i \cup \mathcal{Y}_{i}^{(s_i)}} \:  F_i\big( \textbf{x}_i, \textbf{u}_i, \{\textbf{x}_j^{(i)}\}_{j \in \mathcal{N}_i} \big) + (\bm{\lambda}_i^\tau)^\top \Tilde{\textbf{x}}_i \\
			 &\quad\quad\quad+ \frac{\rho}{2}\|\Tilde{\textbf{x}}_i - \Tilde{\textbf{z}}_i^\tau\|_2^2 \label{eq:ADMM_local}
		\end{split}\\
		\textbf{z}_i^{\tau+1} &= \frac{1}{|\mathcal{N}_i|+1} (\textbf{x}_i^{\tau+1} + \sum_{j \in \mathcal{N}_i}\textbf{x}_i^{(j), \tau+1}) \label{eq:ADMM_averaging}\\
		\bm{\lambda}_i^{\tau+1} &= \bm{\lambda}_i^\tau + \rho(\Tilde{\textbf{x}}_i^{\tau+1} - \Tilde{\textbf{z}}_i^{\tau+1}), \label{eq:ADMM_multipliers_update}
	\end{align}
\end{subequations}
with $\bm{\lambda}_i$ the Lagrange multiplier for subsystem $i$ and $\rho > 0$ a penalty parameter.
This is a distributed message passing algorithm where, in step \eqref{eq:ADMM_local} each subsystem solves a local optimization problem, in step \eqref{eq:ADMM_averaging} the local solutions are communicated and averaged between neighboring subsystems, and in step \eqref{eq:ADMM_multipliers_update} the subsystems update their multipliers locally.
Define the residual $r(\textbf{y})$ as the summed disagreement on the values of shared states
\begin{equation}
	r(\textbf{y}) = \sum_{i \in \mathcal{M}} \sum_{j \in \mathcal{N}_i} \|\textbf{x}_j^{(i)} - \textbf{x}_j\|_2.
\end{equation} 
As $\tau \to \infty$ the local solutions converge to the minimizers of $\mathcal{P}_s$, i.e., $\textbf{y}_i^\tau \to \textbf{y}_i^\star, \forall i \in \mathcal{M}$, and the shared states converge to agreement \cite{boydDistributedOptimizationStatistical2010}
\begin{equation}
	\label{eq:consenus_point}
	\textbf{x}_j^{(i), \tau} - \textbf{x}_j^\tau \to 0, \: \forall j \in \mathcal{N}_i, \: \forall i \in \mathcal{M}, 
\end{equation}
i.e., $r(\textbf{y}^\tau) \to 0$.
Additionally, as the optimizer of $\mathcal{P}_s$ is trivially a feasible point of $\mathcal{Y}_s$, and $\mathcal{Y}_s \subset \mathcal{Y}$, we have that as $\tau \to \infty$
\begin{equation}
	\label{eq:feas_point}
	\textbf{y}^\tau \in \mathcal{Y}.
\end{equation}
In the following we will refer to any point where \eqref{eq:consenus_point} and \eqref{eq:feas_point} hold as a \textit{consensus point}.

Our key insight is that new local switching sequences can be identified when the local solutions to \eqref{eq:ADMM_local} are pushed to the boundary of the feasible set $\mathcal{Y}_{i}^{(s_i)}$.
Indeed, in the case where the current $\mathcal{P}_s$ does not contain a local minimum of $\mathcal{P}_\text{d}$ this is the expected behavior, as the minimum of $\mathcal{P}_s$ is then on the boundary of its feasible set.
When this occurs, subsystems change their local switching sequences $s_i$, which globally changes $s$ and $\mathcal{P}_s$.
In this way, the non-convex problem $\mathcal{P}_\text{d}$ is explored by jumping between the different approximations $\mathcal{P}_s$, with the jumping determined distributively by intermediate local solutions to \eqref{eq:ADMM_local}.
Once local solutions are no longer pushed to boundaries of $\mathcal{Y}_{i}^{(s_i)}$ the system-wide $\mathcal{P}_s$ no longer changes, and the ADMM iterations will converge, as $\mathcal{P}_s$ is convex, to a local minimum of $\mathcal{P}_\text{d}$ and a consensus point.

Let us elaborate with the representative visual example in Fig.  \ref{fig:switching_eg} that depicts two systems, each with two PWA regions $P_i^{(1)}$ and $P_i^{(2)}$ and current states $x_1 \in P_1^{(1)}$ and $x_2 \in P_2^{(1)}$.
To facilitate visualization consider scalar states and inputs $u_i, x_i \in \mathbb{R}$ and a horizon of $N = 1$, such that the control input decision variables for each subsystem are scalar and the switching sequences are of length two. 
There are 16 possible global switching sequences 
\begin{equation}
	\mathcal{S} = \big\{s = (s_1, s_2) | s_1, s_2 \in \{(1, 1), (1, 2), (2, 1), (2, 2)\}\big\}.
\end{equation}
However, as $x_1 \in P_1^{(1)}$ and $x_2 \in P_2^{(1)}$, only switching sequences with $s_1(0) = s_2(0) = 1$ correspond to a non-empty $\mathcal{P}_s$.
The top plot of Fig. \ref{fig:switching_eg} shows, over the control input decision space, the four non-empty convex pieces, $\{\mathcal{P}_s\}_{s \in \hat{\mathcal{S}}}, \hat{\mathcal{S}} = \big\{(s_1, s_2) | s_1, s_2 \in \{(1, 1), (1, 2)\}\big\}$, constituting $\mathcal{P}_\text{d}$.
The bottom plot shows the local decision variables $x_1(1)$ and $x_2(1)$ and the PWA partitions.   
An initial guess $\hat{u}_1(0)$ and $\hat{u}_2(0)$ generates $x_1(1) \in P_1^{(1)}$ and $x_2(1) \in P_2^{(1)}$ (blue dot), and thus $s_1 = s_2 = (1, 1)$.
Each subsystem sets the constraint set $\mathcal{Y}_{i}^{(s_i)}$ that defines the problem $\mathcal{P}_s$ (yellow region).
The subsystems commence the ADMM iterations performing one round of the steps in \eqref{eq:ADMM_iters}.
The output of the local minimization yields new solutions (orange dot), and subsystem 2 identifies that this local solution is on the boundary of $\overline{P}_2^{(1)}$ and $\overline{P}_2^{(2)}$ using Algorithm \ref{alg:eval_switching}. 
Subsystem 2 then changes its switching sequence to $s_2 = (1, 2)$, changing the local constraint set $\mathcal{Y}_{2, s_2}$, and changing the global convex problem to a new $\mathcal{P}_s$ (purple region).
The ADMM iterations continue, now with subsystem 2 having changed its local minimization problem \eqref{eq:ADMM_local} by changing $\mathcal{Y}_{2}^{(s_2)}$. 
At the fourth iteration (purple dot), subsystem 1 identifies a new switching sequence $s_1 = (1, 2)$, and changes $\mathcal{Y}_{1}^{(s_1)}$, changing the global problem to the red region. 
Finally, as the local solutions are no longer pushed to the boundaries of the local constraint sets $\mathcal{Y}_{i}^{(s_i)}$, the iterations converge to the minimum of the convex piece (red region), a local minimum for $\mathcal{P}_\text{d}$.
\begin{figure}
	\centering
	\input{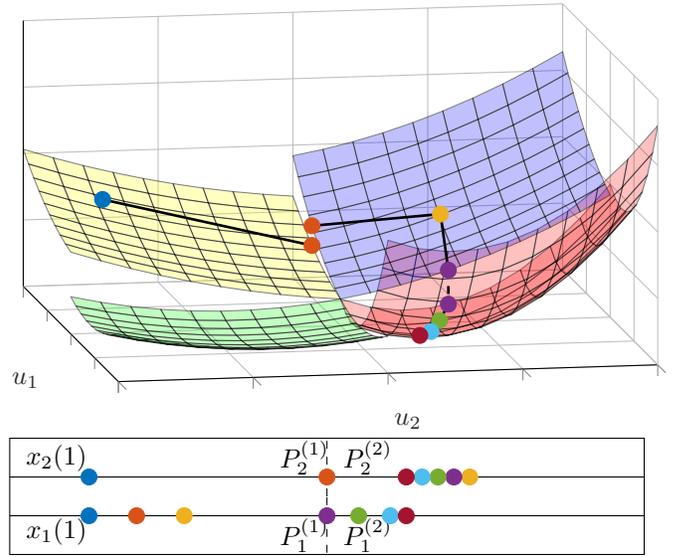}
	\input{media/tikz/eg_2D}
	\caption{Visualization of switching ADMM procedure.}
	\label{fig:switching_eg}
\end{figure}

\begin{color}{black}Generating local switching sequences with Algorithm \ref{alg:eval_switching} requires a set of local copies $\{\textbf{x}_j^{(i)}\}_{j \in \mathcal{N}_i}$.
While these are an output of the local minimization step \eqref{eq:ADMM_local}, for the initial switching sequence, i.e., determining the yellow region from the blue dot in Fig.  \ref{fig:switching_eg}, these values can be obtained with the distributed procedure outlined in Algorithm \ref{alg:dynam_rollout}.
This procedure involves $N$ steps of rolling out local dynamics under an initial guess for the control sequence, and neighbor-to-neighbor communications.
Additional outputs of Algorithm \ref{alg:dynam_rollout} are the predicted local state trajectory $\textbf{x}_i$ and the local cost $F_i$, under the initial control guess.
These become useful in Section \ref{sec:stab_feas}.\end{color}
\begin{algorithm}
	\color{black}
	\caption{$\text{D-rout}$}\label{alg:dynam_rollout}
	\begin{algorithmic}[1]
		\State \textbf{Inputs}: States $\{x_i\}_{i \in \mathcal{M}}$ and control sequences $\{\hat{\textbf{u}}_i\}_{i \in \mathcal{M}}$
		\State \textbf{Initialize}: $x_i(k) \gets 0, x_j^{(i)}(k) \gets 0, F_{i, \text{dr}} \gets 0, \: \forall k \in \mathcal{K}', \: \forall j \in \mathcal{N}_i, \: \forall i \in \mathcal{M}$
		\For{$k = 0,\dots,N-1$}
		
		 $\forall i \in \mathcal{M}$:
		 \State $x_i(k) \gets x_i$
		\State Transmit $x_i$ to each neighbor $j \in \mathcal{N}_i$
		\State For each $j \in \mathcal{N}_i$: receive $x_j$ and set $x_j^{(i)}(k) \gets x_j$
		\If{$k < N$}
			\State $F_{i, \text{dr}} \gets F_{i, \text{dr}} +\ell_i\big(x_i, \hat{u}_i(k), \{x_j^{(i)}(k)\}_{j \in \mathcal{N}_i}\big)$
			\State $x_i \gets f_i\big(x_i, \hat{u}_i(k), \{x_j\}_{j \in \mathcal{N}_i}\big)$
		\Else
			\State $F_{i, \text{dr}} \gets F_{i, \text{dr}} +V_{\text{f}, i}\big(x_i\big)$
		\EndIf
		\EndFor
		\State \textbf{Outputs}: $\big\{\textbf{y}_{i, \text{dr}} = \big(\textbf{x}_i^\top, \text{col}^\top_{j \in \mathcal{N}_i}(\textbf{x}_j^{(i)}), \hat{\textbf{u}}^\top_i\big)^\top, F_{i, \text{dr}}\big\}_{i \in \mathcal{M}}$
	\end{algorithmic}
\end{algorithm}

Algorithm \ref{alg:switching_ADMM} formalizes the switching ADMM procedure.
\begin{color}{black}As an input to Algorithm \ref{alg:switching_ADMM} a set of initial guesses $\{\hat{\textbf{u}}_i\}_{i \in \mathcal{M}}$ is required that generates valid switching sequences in steps \ref{step:d_rout}$-$\ref{step:set_s}.
Such control sequences are defined as follows:
\begin{definition}[Weakly feasible local control sequences]
	The set of local control sequences $\{\textbf{u}_i\}_{i \in \mathcal{M}}$ is \textit{weakly feasible} for $x$ if the corresponding global switching sequence $s = (s_1,\dots,s_M)$, generated in steps \ref{step:d_rout}$-$\ref{step:set_s} of Algorithm \ref{alg:switching_ADMM}, defines a feasible optimal control problem $\mathcal{P}_s(x)$, i.e., $\mathcal{Y}_s(x) \neq \emptyset$.
	Trivially, any globally feasible set of control sequences is also weakly feasible.
\end{definition}
Note that generating weakly feasible control sequences is significantly easier than generating globally feasible control sequences, as weakly feasible control sequences do not need to satisfy the constraints in $\mathcal{P}_\text{d}(x)$, in particular the terminal constraints; rather, the only requirement is that they generate a global switching sequence $s$ for which \textit{there exists} a solution to the corresponding problem $\mathcal{P}_s(x)$.\end{color}
Construction of these guesses is discussed in Sections \ref{sec:sub_dmpc} and \ref{sec:stab_feas}.
\textcolor{black}{Further, we highlight that in Algorithm \ref{alg:switching_ADMM} communication between subsystems occurs only in steps \ref{step:transmit} and \ref{step:receive}, where the communicated quantities are the predicted state trajectories $\textbf{x}_i^{(j), \tau+1}$ for each $j \in \mathcal{N}_i$.}

We introduce a cut-off number of iterations $T_\text{cut}$, after which the switching sequences $s_i$ no longer change.
This protects against the edge case where local solutions pull each other into new switching sequences with equal force, resulting in continuous oscillation between two convex regions of $\mathcal{P}_\text{d}$.
We note that, in our experiments, the continuous switching case is observed only rarely, with the iterations otherwise converging to a problem $\mathcal{P}_s$ that contains a local minimum of $\mathcal{P}_\text{d}$.
\begin{algorithm}
	\caption{Sw-ADMM: procedure for solving $\mathcal{P}_\text{d}$}\label{alg:switching_ADMM}
	\begin{algorithmic}[1]
		\State \textbf{Inputs}: States $\{x_i\}_{i \in \mathcal{M}}$, initial guesses $\{\hat{\textbf{u}}_i\}_{i \in \mathcal{M}}$, and initial local copies $\big\{\{\textbf{x}_j^{(i)}\}_{j \in \mathcal{N}_i}\big\}_{i \in \mathcal{M}}$
		\State \textbf{Initialize:} $\textbf{z}_i \gets 0$ and $\bm{\lambda}_i \gets 0, \: \forall i \in \mathcal{M}$
		\State $S_i \gets \text{Eval-switching}\big(x_i, \hat{\textbf{u}}_i, \{\textbf{x}_j^{(i)}\}_{j \in \mathcal{N}_i}\big), \: \forall i \in \mathcal{M}$
		\State $s_i \gets s_i' \in S_i$ and set local $\mathcal{Y}_{i}^{(s_i)}, \: \forall i \in \mathcal{M}$\label{step:set_s}
		\For{$\tau = 0,\dots, T_\text{ADMM} - 1$}
		
			$\forall i \in \mathcal{M}$:
			
			\State Get $\textbf{y}_i^{\tau+1} = (\Tilde{\textbf{x}}^{\tau+1}_i, \textbf{u}^{\tau+1}_i)^\top$ from local minimization \eqref{eq:ADMM_local}, where $\Tilde{\textbf{x}}^{\tau+1}_i = \big((\textbf{x}_i^{\tau+1})^\top, \text{col}^\top_{j \in \mathcal{N}_i}(\textbf{x}_j^{(i), \tau+1})\big)^\top$
			\State \begin{color}{black}For each $j \in \mathcal{N}_i$: transmit $\textbf{x}_j^{(i), \tau+1}$ to neighbor $j$\end{color}\label{step:transmit}
			\State \begin{color}{black}For each $j \in \mathcal{N}_i$: receive $\textbf{x}_i^{(j), \tau+1}$ from neighbor $j$\end{color}\label{step:receive}
			\State Get $\textbf{z}_i^{\tau+1}$ as \eqref{eq:ADMM_averaging}
			\State Get $\bm{\lambda}_i^{\tau+1}$ as \eqref{eq:ADMM_multipliers_update}
			\If{$\tau < T_\text{cut}$}
				\State $S_i \gets \text{Eval-switching}\big(x_i, \textbf{u}^{\tau+1}_i, \{\textbf{x}_j^{(i), \tau+1}\}_{j \in \mathcal{N}_i}\big)$
				\If{$S_i \setminus \{s_i\} \neq \emptyset$}
					\State $s_i \gets s_i' \in S_i \setminus \{s_i\}$ and set local $\mathcal{Y}_{i}^{(s_i)}$
				\EndIf
			\EndIf
		\EndFor
		\State \textbf{Outputs}: $\{\textbf{y}_i^{T_\text{ADMM}}\}_{i \in \mathcal{M}}$
	\end{algorithmic}
\end{algorithm}

We now give the convergence result for Algorithm \ref{alg:switching_ADMM}.
Following the introduction of $T_\text{cut}$ an additional assumption is required to prove convergence of the algorithm.
\begin{assumption}
	\label{as:switching_assumption}
	\begin{color}{black}For all $x \in \mathcal{X}_0$ there does not exist a global switching sequence $s = (s_1, \dots, s_M)$ such that
	\begin{equation}
		 \big(\mathcal{Y}_s(x) = \emptyset\big) \land \big(\mathcal{Y}_i(x_i) \cup \mathcal{Y}_{i}^{(s_i)} \neq \emptyset, \: \forall i \in \mathcal{M}\big).
	\end{equation}\end{color}
\end{assumption}
Assumption \ref{as:switching_assumption} protects against the edge case where, when $\tau = T_\text{cut}$, the local systems are in a configuration of local switching sequences that represent a $\mathcal{P}_s$ with an empty feasible set.
\begin{color}{black}Additionally, as the convergence of ADMM is an asymptotic result, we make the following assumption:
\begin{assumption}\label{ass:admm}
		$T_\text{ADMM} \to \infty$.
\end{assumption}\end{color}
\begin{proposition} \label{prop:switch_converge}
	For all $x \in \mathcal{X}_0$ and weakly feasible $\{\hat{\textbf{u}}_i\}_{i \in \mathcal{M}}$, under Assumptions \ref{as:switching_assumption} and \ref{ass:admm}, the outputs of Algorithm \ref{alg:switching_ADMM}, $\{\textbf{y}_i^{T_\textup{ADMM}}\}_{i \in \mathcal{M}}$, converge to a consensus point, i.e., $r(\textbf{y}^{T_\textup{ADMM}}) = 0$ and $\textbf{y}^{T_\textup{ADMM}} \in \mathcal{Y}$.
\end{proposition}
\begin{proof}
	For iterates $\tau \geq T_\text{cut}$ the local sets $\mathcal{Y}_{i}^{(s_i)}$ are fixed and, by Assumption \ref{as:switching_assumption}, the corresponding $\mathcal{P}_s$ has a non-empty feasible set. Therefore, as $\mathcal{P}_s$ is convex, the standard ADMM iterations \eqref{eq:ADMM_iters} converge to a consensus point \cite{boydDistributedOptimizationStatistical2010}.
\end{proof}
\begin{remark}
	Assumption \ref{as:switching_assumption} is an assumption on the strength of the coupling between subsystems.
	Providing conditions to easily verify Assumption \ref{as:switching_assumption} is beyond the scope of the current paper.
	An alternative approach, avoiding Assumption \ref{as:switching_assumption}, would be to permit subsystems to switch their local switching sequence only after the new global convex region has been checked for feasibility.
	This is left to future work.
\end{remark}
\begin{color}{black}
\begin{remark}
	While the convergence of ADMM, and hence Algorithm \ref{alg:switching_ADMM}, is asymptotic, in practice ADMM often converges to an acceptable accuracy within a few iterations \cite{boydDistributedOptimizationStatistical2010}, motivating the use of a finite $T_\text{ADMM}$.
	Furthermore, if errors induced by finite termination were to become an issue, existing work that ensures feasibility and stability under finitely terminated distributed optimization \cite{kohlerDistributedModelPredictive2019, rostamiADMMBasedAlgorithmStabilizing2023} can be added to our approach.
\end{remark}\end{color}

\subsection{Distributed MPC Algorithm} \label{sec:sub_dmpc}
Leveraging the switching ADMM procedure we now present the distributed MPC algorithm in Algorithm \ref{alg:dmpc}.
In step \ref{step:guess} each subsystem generates a weakly feasible guess $\hat{\textbf{u}}_i$.
In Section \ref{sec:stab_feas} it is shown how subsystems can locally generate these guesses using a terminal control law.
Here, however, we highlight that these initial guesses can often be generated locally using heuristics specific to the application.
Additionally, a common heuristic is to use the shifted solution from the previous time step, i.e., $\hat{\textbf{u}}_i = \big(\bar{u}_i^\top(1),\dots,\bar{u}_i^\top(N-2), \bar{u}_i^\top(N-2)\big)^\top$, where $\bar{\textbf{u}}_i = \big(\bar{u}_i^\top(0),\dots,\bar{u}_i^\top(N-1)\big)^\top$ is the solution of the previous time step.
Finally, state constraints can be made soft constraints through the introduction of slack variables, penalized in the objective function, to ensure any initial feasible guess is globally, and hence also weakly, feasible.

As Algorithm \ref{alg:switching_ADMM} generates, in general, a local minimum of $\mathcal{P}_\text{d}$, the quality of the solution could be improved by running it in parallel with different initial guesses, introducing an effective multi-start approach.
This is of course limited by the communication and computation limits of the subsystems.
\begin{algorithm}
	\caption{$\text{Distributed MPC, executed each time step}$}\label{alg:dmpc}
	\begin{algorithmic}[1]
		\State Measure local states $x_i, \: \forall i \in \mathcal{M}$
		\State Generate initial weakly feasible guesses $\hat{\textbf{u}}_i, \: \forall i \in \mathcal{M}$ \label{step:guess}
		\begin{color}{black}
			\State $\begin{aligned}\Big\{\textbf{y}_{i, \text{dr}}, F_{i, \text{dr}}\Big\}_{i \in \mathcal{M}} \gets \text{D-rout}\big(\{x_i\}_{i \in \mathcal{M}}, \{\hat{\textbf{u}}_i\}_{i \in \mathcal{M}}\big)\end{aligned}$\label{step:d_rout}
		\end{color}
		\State Solve $\mathcal{P}_\text{d}(x)$ with Algorithm \ref{alg:switching_ADMM}: $\big\{\textbf{y}_i = (\Tilde{\textbf{x}}_i^\top, \textbf{u}_i^\top)^\top\big\}_{i \in \mathcal{M}} \gets \text{Sw-ADMM}\big(\{x_i\}_{i \in \mathcal{M}}, \{\hat{\textbf{u}}_i\}_{i \in \mathcal{M}}, \big\{\{\textbf{x}_j^{(i)}\}_{j \in \mathcal{N}_i}\big\}_{i \in \mathcal{M}}\big)$
		\State Apply local inputs $u_i(0), \: \forall i \in \mathcal{M}$
	\end{algorithmic}
\end{algorithm}

\section{Stability and Recursive Feasibility}\label{sec:stab_feas}
In this section we give a stable and recursively feasible extension to Algorithm \ref{alg:dmpc} under additional assumptions on the system.
Given the local properties of Algorithm \ref{alg:switching_ADMM}, we leverage a dual-mode approach to stability \cite{maDistributedMPCBased2023, scokaertSuboptimalModelPredictive1999}, where terminal constraints are added to the MPC scheme and subsystems use stable switching linear controllers when they enter a terminal set.
This introduces a complex trade-off between the size of the terminal set, the length of the prediction horizon, and closed-loop performance.
For a larger terminal set the systems use linear control laws for a larger portion of the state-space, using the optimization-based MPC control law less and foreseeably sacrificing performance.
With a smaller terminal set, on the other hand, the MPC controller is used more, but requires a longer horizon to satisfy the terminal constraint.
We note that, as Algorithm \ref{alg:switching_ADMM} involves only convex programs, long horizons present less computational demand than for approaches that rely on mixed-integer programming.

\subsection{Assumptions}
Define $\hat{\mathcal{L}}_i = \{l \in \mathcal{L}_i | 0 \in \overline{P}_i^{(l)}\}$ the set of subsystems $i$'s PWA regions containing or touching the origin.
We introduce the following extra assumptions on the systems considered.
\begin{assumption}[\cite{maDistributedMPCBased2023}] \label{as:system}
	\begin{color}{black}For all $i \in \mathcal{M}$ the matrix pairs $\Big(A_i^{(l)}, B_i^{(l)}\Big), \: \forall l \in \mathcal{L}_i$, are controllable; so for all $l \in {\mathcal{L}}_i$ there exists a gain $K_i^{(l)}$ such that $A_i^{(l)} + B_i^{(l)}K_i^{(l)}$ is a Schur matrix.\end{color} 
	The origin is an equilibrium point with $u_i = 0$ and $x_j = 0, \: \forall j \in \mathcal{N}_i$, i.e., $c_i^{(l)} = 0, \: \forall l \in \hat{\mathcal{L}}_i$.
\end{assumption}
Assumption \ref{as:system} says that, for each subsystem, the origin is an equilibrium of the dynamics in all PWA regions that touch it, and that there exists a stabilizing linear feedback law within each of these regions.
\begin{assumption} \label{as:term}
	For all $i \in \mathcal{M}$ there exists a non-empty terminal set $\mathcal{X}_{\text{T}, i} \subseteq \bigcup_{l \in \hat{\mathcal{L}}_i} P_i^{(l)}$, containing the origin, such that $\forall x_i \in \mathcal{X}_{\text{T}, i}$ we have $f_i\big(x_i, u_i, \{x_j\}_{j \in \mathcal{N}_i}\big) \in \mathcal{X}_{\text{T}, i}, \: \forall x_j \in \mathcal{X}_{\text{T}, j}, \: \forall j \in \mathcal{N}_i$, with $u_i = K_i^{(l)} x_i$ and $K_i^{(l)} x_i \in \mathcal{U}_i$. 
\end{assumption}
Define the global terminal set to be $\mathcal{X}_\text{T} = \{x | x_i \in \mathcal{X}_{\text{T}, i}, \forall i \in \mathcal{M}\}$.
Assumption \ref{as:term} guarantees that there exist local terminal sets that are positively invariant under the switching terminal controllers and that are robust to the effects of coupling when neighboring subsystems are also within their terminal sets.
This is a less restrictive assumption than that in \cite{maDistributedMPCBased2023}, where the terminal sets are assumed to be robust to the effects of coupling for the entire state space of neighboring systems $\mathcal{X}_j$.
In contrast to \cite{maDistributedMPCBased2023}, this less restrictive assumption is possible due to the proposed approach explicitly giving agreement on the values of shared states over the prediction horizon, rendering the approach applicable to systems with stronger coupling.
Note that \cite{maDistributedMPCBased2023} proposes a method to compute the terminal sets required for the controller in \cite{maDistributedMPCBased2023}.
These sets trivially satisfy Assumption \ref{as:term} as $\mathcal{X}_{\text{T}, j} \subset \mathcal{X}_j$.
\begin{color}{black}In Appendix \ref{ap:sets} we present a modification to the procedure given in \cite{maDistributedMPCBased2023} to calculate less restrictive sets satisfying Assumption \ref{as:term}.
Additionally, we provide a formal proof, missing in \cite{maDistributedMPCBased2023}, for the forward invariance property of the resulting sets.\end{color}
Clearly, the sets required by \cite{maDistributedMPCBased2023} do not exist for systems with unbounded state spaces, while the sets in Assumption \ref{as:term} can still exist.
\begin{assumption} \label{as:lypanov}
	The stage and terminal costs are of the form $\ell_i(x_i, u_i, \{x_j\}_{j \in \mathcal{N}_i}) = x_i^\top Q_i x_i + u_i^\top R u_i + \sum_{j \in \mathcal{N}_i} x_j^\top Q_{ij} x_j$ and $V_{\text{f}, i}(x) = x_i^\top \Phi_i x_i$ for positive definite $Q_i, R_i$, and $\Phi_i$. 
	The terminal costs $V_{\text{f}, i}$ and controllers $K_i^{(l)}$ are such that for all $x \in \mathcal{X}_\text{T}$,
	\begin{equation} \label{eq:lyapnuov}
		\begin{aligned}
			\sum_{i \in \mathcal{M}}\Big(&V_{\text{f}, i} \big( f_i(x_i, {u}_i, \{x_j\}_{j \in \mathcal{N}_i}) \big) - V_{\text{f}, i}(x_i) \\
			&\quad\quad + \ell_i(x_i, {u}_i, \{x_j\}_{j \in \mathcal{N}_i}) \Big) \leq 0,
		\end{aligned}
	\end{equation}
	where ${u}_i = K_i^{(l)} x_i, x_i \in P_i^{(l)}$.
\end{assumption}
Assumption \ref{as:lypanov} says that the terminal cost matrices $\Phi_i$ result in $\sum_{i \in \mathcal{M}}V_{\text{f}, i}$ being a global common quadratic Lyapunov function for all possible combinations of local PWA regions within the terminal set.
The computation of such terminal costs has been covered for single PWA systems in \cite{lazarStabilizingModelPredictive2006} and for distributed linear systems in \cite{conteDistributedSynthesisStability2016}.
In Appendix \ref{ap:LMI} we outline an LMI approach to constructing terminal costs $V_{\text{f}, i}$ that satisfy Assumption \ref{as:lypanov}.
\begin{color}{black}While Assumption \ref{as:lypanov} may be restrictive if the origin is on the boundary of many PWA regions, note that we could restrict the local terminal sets to a single PWA region, $\mathcal{X}_{\text{T}, i} \subseteq P_i^{(l_i)}, l_i \in \hat{\mathcal{L}}_i$, such that finding a terminal cost satisfying Assumption \ref{as:lypanov} reduces to finding a Lyapunov function for the linear system defined by the PWA regions $\{P_i^{(l_i)}\}_{i \in \mathcal{M}}$, rather than a common Lyapnuov function over arbitrary switching.
Again we note that, as Algorithm \ref{alg:switching_ADMM} involves only convex programs, long horizons are computationally tractable and smaller terminal sets are therefore viable.\end{color}
\begin{color}{black}\begin{assumption}\label{as:initial_guess}
	A weakly feasible initial set of control sequences is assumed to be known for the first time step.
\end{assumption}
We highlight that in the distributed MPC literature it is common to assume access to a \textit{globally feasible} set of initial control sequence \cite{richardsRobustDistributedModel2007, asadiScalableDistributedModel2018}, which is a much stronger assumption than Assumption \ref{as:initial_guess}.\end{color}

\subsection{Stabilizing Distributed MPC Algorithm}
With a slight abuse of notation, we redefine the optimal control problems $\mathcal{P}_\text{d}(x)$ and $\mathcal{P}_s(x)$ by modifying the feasible sets $\mathcal{Y}$ in \eqref{eq:feasible_P} and $\mathcal{Y}_i$ in \eqref{eq:local_feas_set} to include terminal constraints
\begin{equation}
	\begin{aligned}
		\mathcal{Y}(x) &= \big\{\textbf{y} \big| \text{\eqref{eq:dynam_cnstrs_dist} - \eqref{eq:consistency_constraint}}, x_i(N) \in \mathcal{X}_{\text{T}, i}, \forall i \in \mathcal{M}\big\} \\
		\mathcal{Y}_i(x_i) &= \big\{\textbf{y}_i \big| \eqref{eq:couple_cnstr_dist}, \eqref{eq:dist_local_cnstr}, x_i(N) \in \mathcal{X}_{\text{T}, i} \big\},
	\end{aligned}
\end{equation}
naturally redefining $\mathcal{X}_0 = \{x \: | \: \mathcal{Y}(x) \neq \emptyset\}$.
We then introduce the stabilizing distributed MPC algorithm in Algorithm \ref{alg:stable_dmpc}.
\begin{algorithm}[H]
	\caption{$\text{Stable distributed MPC, executed each time step}$}\label{alg:stable_dmpc}
	\begin{algorithmic}[1]
		\State Measure local states $x_i, \: \forall i \in \mathcal{M}$
		\If{$x_i \in \mathcal{X}_{\text{T}, i}, \forall i \in \mathcal{M}$}
			\State Apply local inputs $u_i = K_i^{(l)}x_i, x_i \in P_i^{(l)}, \forall i \in \mathcal{M}$
		\Else 
			\If{first time step}
				\State $\hat{\textbf{u}}_i \gets$ known feasible guess, $\forall i \in \mathcal{M}$
			\Else
				\State Construct initial guess from terminal controller and previous solution $\hat{\textbf{u}}_i \gets \Big(\bar{u}_i^\top(1),\dots,\bar{u}_i^\top(N-1), \big(K_i^{(l)}\bar{x}_i(N)\big)^\top\Big)^\top, \bar{x}_i(N) \in P_i^{(l)}, \: \forall i \in \mathcal{M}$ \label{step:feas_guess}
			\EndIf
			\State \begin{color}{black}
				\State $\begin{aligned}\Big\{\textbf{y}_{i, \text{dr}}, F_{i, \text{dr}}\Big\}_{i \in \mathcal{M}} \gets \text{D-rout}\big(\{x_i\}_{i \in \mathcal{M}}, \{\hat{\textbf{u}}_i\}_{i \in \mathcal{M}}\big)\end{aligned}$\label{step:d_rout2}
			\end{color}
			\State Solve $\mathcal{P}_\text{d}(x)$ with Algorithm \ref{alg:switching_ADMM}: $\{\textbf{y}_i\}_{i \in \mathcal{M}} \gets \text{Sw-ADMM}\big(\{x_i\}_{i \in \mathcal{M}}, \{\hat{\textbf{u}}_i\}_{i \in \mathcal{M}}, \big\{\{\textbf{x}_j^{(i)}\}_{j \in \mathcal{N}_i}\big\}_{i \in \mathcal{M}}\big)$ \label{step:solve}
				\If{\begin{color}{black}not $\big(F_i\big( \textbf{x}_i, \textbf{u}_i, \{\textbf{x}_j^{(i)}\}_{j \in \mathcal{N}_i} \big) \leq F_{i, \text{dr}}, \: \forall i \in \mathcal{M}\big)$\end{color}} \label{step:if_not_improve}
				\State \begin{color}{black}$\textbf{y}_i \gets \textbf{y}_{i, \text{dr}}, \: \forall i \in \mathcal{M}$\end{color}
				\EndIf \label{step:end_if}
			\State $\bar{\textbf{x}}_i \gets \textbf{x}_i$ and $\bar{\textbf{u}}_i \gets \textbf{u}_i, \: \forall i \in \mathcal{M}$
			\State Apply local inputs $u_i(0), \: \forall i \in \mathcal{M}$
		\EndIf
	\end{algorithmic}
\end{algorithm}
When all subsystems are in the terminal regions the switching linear feedback control laws are used.
Otherwise, for Algorithm \ref{alg:switching_ADMM}, subsystems have local access to a feasible initial guess through shifting the previous control solution and adding the terminal controller as the final element of the initial feasible guess.
\begin{color}{black}Steps \ref{step:if_not_improve}$-$\ref{step:end_if} provide a verification that the solution generated by the switching ADMM procedure has improved the value with respect to the initial guesses.
The verification is local, as both $F_i\big( \textbf{x}_i, \textbf{u}_i, \{\textbf{x}_j^{(i)}\}_{j \in \mathcal{N}_i} \big)$ and $F_{i, \text{dr}}$ are known by subsystem $i$.
This protects against the edge case where the local minimum found by the procedure is worse than the initialization.\end{color}

We now prove the recursive feasibility and the closed-loop stability of Algorithm \ref{alg:stable_dmpc}.
\begin{theorem} \label{prop:feas}
	\begin{color}{black}Assume Assumptions \ref{as:switching_assumption}$-$\ref{as:initial_guess} hold.\end{color}
	For an initial state $x \in \mathcal{X}_0 \setminus \mathcal{X}_\text{T}$, at time step $t=0$, a weakly feasible set of initial control guesses $\{\hat{\textbf{u}}_{i}\}_{i \in \mathcal{M}}$, and generating control inputs via Algorithm \ref{alg:stable_dmpc}, $\mathcal{P}_\text{d}(x)$ is feasible for all time steps $t \geq 0$ in which $x \notin \mathcal{X}_\text{T}$.
\end{theorem}
\begin{proof}
	The proof works by induction.
	Assume that, for state $x \notin \mathcal{X}_\text{T}$, at time step $t$, $\mathcal{P}_\text{d}(x)$ is feasible, and that subsystems have weakly feasible initial control guesses $\hat{\textbf{u}}_i, \forall i \in \mathcal{M}$.
	By Proposition \ref{prop:switch_converge}, with weakly feasible initial guesses, the outputs $\textbf{y}_i = (\Tilde{\textbf{x}}_i^\top, \textbf{u}_i^\top)^\top$ of Algorithm \ref{alg:switching_ADMM} are feasible for $\mathcal{P}_\text{d}(x)$, and therefore $x_i(N) \in \mathcal{X}_{\text{T}, i}, \forall i \in \mathcal{M}$.
	The subsystems apply local inputs $u_i(0)$, propagating the state dynamics as $x_{i}^+ = f_i\big(x_{i}, u_i(0), \{x_{j}\}_{j \in \mathcal{N}_i}\big)$ and, as by Proposition \ref{prop:switch_converge} $\textbf{x}_j^{(i)} = \textbf{x}_j, \forall j \in \mathcal{N}_i$, the new states coincide with the first predicted state, i.e., $x_{i}^+ = x_i(1)$.
	\begin{color}{black}At the next time step $t+1$, in state $x^+$, the initial guesses are  $\hat{\textbf{u}}_i = \Big({u}_i^\top(1),\dots,{u}_i^\top(N-1), \big(K_i^{(l)}{x}_i(N)\big)^\top\Big)^\top, {x}_i(N) \in P_i^{(l)} \cap \mathcal{X}_{\text{T}, i}, \: i \in \mathcal{M}$, and are globally feasible, and therefore also weakly feasible, as $\mathcal{X}_{\text{T}, i}$ is forward invariant and robust to coupling under the terminal controllers when $x_j \in \mathcal{X}_{\text{T}, j}, \forall j \in \mathcal{N}_i$.\end{color}
	Therefore $\mathcal{P}_\text{d}(x^+)$ is feasible at time step $t+1$.
	
	Now, for time step $t = 0$ and initial state $x \in \mathcal{X}_0 \setminus \mathcal{X}_\text{T}$, $\mathcal{P}_\text{d}(x)$ is feasible as $x \in \mathcal{X}_0$, and the initial guesses $\{\hat{\textbf{u}}_{i}\}_{i \in \mathcal{M}}$ for Algorithm \ref{alg:switching_ADMM} are weakly feasible.
	It follows then by induction that $\mathcal{P}_\text{d}(x)$ is feasible for all time steps $t > 0$ in which $x \notin \mathcal{X}_\text{T}$.
\end{proof}
\begin{remark}
	Theorem \ref{prop:feas} considers only states $x \notin \mathcal{X}_\text{T}$ as $\mathcal{P}_\text{d}$ is not relevant in Algorithm \ref{alg:stable_dmpc} when $x \in \mathcal{X}_{\text{T}}$.
\end{remark}
\begin{color}{black}
\begin{lemma}\label{lemma:2}
	Assume Assumptions \ref{as:switching_assumption}$-$\ref{as:initial_guess} hold.
	For all $x \in \mathcal{X}_0 \setminus \mathcal{X}_\text{T}$ and time steps $t > 0$, generating control inputs via Algorithm \ref{alg:stable_dmpc} we have
	\begin{equation}
		V(x, \textbf{u}) \leq V(x, \hat{\textbf{u}}),
	\end{equation}
	where $\hat{\textbf{u}} = (\hat{\textbf{u}}_1^\top,\dots,\hat{\textbf{u}}_M^\top)^\top$ is the initial guess for Algorithm \ref{alg:switching_ADMM} (step \ref{step:solve} of Algorithm \ref{alg:stable_dmpc}), and $\textbf{u}$ is the resulting global control input sequence (step \ref{step:end_if} of Algorithm \ref{alg:stable_dmpc}), with $V$ defined in \eqref{eq:cost}.
\end{lemma}
\begin{proof}
For time steps $t > 0$, $\hat{\textbf{u}}$ is globally feasible, i.e., $V(x, \hat{\textbf{u}}) < \infty$.
Observe that $\sum_{i \in \mathcal{M}} F_{i, \text{dr}} = V(x, \hat{\textbf{u}})$.
If there exists $i \in \mathcal{M}$ such that $F_i\big( \textbf{x}_i, \textbf{u}_i, \{\textbf{x}_j^{(i)}\}_{j \in \mathcal{N}_i} \big) > F_{i, \text{dr}}$ in step \ref{step:if_not_improve}, then the input $\textbf{u}$ is set to $\hat{\textbf{u}}$, and $V(x, \textbf{u}) = V(x, \hat{\textbf{u}})$.
If $F_i\big( \textbf{x}_i, \textbf{u}_i, \{\textbf{x}_j^{(i)}\}_{j \in \mathcal{N}_i} \big) \leq F_{i, \text{dr}}, \: \forall i \in \mathcal{M}$ in step \ref{step:if_not_improve}, then $\sum_{i \in \mathcal{M}} F_i\big( \textbf{x}_i, \textbf{u}_i, \{\textbf{x}_j^{(i)}\}_{j \in \mathcal{N}_i} \big) \leq V(x, \hat{\textbf{u}})$.
By Proposition \ref{prop:switch_converge} $\textbf{x}_j^{(i)} = \textbf{x}_j, \forall j \in \mathcal{N}_i$, and therefore the summation is equal to $V(x, \textbf{u})$, i.e., $V(x, \textbf{u}) \leq V(x, \hat{\textbf{u}})$.
\end{proof}\end{color}
\begin{theorem}
	\begin{color}{black}Assume Assumptions \ref{as:switching_assumption}$-$\ref{as:initial_guess} hold.\end{color}
	The control law generated by Algorithm \ref{alg:stable_dmpc} is exponentially stabilizing with a region of attraction $\mathcal{X}_0$.
\end{theorem}
\begin{proof}
	\begin{color}{black}Take $\textbf{u} = \{\textbf{u}_i\}_{i \in \mathcal{M}}$ as the global input generated in Algorithm \ref{alg:stable_dmpc} at time step $t$ with global state $x \in \mathcal{X}_0 \setminus \mathcal{X}_{\text{T}}$.
	Furthermore, take $\{x_i(N)\}_{i \in \mathcal{M}}$ as the set of predicted states at time step $t+N$ after applying the control inputs $\textbf{u}$.
	By Proposition \ref{prop:switch_converge} we have $x_i(N) \in \mathcal{X}_\text{T, i}, \forall i \in \mathcal{M}$.
	Take $\hat{\textbf{u}}_i = \Big({u}_i^\top(1),\dots,{u}_i^\top(N-1), \big(K_i^{(l)}{x}_i(N)\big)^\top\Big)^\top, {x}_i(N) \in P_i^{(l)}$ as the globally feasible initial guess at the next time step $t+1$, constructed from $\textbf{u}_i$, $x^+ \in \mathcal{X}_0 \setminus \mathcal{X}_{\text{T}}$ the global state at time step $t+1$ after applying local inputs $u_i(0)$, and $\textbf{u}^+$ as the global input generated in Algorithm \ref{alg:stable_dmpc} at time step $t+1$.
	Consider the values $V({x}^+, \hat{\textbf{u}}) < \infty$ and $V({x}, \textbf{u}) < \infty$.
	As the control sequence $ \hat{\textbf{u}}$ is a shifted version of $\textbf{u}$, the resulting state trajectories will overlap except for the first and last time steps.
	We thus have the relation \cite{mayneConstrainedModelPredictive2000}
	\begin{equation}
		\begin{aligned}
			V&({x}^+, \hat{\textbf{u}}) - V({x}, \textbf{u}) = -\sum_{i \in \mathcal{M}}  \ell_i \big(x_i, u_i(0), \{x_j\}_{j \in \mathcal{N}_i}\big) \\
			&+\sum_{i \in \mathcal{M}} \bigg(V_{\text{f}, i} \Big( f_i\big(x_i(N), K_i^{(l)}{x}_i(N), \{x_j(N)\}_{j \in \mathcal{N}_i}\big) \Big) \\
			& - V_{\text{f}, i}\big(x_i(N)\big) + \ell_i\big(x_i(N), K_i^{(l)}{x}_i(N), \{x_j(N)\}_{j \in \mathcal{N}_i}\big) \bigg).
		\end{aligned}
	\end{equation}
	By Assumption \ref{as:lypanov} the second summation term is non-positive, thus $	V({x}^+, \hat{\textbf{u}}) - V({x}, \textbf{u}) \leq -\sum_{i \in \mathcal{M}}  \ell_i \big(x_i, u_i(0), \{x_j\}_{j \in \mathcal{N}_i}\big)$.
	Further, the conditions of Lemma \ref{lemma:2} hold at time step $t+1$, such that $V(x^+, \textbf{u}^+) \leq V(x^+, \hat{\textbf{u}})$, giving the expression
	\begin{equation}
		\begin{aligned}
			&V(x^+, \textbf{u}^+) - V(x, \textbf{u}) \leq V(x^+, \hat{\textbf{u}}) - V(x, \textbf{u}) \\
			 &\quad \leq -\sum_{i \in \mathcal{M}} \ell_i \big(x_i, u(0), \{x_j\}_{j \in \mathcal{N}_i}\big),
		\end{aligned}
	\end{equation}
	for all $x \in \mathcal{X}_0 \setminus \mathcal{X}_{\text{T}}$.\end{color}
	Assumption \ref{as:lypanov} on $\ell_i$ guarantees the existence of a class-$\mathcal{K}$ function $\beta(\cdot)$ such that 
	\begin{equation}
		\sum_{i \in \mathcal{M}} \ell_i(x_i, u_i, \{x_j\}_{j \in \mathcal{N}_i}) \geq \beta(\|(x, u)\|),
	\end{equation}
	for all $x \notin \mathcal{X}_T$, and for all $u \in \mathcal{U}_1 \times \dots \times \mathcal{U}_M$.
	As $\mathcal{X}_{\text{T}}$ contains the origin there exists $b > 0$ such that $V(x^+, \textbf{u}^+) - V(x, \textbf{u}) \leq -b$ for all $x \notin \mathcal{X}_{\text{T}}$.
	
	\begin{color}{black}For initial state $x \in \mathcal{X}_0\setminus \mathcal{X}_{\text{T}}$, at time step $t = 0$, we have $V\big(x, \textbf{u}\big) < \infty$ by Assumption \ref{as:initial_guess}.
	Let $x^{+\bar{t}}$ and $\textbf{u}^{+\bar{t}}$ denote the global state and global control input sequence at time step $\bar{t}$, with $\bar{t}$ a finite integer such that $(\bar{t}-1)b > V\big(x, \textbf{u}\big)$.
	If $x^{+\bar{t}} \notin \mathcal{X}_\text{T}$, then $V\big(x^+, \textbf{u}^+\big) - V\big(x, \textbf{u}\big) \leq -b$ holds for all time steps $t = 0,\dots,\bar{t}-1$.
	It follows that $V\big(x^{+\bar{t}}, \textbf{u}^{+\bar{t}}\big) \leq V\big(x, \textbf{u}\big) - (\bar{t}-1)b < 0$, which is a contradiction as the value V is non-negative.
	There therefore exists, for any initial state $x \in \mathcal{X}_0\setminus \mathcal{X}_{\text{T}}$, a finite time $\bar{t}$ at which the global state enters the global terminal set $x \in \mathcal{X}_\text{T}$ \cite{scokaertSuboptimalModelPredictive1999}.\end{color}
	
	As all local systems use the terminal controllers $u_i = K_i^{(l)}x_i, x_i \in P_i^{(l)}$ for $x_i \in \mathcal{X}_{\text{T}, i}$, and the local terminal sets are robustly positively invariant under this terminal control, we have $x \in \mathcal{X}_\text{T}$ for all $t \geq \bar{t}$.
	Finally, by Assumption \ref{as:lypanov}, under the switching linear controllers there exists a global common quadratic Lyapunov function for all possible combinations of local PWA regions when the global state $x \in \mathcal{X}_\text{T}$.
	A common quadratic Lyapunov function implies exponential stability under arbitrary switching \cite{linStabilityStabilizabilitySwitched2009}, hence the global system is exponentially stable under the terminal controllers.
	Therefore as, for any initial state $x \in \mathcal{X}_0$, the global state enters the terminal region $\mathcal{X}_\text{T}$ in finite time and is exponentially stable within $\mathcal{X}_\text{T}$, the origin is exponentially stable under Algorithm \ref{alg:stable_dmpc} with regions of attraction $\mathcal{X}_0$.
\end{proof}

\section{Illustrative Examples}
\label{sec:ex}
In this section we provide two examples demonstrating our approach.
The first considers the network from \cite{maDistributedMPCBased2023}, satisfying Assumptions \ref{as:system}$\--$\ref{as:initial_guess}.
This example demonstrates the stabilizing properties of Algorithm \ref{alg:stable_dmpc} and compares the approach to that of \cite{maDistributedMPCBased2023}.
The second example considers the more realistic control challenge of a platoon of vehicles, proposed as a distributed hybrid MPC benchmark in \cite{mallick2024comparison}, and demonstrates the efficacy of Algorithm \ref{alg:dmpc} as a practical control scheme.
All examples are simulated on an 11th Gen Intel laptop with four i7 cores, 3.00GHz clock speed, and 16Gb of RAM.
Source code for the simulations can be found at \url{https://github.com/SamuelMallick/{distributed-mpc-pwa, stable-dmpc-pwa/tree/paper_2024, hybrid-vehicle-platoon/tree/paper-2024}}.{\tiny}
 
\subsection{Stabilizing Example}
In the following we refer to Algorithm \ref{alg:stable_dmpc} as SwA and the approach based on robust controllable sets presented in \cite{maDistributedMPCBased2023} as RCS.
\begin{color}{black}We consider also a centralized MPC controller that solves $\mathcal{P}$ directly as a MIQP, referred to as CENT.\end{color}
Consider the network of three PWA systems, from \cite{maDistributedMPCBased2023}, with identical dynamics defined over $L = 4$ regions, as shown in Fig.  \ref{fig:ex_1_traj}.
The systems' dynamics parameters, for $i = 1,2,3$, are: 
\begin{equation}
	\begin{aligned}
		A_i^{(1)} &= A_i^{(3)} = \begin{bmatrix}
			0.6324 & 0.2785 \\ 0.0975 & 0.5469
		\end{bmatrix} \\ 
		\quad A_i^{(2)} &= A_i^{(4)} = \begin{bmatrix}
			0.6555 & 0.7060 \\ 0.1712 & 0.0318
		\end{bmatrix} \\
		B_i^{(1)} &= B_i^{(2)} = B_i^{(3)} = B_i^{(4)} = \begin{bmatrix}
			1 \\ 0
		\end{bmatrix},
	\end{aligned}
\end{equation}
with state and action constraints $\mathcal{X}_i = \{x_i \big| |x_i|_\infty \leq 20\}$ and $\mathcal{U}_i = \{u_i \big| |u_i| \leq 3\}$, and stage costs $\ell_i = x_i^\top Q x_i + u_i^\top R u_i$ with
\begin{equation}
	Q_i = \begin{bmatrix}
		2 & 0 \\ 0 & 2
	\end{bmatrix}, \quad  R = 0.2.
\end{equation} 
The coupling in the network is 
\begin{equation}
	A_{12} = A_{21} = A_{23} = A_{31} = 2 \cdot 10^{-3} \cdot \begin{bmatrix}
		1 & 0 \\ 0 & 1
	\end{bmatrix},
\end{equation}
with all other coupling matrices zero.
For a fair comparison\footnote{In \cite{maDistributedMPCBased2023} a small disturbance is added to the dynamics and is accounted for in the computation of the robustly controllable sets. The approach of \cite{maDistributedMPCBased2023} requires no extra mechanism to account for the disturbance as the coupling between subsystems is also considered as a local disturbance. For a fair comparison we recompute all the sets required for the RCS approach with no additional disturbance in the dynamics.} we use the same terminal controllers as \cite{maDistributedMPCBased2023}, i.e., for $i = 1,2,3$:
\begin{equation}
	\begin{aligned}
		K_i^{(1)} &= K_i^{(3)} = \begin{bmatrix}
			-0.0544 & -0.1398
		\end{bmatrix} \\
		K_i^{(2)} &= K_i^{(4)} = \begin{bmatrix}
			-0.1544 & -0.0295
		\end{bmatrix},
	\end{aligned}
\end{equation}
and the same robustly positive invariant terminal sets, i.e., for $i = 1,2,3$:
\begin{equation}\label{eq:term_sets}
	\begin{aligned}
		\mathcal{X}_{\text{T}, i} = \Bigg\{x \Bigg| &\begin{bmatrix}
			P \\ -P
		\end{bmatrix}x \leq \gamma \mathbf{1}, \\ &P = \begin{bmatrix}
			7.8514 & 8.1971 \\ 8.1957 & -7.8503
		\end{bmatrix}, \gamma=47\Bigg\}.
	\end{aligned}
\end{equation}
Finally, we compute local terminal costs $V_{\text{f}, i}(x_i) = x_i^\top \Phi_i x_i$, satisfying Assumption \ref{as:lypanov}, using the procedure in Appendix \ref{ap:LMI}, which gives
\begin{equation}
	\begin{aligned}
		\Phi_1 &= \begin{bmatrix}
			12.67 & 8.87 \\ 8.87 & 8.14
		\end{bmatrix} \cdot 10^{-4} + \Phi \\
		\Phi_2 &= \begin{bmatrix}
			10.58 & 7.90 \\ 7.90 & 8.26
		\end{bmatrix} \cdot 10^{-4} + \Phi \\
		\Phi_3 &= \begin{bmatrix}
			8.53 & 5.43 \\ 5.43 & 7.10
		\end{bmatrix} \cdot 10^{-4} + \Phi \\
		\Phi &= \begin{bmatrix}
			3.938 & 1.262 \\ 1.262 & 4.346
		\end{bmatrix}.
	\end{aligned}
\end{equation}
For the SwA approach the horizon is $N = 5$.
Algorithm \ref{alg:switching_ADMM} is run with $T_\text{ADMM}=50$ and $\rho = 0.5$.
These values are hand chosen to ensure the total residual is less than 0.01. 
For this system the edge case of continuous switching is never observed; $T_\text{cut}$ is consequently set to also be 50.
\begin{color}{black}For the weakly feasible guesses in the first time step, the trivial guesses $\hat{\textbf{u}}_i = (0,\dots,0)^\top$ can be used.
These trivial guesses are weakly feasible but not globally feasible, demonstrating the practical satisfaction of Assumption \ref{as:initial_guess}.\end{color}

Fig.  \ref{fig:ex_1_traj} shows the trajectories under each approach for the initial condition, $x_1(0) = \begin{bmatrix}
	-11 & -18
\end{bmatrix}^\top, x_2(0) = \begin{bmatrix}
	2 & -19
\end{bmatrix}^\top, x_3(0) = \begin{bmatrix}
	15 & 19
\end{bmatrix}^\top$.
Both approaches drive the states to the origin.
\begin{color}{black}Fig. \ref{fig:ex_1_cost} shows, for 100 randomly sampled initial conditions, the relative performance drop with respect to the centralized controller $J_o/J_\text{CENT}, o \in \{\text{SwA}, \text{RCS}\}$ with
\begin{equation}\label{eq:closed_loop_cost}
	J_o = \sum_{t = 0}^{30}\sum_{i = 1}^3 \ell_i\big(x_{i}(t), u_{i}(t), \{x_{j}(t)\}_{j \in \mathcal{N}_i}\big).
\end{equation}
The SwA approach generally performs slightly better, in terms of \eqref{eq:closed_loop_cost}, than the RCS approach, with occasional outlier initial conditions for which the performance improvement is much more significant\footnote{\begin{color}{black} While unlikely in practical problems, in theory Algorithm \ref{alg:switching_ADMM} can converge to an arbitrarily bad local minimum. This is explored with randomized systems in Appendix \ref{sec:ap3}.\end{color}}.
There is one outlier in which the RCS approach performs better, even outperforming the centralized controller, which is indeed possible with a finite prediction horizon.\end{color}

\begin{color}{black}Fig. \ref{fig:ex_1_time} compares the online computation time required over all simulations for each approach using a range of solvers.
The SwA approach allows the use of quadratic programming solvers, while the RCS approach requires MIQP solvers.
The use of quadratic programming solvers allows the SwA approach to be much faster than the RCS approach, despite its iterative nature.\end{color}
\begin{color}{black}The computation times for the centralized controller, solved using Gurobi \cite{gurobi}, are in the order of tens of seconds, and are not included in Fig. \ref{fig:ex_1_time} for clarity.\end{color}
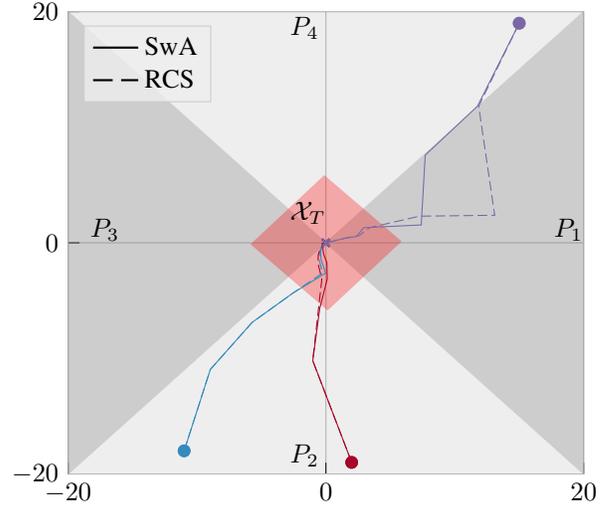
\begin{figure}
	\centering
	\input{media/tikz/ex_1_traj}
	\caption{State trajectories for SwA and RCS with weak coupling.}
	\label{fig:ex_1_traj}
\end{figure}
\begin{figure*}
	\color{black}
	\centering
	\input{media/tikz/ex_1_cost}
	\caption{Performance drop with respect to the centralized controller for 100 different initial conditions.}
	\label{fig:ex_1_cost}
\end{figure*}
\begin{figure}
	\color{black}
	\centering
	\input{media/tikz/ex_1_time}
	\caption{Computation times with different solvers over 100 different initial conditions. The quadratic programming solvers are \textit{osqp} \cite{osqp}, \textit{qpOASES} \cite{Ferreau2014}, and \textit{qrqp} \cite{Andersson2019}. The MIQP solvers are \textit{CPLEX} \cite{cplex2009v12}, \textit{Mosek} \cite{aps2019mosek}, and \textit{Gurobi} \cite{gurobi}.}
	\label{fig:ex_1_time}
\end{figure}

Now consider the same network, with the same terminal controllers, but with significantly higher coupling:
\begin{equation}
	\label{eq:strong_coupling}
	A_{12} = A_{21} = A_{23} = A_{31} = \begin{bmatrix}
		0.16 & 0 \\ 0 & 0.16
	\end{bmatrix}.
\end{equation}
The uncontrolled global system is now unstable.
Under this coupling there do not exist local terminal sets that are robust to the coupling effects over the entire neighboring state spaces, and the RCS approach is hence not applicable.
There do, however, exist terminal sets satisfying Assumption \ref{as:term}.
Using the procedure in Appendix \ref{ap:sets}, initialized with the sets in \eqref{eq:term_sets}, we determine that the terminal sets in \eqref{eq:term_sets} satisfy Assumption \ref{as:term} under the stronger coupling \eqref{eq:strong_coupling}.
Using the procedure in Appendix \ref{ap:LMI} the terminal costs $V_{\text{f}, i}(x_i) = x_i^\top \Phi_i x_i$ are computed as
\begin{equation}
	\begin{aligned}
		\Phi_1 &= \begin{bmatrix}
			40.98 & 28.29 \\ 28.29 & 43.73
		\end{bmatrix} 
		\Phi_2 = \begin{bmatrix}
			32.07 & 20.90 \\ 20.90 & 35.91
		\end{bmatrix} \\
		&\quad\quad\quad \Phi_3 = \begin{bmatrix}
			31.97 & 20.83 \\ 20.83 & 35.07
		\end{bmatrix}.
	\end{aligned}
\end{equation}
For the system with the increased coupling \eqref{eq:strong_coupling} the continuous switching case is observed to occur rarely.
Algorithm \ref{alg:switching_ADMM} is hence run with $T_\text{ADMM}=75$, $T_\text{cut}=50$, and $\rho = 5$, with these values again hand chosen to ensure that the total residual is less than 0.01.
Fig.  \ref{fig:ex_1_unstab} shows a closed-loop trajectory under this stronger coupling, demonstrating how our approach can provide stabilizing control for stronger coupling, due to its explicit handling of shared states and consequently less strict requirements on the terminal sets.
Furthermore, Fig. \ref{fig:ex_1_u} demonstrates the convergence of Algorithm \ref{alg:switching_ADMM} in the third time step when the continuous switching edge case occurs.
Subsystem 1 continuously switches its local switching sequence, causing the local input and global residual to not converge.
For iterations larger than $T_\text{cut}=50$ the local switching sequence does not change, and the local input and global residual converge.

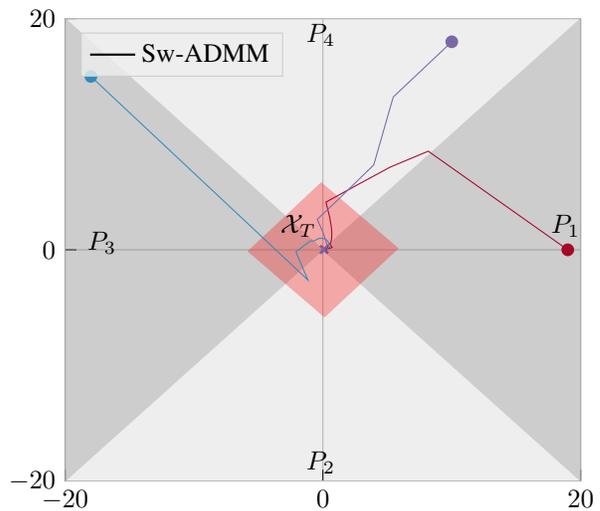
\begin{figure}
	\centering
	\input{media/tikz/ex_1_unstab}
	\caption{State trajectories for SwA with strong coupling.}
	\label{fig:ex_1_unstab}
\end{figure}
\begin{figure}
	\centering
	\subfloat[Control inputs. Red dots indicate when a local switch of $s_i$ occurred.]{\input{media/tikz/ex_1_u}}\\
	\subfloat[Residual.]{\input{media/tikz/ex_1_r}}
	\caption{Demonstration of the effect of $T_\text{cut}$ in Algorithm \ref{alg:switching_ADMM}.}
	\label{fig:ex_1_u}
\end{figure}
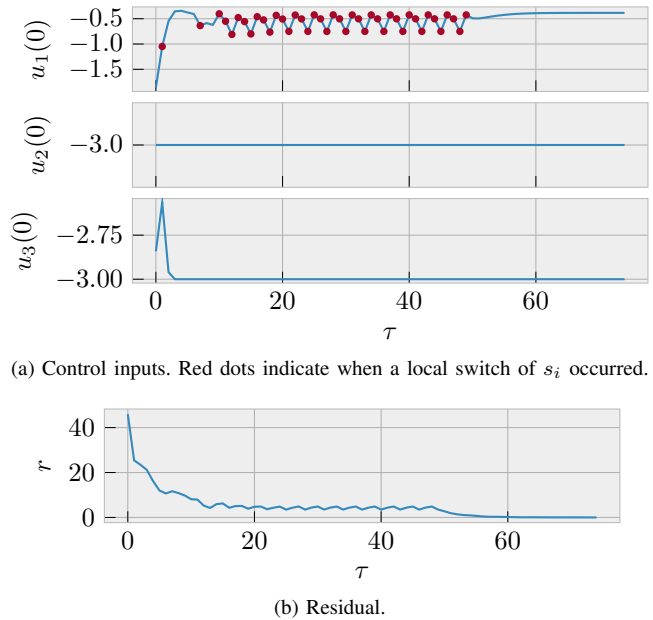

\subsection{Hybrid Vehicle Platooning}
We now consider the problem of hybrid vehicle platooning that was presented as a benchmark problem for hybrid distributed MPC methods in \cite{mallick2024comparison}.
\begin{color}{black}This example is a tracking problem and does not satisfy Assumptions \ref{as:system}$-$\ref{as:lypanov}; therefore, the approach of \cite{maDistributedMPCBased2023} cannot be used.\end{color}
Consider $M$ vehicles that must form a platoon configuration while traveling in a single lane.
Each vehicle $i$ has state $x_i = \begin{bmatrix}
	p_i & v_i
\end{bmatrix}^\top$, where $p_i \in \mathbb{R}$ and $v_i \in \mathbb{R}$ are the position and velocity of the vehicle, and control $u_i \in \mathbb{R}$, the normalized throttle position.
The PWA vehicle dynamics $x_i^+ = f_i(x_i, u_i)$ have seven PWA regions, modeling discrete gear changes that switch dependent on velocity, changing the traction force generated by the vehicle's control input.
For the sake of space the reader is referred to \cite{mallick2024comparison} for the full dynamics.
The control goal is for each vehicle to maintain a desired separation from the vehicle in front of it while the front vehicle follows a reference trajectory.
Without loss of generality it is assumed that the set of vehicles $\mathcal{M}$ is ordered by the positions of vehicles, such that vehicle $1$ is the front vehicle, and vehicle $i$ tracks the trajectory of vehicle $i-1$.
The stage costs are then
\begin{equation}
	\begin{aligned}
		\ell_1\big(x_1, u_1\big) &= \|x_i - r - \eta\|_Q + \|u_i\|_R \\
		\ell_i\big(x_i, u_i, x_{i-1}\big) &= \|x_i - x_{i-1} - \eta\|_Q + \|u_i\|_R,
	\end{aligned}
\end{equation}
for $ i = 2,\cdots,M$ with $Q = \begin{bmatrix}
	1 & 0 \\ 0 & 0.1
\end{bmatrix}$, $R = 1$, $\eta = \begin{bmatrix}
50 & 0
\end{bmatrix}^\top$ and $r \in \mathbb{R}^{2}$ the reference state.
Velocity and position bounds 
\begin{equation}
	\begin{aligned}
		3.94 \leq v_i \leq 45.84 \\
		0 \leq p_i \leq 10000
	\end{aligned}
\end{equation}
define local convex constraints on states and inputs, and the coupling constraints
\begin{equation}
	h_i(x_i, x_{i-1}) = x_i - x_{i-1} + d_\text{safe} \leq 0
\end{equation}
enforce a safe distance $d_\text{safe} = 25$ between vehicles.
Vehicles are initialized with uniformly randomized velocities in the range $\begin{bmatrix}
	5 & 30
\end{bmatrix}$ and with uniformly randomized positions in the range $\begin{bmatrix}
50 & 100
\end{bmatrix}$ behind the preceding vehicle. 

We compare our approach against existing controllers outlined in \cite{mallick2024comparison}.
For our approach, referred to as SwA, we again use a horizon of $N = 5$, $T_\text{ADMM} = 100$ and $\rho = 0.5$ for Algorithm \ref{alg:switching_ADMM}, again ensuring the total residual is less than 0.01.
Once again the edge case of continuous switching is never observed; $T_\text{cut}$ is consequently set to also be 100.
We implement a multi-start approach with two initial guesses, choosing the solution that yields a lower cost.
At time step $t$ the shifted solution from time step $t-1$ is used, appending a constant control value,
\begin{equation}
	\textbf{u}_i = \big(\bar{u}^\top_i(1), \dots, \bar{u}^\top_i(N-1), \bar{u}^\top_i(N-1) \big)^\top,
\end{equation}
as well as a constant control sequence, maintaining the current velocity $v_{i}$.
The MPC controllers for comparison from \cite{mallick2024comparison} are based on converting the PWA dynamics to MLD form \cite{bemporadControlSystemsIntegrating1999} and solving MIQPs.
The centralized controller, referred to as C, solves $\mathcal{P}_\text{d}$ directly as one MIQP, considering all vehicles in the one optimization problem, solving for all control inputs while considering all couplings.
The sequential controller, referred to as S, is a distributed controller where each vehicle solves a local MIQP once, in a predefined order, communicating the solutions down the sequence of vehicles.
Finally, the non-convex ADMM controller, referred to as NCA, applies ADMM directly to $\mathcal{P}_\text{d}$, with vehicles iteratively solving local MIQPs.
For NCA 100 iterations are used for a fair comparison between the two ADMM-based approaches.
\begin{color}{black}Optimizing the solver choice for computation time, for the SwA approach the quadratic programs are solved with the qpOASES solver \cite{Ferreau2014}, while all MIQPs are solved with Gurobi \cite{gurobi}.\end{color}

To further illustrate the mechanism of Algorithm \ref{alg:switching_ADMM} Fig.  \ref{fig:ex_2_u} shows the local solutions $u_i(0)$ throughout the first 20 iterations\footnote{For the remaining 80 iterations the values change negligibly.} at the first time step $t = 0$ of a simulation with $M = 5$.
Red dots indicate iterations in which local systems changed to new switching sequences, changing the global approximation $\mathcal{P}_s$.
It is seen that the inputs vary sharply as vehicles iteratively solve their local minimization problems, communicate solutions, and identify new local switching sequences.
Eventually the changing of switching sequences stops, indicating that the local solutions are no longer pushed to the boundaries of the local constraint sets, and the solutions have settled in a local minimum of $\mathcal{P}_\text{d}$.
Fig.  \ref{fig:ex_2_r} shows the total residual over the entire 100 iterations, demonstrating that subsystems come to agree on the values of the coupled states, as in Theorem \ref{prop:feas}.
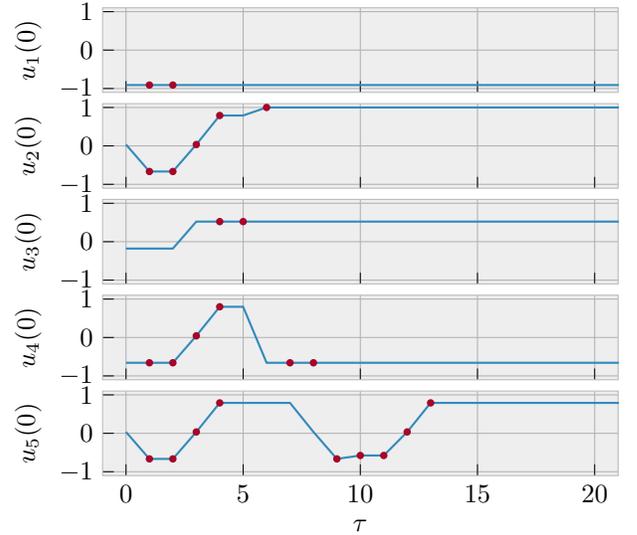
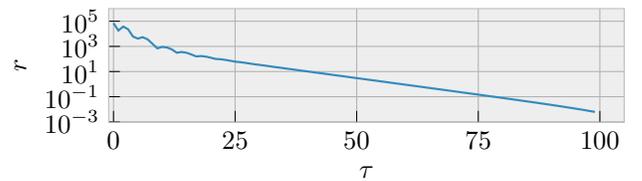
\begin{figure}
	\centering
	\subfloat[Control inputs. Red dots indicate when a local switch of $s_i$ occurred.\label{fig:ex_2_u}]{\input{media/tikz/ex_2_u}}\\
	\subfloat[Residual (log scale).\label{fig:ex_2_r}]{\input{media/tikz/ex_2_r}}
	\caption{Iterations of Algorithm \ref{alg:switching_ADMM}.}
\end{figure}

Fig.  \ref{fig:ex_2_J} shows, for $M = 10$ vehicles, the cumulative tracking cost
\begin{equation}
	\begin{aligned}
		J(t) = &\sum_{\tau = 0}^t \ell_{1}\big(x_1(\tau), u_1(\tau)\big) \\
		 &+ \sum_{i \in \mathcal{M}\setminus\{1\}} \ell_i\big(x_i(\tau), u_i(\tau), x_{i-1}(\tau)\big) 
	\end{aligned}
\end{equation} 
under each of the different controllers.
All controllers manage to reach a stable platoon formation, where the tracking cost no longer grows.
The centralized controller naturally performs the best as it finds the global optimum of $\mathcal{P}_\text{d}$. 
The SwA controller outperforms both other distributed control solutions.
Fig.  \ref{fig:ex_2_n} shows the total tracking cost and the computation times for each controller with $M = 5,10$, and $15$ vehicles.
As the number of vehicles increases the performance of the sequential and non-convex ADMM controllers decreases significantly with respect to the centralized controller, while the performance drop of the SwA controller scales much better.
The computation times show that the SwA is the fastest controller as it involves the solutions to only QPs rather than MIQPs.
Both the NCA and SwA approaches have computation times that scale elegantly with $M$, as the size of the optimization problems and the number of iterations is independent of $M$.
However, as NCA solves MIQPs iteratively rather than QPs, it has a significantly higher computational burden.
\begin{figure}
	\centering
	\input{media/tikz/ex_2_J}
	\caption{Cumulative tracking cost for $M = 10$.}
	\label{fig:ex_2_J}
\end{figure}
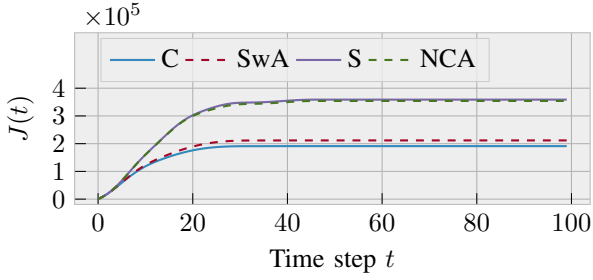
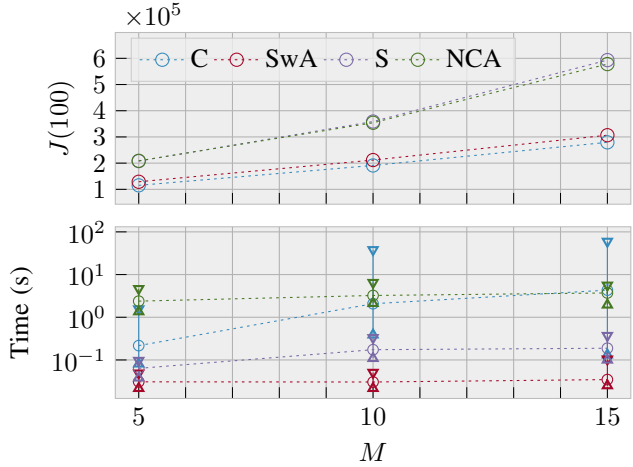
\begin{figure}
	\centering
	\input{media/tikz/ex_2_n}
	\caption{Tracking cost (top) and max/average/min computation time (bottom) as $M$ increases.}
	\label{fig:ex_2_n}
\end{figure}

\section{Conclusions}\label{sec:conclusion}
This paper has presented a new approach for distributed MPC for PWA systems.
The proposed controller is based on a novel switching ADMM procedure that solves a globally formulated non-convex optimal control problem, without requiring mixed-integer programming, and provably giving agreement on the values of shared states.
The recursive feasibility and stability of the resulting control scheme is proven under additional assumptions on the system.
The theoretical properties of the approach have been demonstrated on a small numerical example, whilst the practical utility has been demonstrated on a larger, more realistic, hybrid control problem.

Future work will look at bounding the level of suboptimality between the closed loop performance of the proposed controller and the theoretically optimal performance of a centralized mixed-integer-based MPC controller.

\appendices
\section{Computation of terminal sets in Assumption \ref{as:term}}
\label{ap:sets}
The procedure presented in \cite{maDistributedMPCBased2023} computes local positively invariant sets that are robust to the effects of coupling from neighboring sub-systems.
The procedure involves, for each sub-system $i$, computing a sequence of sets, beginning from a \textit{reasonably large polytope} $\mathcal{X}_{0, i} \subseteq \bigcup_{l \in \hat{\mathcal{L}}_i}\{x_i \in P_i^{(l)} | K_i^{(l)} x_i \in \mathcal{U}_i\}$, terminating in the robustly positive invariant set $\mathcal{X}_{\text{T}, i}$, which each set in the sequence being a contraction of the previous one.
See \cite{maDistributedMPCBased2023} for construction of $\mathcal{X}_{0, i}$.

In this appendix a modified procedure with respect to that in \cite{maDistributedMPCBased2023} is presented for which the resulting robustly positive invariant sets are robust only to coupling for neighboring sub-systems within their own local terminal sets, as by Assumption \ref{as:term}.
A formal proof of the positive invariance, missing in \cite{maDistributedMPCBased2023}, is also provided.

Define the matrix $A_{\text{cl}, i}^{(l)} = A_i^{(l)} + B_i^{(l)} K_i^{(l)}$. 
Define the set operation
\begin{equation}
	\begin{aligned}
		\mathcal{Q}_i^{(l)}(\mathcal{X}, \mathcal{W}) &= \{x_i \in \mathbb{R}^{n_i} | A_{\text{cl}, i}^{(l)} x_i + w \in \mathcal{X}, \forall w \in \mathcal{W} \} \\
		&= (A_{\text{cl}, i}^{(l)})^{-1} (\mathcal{X} \ominus \mathcal{W}),
	\end{aligned}
\end{equation}
i.e., a one-step controllable set, robust to disturbances in the set $\mathcal{W}$.
Algorithm \ref{alg:sets} outlines the procedure for computing $\mathcal{X}_{\text{T}, i}$ for $i \in \mathcal{M}$.
\begin{algorithm}
	\caption{$\text{Computation of $\mathcal{X}_{\text{T}, i}$'s}$}\label{alg:sets}
	\begin{algorithmic}[1]
		\State \textbf{Inputs:} Initial sets $\mathcal{X}_{0, i}, \forall i \in \mathcal{M}$
		\State \textbf{Initialize:} $\mathcal{X}_{1, i} \gets \emptyset, \forall i \in \mathcal{M}$ and $\tau \gets 1$
 		\While{$\mathcal{X}_{\tau, i} \neq \mathcal{X}_{\tau-1, i}$ for some $i \in \mathcal{M}$}
 			\For{$i \in \mathcal{M}$}
				\State $\mathcal{W}_{\tau-1, i} \gets \bigoplus_{j \in \mathcal{N}_i} A_{ij} \mathcal{X}_{j, \tau-1}, \forall j \in \mathcal{N}_i$
				\State $\mathcal{X}_{\tau, i}^{(l)} \gets \mathcal{Q}_i^{(l)}(\mathcal{X}_{\tau-1, i}, \mathcal{W}_{\tau-1, i}) \cap \mathcal{X}_{\tau-1, i}, \forall l \in \hat{\mathcal{L}}_i$
				\State $\mathcal{X}_{\tau, i} \gets \bigcap_{l \in \hat{\mathcal{L}}_i} \mathcal{X}_{\tau, i}^{(l)}$
			\EndFor
			\State $\tau \gets \tau + 1$
		\EndWhile
		\State $\mathcal{X}_{\text{T}, i} \gets \mathcal{X}_{\tau-1, i}, \forall i \in \mathcal{M}$
		\State \textbf{Outputs:} $\mathcal{X}_{\text{T}, i}, \forall i \in \mathcal{M}$
	\end{algorithmic}
\end{algorithm}
\begin{assumption}[\cite{maDistributedMPCBased2023}]
	\label{as:sets_exist}
	For all $i \in \mathcal{M}$ there exists $\tau \in \mathbb{N}^+$ and $\tau < \infty$ such that $\mathcal{X}_{\tau, i} = \mathcal{X}_{\tau-1, i}$ and $\mathcal{X}_{\tau, i} \neq \emptyset$.
\end{assumption}
\begin{theorem}
	Under Assumption \ref{as:sets_exist} the output sets from Algorithm \ref{alg:sets} are robustly positive invariant as in Assumption \ref{as:term}.
\end{theorem}
\begin{proof}
	As at each iteration $\tau$ of Algorithm \ref{alg:sets} the set $\mathcal{X}_{\tau, i}$ shrinks, then for all $i \in \mathcal{M}$, if $x_i \in \mathcal{X}_{\text{T}, i}$ and $x_j \in \mathcal{X}_{\text{T}, j}$ for all $j \in \mathcal{N}_i$, then $x_i \in \mathcal{X}_{\tau, i}$ and $x_j \in \mathcal{X}_{\tau-1, i}$ for all $j \in \mathcal{N}_i$ and for all $\tau \geq 1$.
	Then, by the definitions of $\mathcal{Q}_i^{(l)}$ and $\mathcal{W}_{\tau-1, i}$, $A_{\text{cl}, i}^{(l)} x_i + \sum_{j \in \mathcal{N}_i} A_{ij} x_j \in \mathcal{X}_{\tau-1, i}$ for all $l \in \Hat{\mathcal{L}}_i$ and all $\tau \geq 2$.
	Then $A_{\text{cl}, i}^{(l)} x_i + \sum_{j \in \mathcal{N}_i} A_{ij} x_j \in \mathcal{X}_{\text{T}, i}$ for all $l \in \Hat{\mathcal{L}}_i$.
	Hence, $\mathcal{X}_{\text{T}, i}$ is robustly positive invariant, robust to the coupling of neighboring sub-systems within their respective terminal sets, under arbitrary switching between the PWA regions $l \in \Hat{\mathcal{L}}_i$.
	It follows directly that $\mathcal{X}_{\text{T}, i}$ is robustly positive invariant for the PWA sub-system under the true switching.
\end{proof}
\section{Computation of $V_{{\lowercase{f}}, i}$} \label{ap:LMI}
Given the structures of $\ell_i$ and $V_{\text{T}, i}$ \eqref{eq:lyapnuov} can be written as
\begin{equation}
	\label{eq:lyapunov_expanded}
	\begin{aligned}
		\sum_{i \in \mathcal{M}} &\big(\star\big)^\top \Phi_i \big(  (A_i^{(l_i)} + B_i^{(l_i)} K_i^{(l_i)}) x_i + \sum_{j \in \mathcal{N}_i} A_{ij}^{(l_i)}x_j \big) \\
		&- x_i^\top \Phi_i x_i + x_i^\top Q_i x_i + (\star)^\top R_i (K_i^{(l_i)}x_i) \\
		&+ \sum_{j \in \mathcal{N}_i} x_j^\top Q_{ij} x_j \leq 0, x_i \in P_i^{(l_i)},
	\end{aligned}
\end{equation}
for all $x_i \in \mathcal{X}_{\text{T}, i}, l_i \in \hat{\mathcal{L}}_i, i \in \mathcal{M}$, recalling that $c_i^{(l_i)} = 0$ for $l_i \in \hat{\mathcal{L}}_i$.
Finding $\Phi_i$'s such that \eqref{eq:lyapunov_expanded} is satisfied can be achieved by considering the system from a global view, ensuring the inequality holds for every combination of PWA regions of the subsystems.
This is equivalent to considering the global system as a single large PWA system with $\prod_{i \in \mathcal{M}}L_i$ regions and using the approach presented in \cite{lazarStabilizingModelPredictive2006} for single PWA systems.
Condition \eqref{eq:lyapunov_expanded} can be reformulated as the set of LMIs
\begin{equation}
	\label{eq:LMI}
	\begin{aligned}
		A_\text{cl}^\top(l) &\Phi A_\text{cl}(l) - \Phi + Q + K(l) R K(l) \prec 0, \\
		 &\forall l = (l_1,\dots,l_M) \in \hat{\mathcal{L}}_1 \times \dots \times \hat{\mathcal{L}}_M
	\end{aligned}
\end{equation}
where $\Phi = \text{blck}(\Phi_1, \dots, \Phi_M)$, $Q = \text{blck}(Q_1 + \sum_{j|1 \in \mathcal{N}_j} Q_{j1},\dots,Q_M + \sum_{j | M \in \mathcal{N}_j}Q_{jM})$, $R = \text{blck}(R_1, \dots, R_M)$, $K(l) = \text{blck}(K_1^{(l_1)}, \dots, K_M^{(M)})$, and
\begin{equation}
	\begin{aligned}
		A_\text{cl}(l) &= \begin{bmatrix}
		A_\text{cl, 1}^{(l_1)} & A_{12}^{(l_1)} & \dots & A_{1M}^{(l_1)} \\
		A_{21}^{(l_2)} & \ddots & & \vdots \\
		\vdots & & \ddots & \\
		A_{M1}^{(l_M)} & & & A_\text{cl, M}^{(l_M)}.
	\end{bmatrix}
	\end{aligned}
\end{equation}
with $A_{ij}^{(l_i)} = 0$ if $j \notin \mathcal{N}_i$.
Evidently, \eqref{eq:LMI} requires centralized computation, and may become intractable for large $M$ or $|\hat{\mathcal{L}}_i|$.
We leave, however, distributed and consequently less computationally demanding solutions to future work.
Finally, \eqref{eq:LMI} is linear in the matrix $\Phi$ as the linear controllers $K_i^{(l)}$ have been assumed to be computed \textit{a priori} as in \cite{maDistributedMPCBased2023}.
The linear controllers could be computed simultaneously using the same expression, leveraging the Schur complement and a change of variables to maintain an LMI expression \cite{lazarStabilizingModelPredictive2006}.

\begin{color}{black}\section{Suboptimality of Algorithm \ref{alg:switching_ADMM}}
	\label{sec:ap3}
	To investigate the suboptimality of Algorithm \ref{alg:switching_ADMM}, the MPC optimization problem $\mathcal{P}_\text{d}$ is solved for 1000 randomly generated two-subsystem PWA networks and states. 
	Each subsystem has two states, one input, and four PWA regions.
	All values defining the local state-space matrices take random values between $\begin{bmatrix}
		-1.5 & 1.5
	\end{bmatrix}$ (stable and unstable systems), while the coupling matrices take random values between $\begin{bmatrix}
		-0.1 & 0.1
	\end{bmatrix}$.
	The four PWA regions are the four quadrants of the Cartesian plane. The local stage costs are $\ell_i = x_i^\top x_i + u_i^\top u_i$, and the local constraints are $|u_i| \leq 1$ and $\|x_i\|_\infty \leq 10$.
	No terminal costs or constraints are considered.
	Problem $\mathcal{P}_\text{d}$ is solved by Algorithm \ref{alg:switching_ADMM} for states $x$ randomly selected with values in the range $\begin{bmatrix}
		-10 & 10
	\end{bmatrix}$, with $\textbf{u}$ the resulting global control sequence.
	To avoid the manual tuning of $\rho$, $T_\text{ADMM}$, and $T_\text{cut}$ for each random system, $T_\text{cut}$ is set to 100 and the algorithm is run until the residual is less than 0.01.
	Control sequences of all zeros are used as weakly feasible initial guesses.
	Problem $\mathcal{P}_\text{d}$ is also reformulated as an MIQP and solved to optimality with Gurobi \cite{gurobi}, with $\textbf{u}^\ast$ the optimal global control sequence.
	Over all randomized networks and states the suboptimality
	\begin{equation}
		\Tilde{V} = 100 \cdot \frac{V(x, \textbf{u}) - V(x, \textbf{u}^\ast)}{V(x, \textbf{u}^\ast)}
	\end{equation} 
	has median $0$, mean $16.0$, standard deviation $120.7$, maximum $1664.2$, and minimum $0$.
	This demonstrates that, while there exists cases in which arbitrarily bad local minima are found, these are outliers, with Algorithm \ref{alg:switching_ADMM} in general finding solutions with low suboptimality.
	As demonstrated and discussed in Section \ref{sec:ex}, the suboptimality can be reduced using a multi-start approach, and, when Algorithm \ref{alg:stable_dmpc} is deployed, stability is guaranteed.
\end{color}

\section*{Acknowledgment}
The authors would like to thank the authors of \cite{maDistributedMPCBased2023}, A. Ma, D. Li, and Y. Li, for helpfully providing the code for their approach.

\bibliographystyle{IEEEtran}
\bibliography{references.bib}
\vspace*{-1cm}

\begin{IEEEbiography}[{\includegraphics[width=1in,height=1.25in,clip,keepaspectratio]{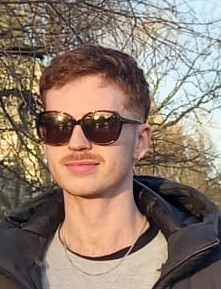}}]{Samuel Mallick} 
	received the B.Sc. and M.Sc. degrees from The University of Melbourne in 2020 and 2022, respectively. He is currently a Ph.D. candidate at the Delft Center for Systems and Control, Delft University of Technology, The Netherlands.
	
	His research interests include model predictive control, reinforcement learning, and distributed control of large-scale and hybrid systems.
\end{IEEEbiography}

\vspace*{-1cm}

\begin{IEEEbiography}[{\includegraphics[width=1in,height=1.25in,clip,keepaspectratio]{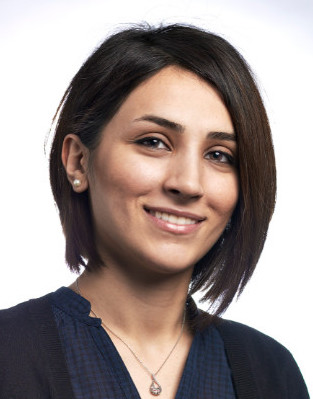}}]{Azita Dabiri} 
	received the Ph.D. degree from the Automatic Control Group, Chalmers University of Technology, in 2016. She was a Post-Doctoral Researcher with the Department of Transport and Planning, Delft University of Technology, from 2017 to 2019. In 2019, she received an ERCIM Fellowship and also a Marie Curie Individual Fellowship, which allowed her to perform research at the Norwegian University of Technology (NTNU), as a Post-Doctoral Researcher, from 2019 to 2020, before joining the Delft Center for Systems and Control, Delft University of Technology, as an Assistant Professor. Her research interests are in the areas of integration of model-based and learning-based control and its applications in transportation networks.
\end{IEEEbiography}
\vspace*{-1cm}

\begin{IEEEbiography}[{\includegraphics[width=1in,height=1.25in,clip,keepaspectratio]{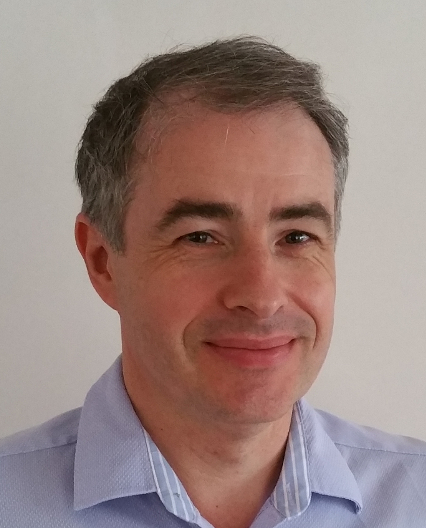}}]{Bart De Schutter}(Fellow, IEEE) 
	received the PhD degree (\emph{summa cum laude}) in applied sciences from KU Leuven, Belgium, in 1996. He is currently a Full Professor and Head of Department at the Delft Center for Systems and Control, Delft
	University of Technology, The Netherlands. His research interests include multi-level
	and multi-agent control, model predictive control, learning-based control, and control
	of hybrid systems, with applications in intelligent transportation systems and smart energy systems. 
	
	Prof.\ De Schutter is a Senior Editor of the IEEE Transactions on Intelligent Transportation Systems.
\end{IEEEbiography}

\vfill

\end{document}

%% file: media/tikz/eg_2D.tex
%
%
\definecolor{mycolor1}{rgb}{0.92900,0.69400,0.12500}%
\definecolor{mycolor2}{rgb}{0.00000,0.44700,0.74100}%
\definecolor{mycolor3}{rgb}{0.49400,0.18400,0.55600}%
\definecolor{mycolor4}{rgb}{0.46600,0.67400,0.18800}%
\definecolor{mycolor5}{rgb}{0.85000,0.32500,0.09800}%
\definecolor{mycolor6}{rgb}{0.30100,0.74500,0.93300}%
\definecolor{mycolor7}{rgb}{0.63500,0.07800,0.18400}%
\begin{tikzpicture}

\begin{axis}[%
height=0.4*\axisdefaultheight,
width=\axisdefaultwidth,
at={(0.758in,0.481in)},
scale only axis,
xmin=-2,
xmax=2,
ticks=none,
xtick={},
xticklabels={{},{},{},{},{},{},{},{},{}},
ymin=-1,
ymax=2,
ytick={-1,-0.5,0,0.5,1,1.5,2},
yticklabels={{},{},{},{},{},{},{}},
axis background/.style={fill=white}
]
\addplot [color=black, forget plot]
  table[row sep=crcr]{%
-2	0\\
2	0\\
};
\addplot [color=black, forget plot]
  table[row sep=crcr]{%
-2	1\\
2	1\\
};

\newcommand{\ms}{3.0pt}
\addplot [color=mycolor1, mark size=\ms, mark=*, mark options={solid, fill=mycolor2, draw=mycolor2}, forget plot]
  table[row sep=crcr]{%
-1.5	0\\
};
\addplot [color=mycolor3, mark size=\ms, mark=*, mark options={solid, fill=mycolor2, draw=mycolor2}, forget plot]
  table[row sep=crcr]{%
-1.5	1\\
};
\addplot [color=mycolor4, mark size=\ms, mark=*, mark options={solid, fill=mycolor5, draw=mycolor5}, forget plot]
  table[row sep=crcr]{%
-1.2	0\\
};
\addplot [color=mycolor6, mark size=\ms, mark=*, mark options={solid, fill=mycolor5, draw=mycolor5}, forget plot]
  table[row sep=crcr]{%
-0	1\\
};
\addplot [color=mycolor7, mark size=\ms, mark=*, mark options={solid, fill=mycolor1, draw=mycolor1}, forget plot]
  table[row sep=crcr]{%
-0.9	0\\
};
\addplot [color=mycolor2, mark size=\ms, mark=*, mark options={solid, fill=mycolor1, draw=mycolor1}, forget plot]
  table[row sep=crcr]{%
0.9	1\\
};
\addplot [color=mycolor5, mark size=\ms, mark=*, mark options={solid, fill=mycolor3, draw=mycolor3}, forget plot]
  table[row sep=crcr]{%
-0	0\\
};
\addplot [color=mycolor1, mark size=\ms, mark=*, mark options={solid, fill=mycolor3, draw=mycolor3}, forget plot]
  table[row sep=crcr]{%
0.8	1\\
};
\addplot [color=mycolor3, mark size=\ms, mark=*, mark options={solid, fill=mycolor4, draw=mycolor4}, forget plot]
  table[row sep=crcr]{%
0.2	0\\
};
\addplot [color=mycolor4, mark size=\ms, mark=*, mark options={solid, fill=mycolor4, draw=mycolor4}, forget plot]
  table[row sep=crcr]{%
0.7	1\\
};
\addplot [color=mycolor6, mark size=\ms, mark=*, mark options={solid, fill=mycolor6, draw=mycolor6}, forget plot]
  table[row sep=crcr]{%
0.4	0\\
};
\addplot [color=mycolor7, mark size=\ms, mark=*, mark options={solid, fill=mycolor6, draw=mycolor6}, forget plot]
  table[row sep=crcr]{%
0.6	1\\
};
\addplot [color=mycolor2, mark size=\ms, mark=*, mark options={solid, fill=mycolor7, draw=mycolor7}, forget plot]
  table[row sep=crcr]{%
0.5	0\\
};
\addplot [color=mycolor5, mark size=\ms, mark=*, mark options={solid, fill=mycolor7, draw=mycolor7}, forget plot]
  table[row sep=crcr]{%
0.5	1\\
};
\addplot [color=black, dashed, forget plot]
  table[row sep=crcr]{%
0	-3\\
0	0.5\\
};
\addplot [color=black, dashed, forget plot]
  table[row sep=crcr]{%
0	0.5\\
0	3\\
};
\node[right, align=left, inner sep=0, font=\bfseries]
at (axis cs:-1.9,-0.4) {$x_1(1)$};
\node[right, align=left, inner sep=0, font=\bfseries]
at (axis cs:-1.9,1.5) {$x_2(1)$};
\node[right, align=left, inner sep=0, font=\bfseries]
at (axis cs:-0.3,-0.5) {${P}_1^{(1)}$};
\node[right, align=left, inner sep=0, font=\bfseries]
at (axis cs:0.1,-0.5) {${{P}}_1^{(2)}$};
\node[right, align=left, inner sep=0, font=\bfseries]
at (axis cs:-0.3,1.5) {${P}_2^{(1)}$};
\node[right, align=left, inner sep=0, font=\bfseries]
at (axis cs:0.1,1.5) {${{P}}_2^{(2)}$};
\end{axis}
\end{tikzpicture}%

%% file: media/tikz/ex_1_traj.tex
\begin{tikzpicture}

\definecolor{darkgray178}{RGB}{178,178,178}
\definecolor{firebrick166640}{RGB}{166,6,40}
\definecolor{gray}{RGB}{128,128,128}
\definecolor{lightgray204}{RGB}{204,204,204}
\definecolor{silver188}{RGB}{188,188,188}
\definecolor{slategray122104166}{RGB}{122,104,166}
\definecolor{steelblue52138189}{RGB}{52,138,189}
\definecolor{whitesmoke238}{RGB}{238,238,238}

\begin{axis}[
width=\axisdefaultwidth,
height=2*\axisdefaultheight,
axis background/.style={fill=whitesmoke238},
axis line style={silver188},
legend cell align={left},
legend style={
  fill opacity=0.8,
  draw opacity=1,
  text opacity=1,
  at={(0.03,0.97)},
  anchor=north west,
  draw=lightgray204,
  fill=whitesmoke238
},
tick pos=left,
tick scale binop=\times,
x grid style={darkgray178},
xmajorgrids,
xmin=-20, xmax=20,
xtick style={color=black},
xtick={-20,0,20},
xticklabels={
  \(\displaystyle {\ensuremath{-}20}\),
  \(\displaystyle {0}\),
  \(\displaystyle {20}\)
},
y grid style={darkgray178},
ymajorgrids,
ymin=-20, ymax=20,
ytick style={color=black},
ytick={-20,0,20},
yticklabels={
  \(\displaystyle {\ensuremath{-}20}\),
  \(\displaystyle {0}\),
  \(\displaystyle {20}\)
}
]
\path [draw=gray, fill=gray, opacity=0.3, very thin]
(axis cs:0,0)
--(axis cs:0,0)
--(axis cs:0.202020202020202,0.202020202020202)
--(axis cs:0.404040404040404,0.404040404040404)
--(axis cs:0.606060606060606,0.606060606060606)
--(axis cs:0.808080808080808,0.808080808080808)
--(axis cs:1.01010101010101,1.01010101010101)
--(axis cs:1.21212121212121,1.21212121212121)
--(axis cs:1.41414141414141,1.41414141414141)
--(axis cs:1.61616161616162,1.61616161616162)
--(axis cs:1.81818181818182,1.81818181818182)
--(axis cs:2.02020202020202,2.02020202020202)
--(axis cs:2.22222222222222,2.22222222222222)
--(axis cs:2.42424242424242,2.42424242424242)
--(axis cs:2.62626262626263,2.62626262626263)
--(axis cs:2.82828282828283,2.82828282828283)
--(axis cs:3.03030303030303,3.03030303030303)
--(axis cs:3.23232323232323,3.23232323232323)
--(axis cs:3.43434343434343,3.43434343434343)
--(axis cs:3.63636363636364,3.63636363636364)
--(axis cs:3.83838383838384,3.83838383838384)
--(axis cs:4.04040404040404,4.04040404040404)
--(axis cs:4.24242424242424,4.24242424242424)
--(axis cs:4.44444444444444,4.44444444444444)
--(axis cs:4.64646464646465,4.64646464646465)
--(axis cs:4.84848484848485,4.84848484848485)
--(axis cs:5.05050505050505,5.05050505050505)
--(axis cs:5.25252525252525,5.25252525252525)
--(axis cs:5.45454545454545,5.45454545454545)
--(axis cs:5.65656565656566,5.65656565656566)
--(axis cs:5.85858585858586,5.85858585858586)
--(axis cs:6.06060606060606,6.06060606060606)
--(axis cs:6.26262626262626,6.26262626262626)
--(axis cs:6.46464646464646,6.46464646464646)
--(axis cs:6.66666666666667,6.66666666666667)
--(axis cs:6.86868686868687,6.86868686868687)
--(axis cs:7.07070707070707,7.07070707070707)
--(axis cs:7.27272727272727,7.27272727272727)
--(axis cs:7.47474747474747,7.47474747474747)
--(axis cs:7.67676767676768,7.67676767676768)
--(axis cs:7.87878787878788,7.87878787878788)
--(axis cs:8.08080808080808,8.08080808080808)
--(axis cs:8.28282828282828,8.28282828282828)
--(axis cs:8.48484848484848,8.48484848484848)
--(axis cs:8.68686868686869,8.68686868686869)
--(axis cs:8.88888888888889,8.88888888888889)
--(axis cs:9.09090909090909,9.09090909090909)
--(axis cs:9.29292929292929,9.29292929292929)
--(axis cs:9.49494949494949,9.49494949494949)
--(axis cs:9.6969696969697,9.6969696969697)
--(axis cs:9.8989898989899,9.8989898989899)
--(axis cs:10.1010101010101,10.1010101010101)
--(axis cs:10.3030303030303,10.3030303030303)
--(axis cs:10.5050505050505,10.5050505050505)
--(axis cs:10.7070707070707,10.7070707070707)
--(axis cs:10.9090909090909,10.9090909090909)
--(axis cs:11.1111111111111,11.1111111111111)
--(axis cs:11.3131313131313,11.3131313131313)
--(axis cs:11.5151515151515,11.5151515151515)
--(axis cs:11.7171717171717,11.7171717171717)
--(axis cs:11.9191919191919,11.9191919191919)
--(axis cs:12.1212121212121,12.1212121212121)
--(axis cs:12.3232323232323,12.3232323232323)
--(axis cs:12.5252525252525,12.5252525252525)
--(axis cs:12.7272727272727,12.7272727272727)
--(axis cs:12.9292929292929,12.9292929292929)
--(axis cs:13.1313131313131,13.1313131313131)
--(axis cs:13.3333333333333,13.3333333333333)
--(axis cs:13.5353535353535,13.5353535353535)
--(axis cs:13.7373737373737,13.7373737373737)
--(axis cs:13.9393939393939,13.9393939393939)
--(axis cs:14.1414141414141,14.1414141414141)
--(axis cs:14.3434343434343,14.3434343434343)
--(axis cs:14.5454545454545,14.5454545454545)
--(axis cs:14.7474747474747,14.7474747474747)
--(axis cs:14.9494949494949,14.9494949494949)
--(axis cs:15.1515151515152,15.1515151515152)
--(axis cs:15.3535353535354,15.3535353535354)
--(axis cs:15.5555555555556,15.5555555555556)
--(axis cs:15.7575757575758,15.7575757575758)
--(axis cs:15.959595959596,15.959595959596)
--(axis cs:16.1616161616162,16.1616161616162)
--(axis cs:16.3636363636364,16.3636363636364)
--(axis cs:16.5656565656566,16.5656565656566)
--(axis cs:16.7676767676768,16.7676767676768)
--(axis cs:16.969696969697,16.969696969697)
--(axis cs:17.1717171717172,17.1717171717172)
--(axis cs:17.3737373737374,17.3737373737374)
--(axis cs:17.5757575757576,17.5757575757576)
--(axis cs:17.7777777777778,17.7777777777778)
--(axis cs:17.979797979798,17.979797979798)
--(axis cs:18.1818181818182,18.1818181818182)
--(axis cs:18.3838383838384,18.3838383838384)
--(axis cs:18.5858585858586,18.5858585858586)
--(axis cs:18.7878787878788,18.7878787878788)
--(axis cs:18.989898989899,18.989898989899)
--(axis cs:19.1919191919192,19.1919191919192)
--(axis cs:19.3939393939394,19.3939393939394)
--(axis cs:19.5959595959596,19.5959595959596)
--(axis cs:19.7979797979798,19.7979797979798)
--(axis cs:20,20)
--(axis cs:20,-20)
--(axis cs:20,-20)
--(axis cs:19.7979797979798,-19.7979797979798)
--(axis cs:19.5959595959596,-19.5959595959596)
--(axis cs:19.3939393939394,-19.3939393939394)
--(axis cs:19.1919191919192,-19.1919191919192)
--(axis cs:18.989898989899,-18.989898989899)
--(axis cs:18.7878787878788,-18.7878787878788)
--(axis cs:18.5858585858586,-18.5858585858586)
--(axis cs:18.3838383838384,-18.3838383838384)
--(axis cs:18.1818181818182,-18.1818181818182)
--(axis cs:17.979797979798,-17.979797979798)
--(axis cs:17.7777777777778,-17.7777777777778)
--(axis cs:17.5757575757576,-17.5757575757576)
--(axis cs:17.3737373737374,-17.3737373737374)
--(axis cs:17.1717171717172,-17.1717171717172)
--(axis cs:16.969696969697,-16.969696969697)
--(axis cs:16.7676767676768,-16.7676767676768)
--(axis cs:16.5656565656566,-16.5656565656566)
--(axis cs:16.3636363636364,-16.3636363636364)
--(axis cs:16.1616161616162,-16.1616161616162)
--(axis cs:15.959595959596,-15.959595959596)
--(axis cs:15.7575757575758,-15.7575757575758)
--(axis cs:15.5555555555556,-15.5555555555556)
--(axis cs:15.3535353535354,-15.3535353535354)
--(axis cs:15.1515151515152,-15.1515151515152)
--(axis cs:14.9494949494949,-14.9494949494949)
--(axis cs:14.7474747474747,-14.7474747474747)
--(axis cs:14.5454545454545,-14.5454545454545)
--(axis cs:14.3434343434343,-14.3434343434343)
--(axis cs:14.1414141414141,-14.1414141414141)
--(axis cs:13.9393939393939,-13.9393939393939)
--(axis cs:13.7373737373737,-13.7373737373737)
--(axis cs:13.5353535353535,-13.5353535353535)
--(axis cs:13.3333333333333,-13.3333333333333)
--(axis cs:13.1313131313131,-13.1313131313131)
--(axis cs:12.9292929292929,-12.9292929292929)
--(axis cs:12.7272727272727,-12.7272727272727)
--(axis cs:12.5252525252525,-12.5252525252525)
--(axis cs:12.3232323232323,-12.3232323232323)
--(axis cs:12.1212121212121,-12.1212121212121)
--(axis cs:11.9191919191919,-11.9191919191919)
--(axis cs:11.7171717171717,-11.7171717171717)
--(axis cs:11.5151515151515,-11.5151515151515)
--(axis cs:11.3131313131313,-11.3131313131313)
--(axis cs:11.1111111111111,-11.1111111111111)
--(axis cs:10.9090909090909,-10.9090909090909)
--(axis cs:10.7070707070707,-10.7070707070707)
--(axis cs:10.5050505050505,-10.5050505050505)
--(axis cs:10.3030303030303,-10.3030303030303)
--(axis cs:10.1010101010101,-10.1010101010101)
--(axis cs:9.8989898989899,-9.8989898989899)
--(axis cs:9.6969696969697,-9.6969696969697)
--(axis cs:9.49494949494949,-9.49494949494949)
--(axis cs:9.29292929292929,-9.29292929292929)
--(axis cs:9.09090909090909,-9.09090909090909)
--(axis cs:8.88888888888889,-8.88888888888889)
--(axis cs:8.68686868686869,-8.68686868686869)
--(axis cs:8.48484848484848,-8.48484848484848)
--(axis cs:8.28282828282828,-8.28282828282828)
--(axis cs:8.08080808080808,-8.08080808080808)
--(axis cs:7.87878787878788,-7.87878787878788)
--(axis cs:7.67676767676768,-7.67676767676768)
--(axis cs:7.47474747474747,-7.47474747474747)
--(axis cs:7.27272727272727,-7.27272727272727)
--(axis cs:7.07070707070707,-7.07070707070707)
--(axis cs:6.86868686868687,-6.86868686868687)
--(axis cs:6.66666666666667,-6.66666666666667)
--(axis cs:6.46464646464646,-6.46464646464646)
--(axis cs:6.26262626262626,-6.26262626262626)
--(axis cs:6.06060606060606,-6.06060606060606)
--(axis cs:5.85858585858586,-5.85858585858586)
--(axis cs:5.65656565656566,-5.65656565656566)
--(axis cs:5.45454545454545,-5.45454545454545)
--(axis cs:5.25252525252525,-5.25252525252525)
--(axis cs:5.05050505050505,-5.05050505050505)
--(axis cs:4.84848484848485,-4.84848484848485)
--(axis cs:4.64646464646465,-4.64646464646465)
--(axis cs:4.44444444444444,-4.44444444444444)
--(axis cs:4.24242424242424,-4.24242424242424)
--(axis cs:4.04040404040404,-4.04040404040404)
--(axis cs:3.83838383838384,-3.83838383838384)
--(axis cs:3.63636363636364,-3.63636363636364)
--(axis cs:3.43434343434343,-3.43434343434343)
--(axis cs:3.23232323232323,-3.23232323232323)
--(axis cs:3.03030303030303,-3.03030303030303)
--(axis cs:2.82828282828283,-2.82828282828283)
--(axis cs:2.62626262626263,-2.62626262626263)
--(axis cs:2.42424242424242,-2.42424242424242)
--(axis cs:2.22222222222222,-2.22222222222222)
--(axis cs:2.02020202020202,-2.02020202020202)
--(axis cs:1.81818181818182,-1.81818181818182)
--(axis cs:1.61616161616162,-1.61616161616162)
--(axis cs:1.41414141414141,-1.41414141414141)
--(axis cs:1.21212121212121,-1.21212121212121)
--(axis cs:1.01010101010101,-1.01010101010101)
--(axis cs:0.808080808080808,-0.808080808080808)
--(axis cs:0.606060606060606,-0.606060606060606)
--(axis cs:0.404040404040404,-0.404040404040404)
--(axis cs:0.202020202020202,-0.202020202020202)
--(axis cs:0,0)
--cycle;

\path [draw=gray, fill=gray, opacity=0.3, very thin]
(axis cs:-20,-20)
--(axis cs:-20,20)
--(axis cs:-19.7979797979798,19.7979797979798)
--(axis cs:-19.5959595959596,19.5959595959596)
--(axis cs:-19.3939393939394,19.3939393939394)
--(axis cs:-19.1919191919192,19.1919191919192)
--(axis cs:-18.989898989899,18.989898989899)
--(axis cs:-18.7878787878788,18.7878787878788)
--(axis cs:-18.5858585858586,18.5858585858586)
--(axis cs:-18.3838383838384,18.3838383838384)
--(axis cs:-18.1818181818182,18.1818181818182)
--(axis cs:-17.979797979798,17.979797979798)
--(axis cs:-17.7777777777778,17.7777777777778)
--(axis cs:-17.5757575757576,17.5757575757576)
--(axis cs:-17.3737373737374,17.3737373737374)
--(axis cs:-17.1717171717172,17.1717171717172)
--(axis cs:-16.969696969697,16.969696969697)
--(axis cs:-16.7676767676768,16.7676767676768)
--(axis cs:-16.5656565656566,16.5656565656566)
--(axis cs:-16.3636363636364,16.3636363636364)
--(axis cs:-16.1616161616162,16.1616161616162)
--(axis cs:-15.959595959596,15.959595959596)
--(axis cs:-15.7575757575758,15.7575757575758)
--(axis cs:-15.5555555555556,15.5555555555556)
--(axis cs:-15.3535353535354,15.3535353535354)
--(axis cs:-15.1515151515152,15.1515151515152)
--(axis cs:-14.9494949494949,14.9494949494949)
--(axis cs:-14.7474747474747,14.7474747474747)
--(axis cs:-14.5454545454545,14.5454545454545)
--(axis cs:-14.3434343434343,14.3434343434343)
--(axis cs:-14.1414141414141,14.1414141414141)
--(axis cs:-13.9393939393939,13.9393939393939)
--(axis cs:-13.7373737373737,13.7373737373737)
--(axis cs:-13.5353535353535,13.5353535353535)
--(axis cs:-13.3333333333333,13.3333333333333)
--(axis cs:-13.1313131313131,13.1313131313131)
--(axis cs:-12.9292929292929,12.9292929292929)
--(axis cs:-12.7272727272727,12.7272727272727)
--(axis cs:-12.5252525252525,12.5252525252525)
--(axis cs:-12.3232323232323,12.3232323232323)
--(axis cs:-12.1212121212121,12.1212121212121)
--(axis cs:-11.9191919191919,11.9191919191919)
--(axis cs:-11.7171717171717,11.7171717171717)
--(axis cs:-11.5151515151515,11.5151515151515)
--(axis cs:-11.3131313131313,11.3131313131313)
--(axis cs:-11.1111111111111,11.1111111111111)
--(axis cs:-10.9090909090909,10.9090909090909)
--(axis cs:-10.7070707070707,10.7070707070707)
--(axis cs:-10.5050505050505,10.5050505050505)
--(axis cs:-10.3030303030303,10.3030303030303)
--(axis cs:-10.1010101010101,10.1010101010101)
--(axis cs:-9.8989898989899,9.8989898989899)
--(axis cs:-9.6969696969697,9.6969696969697)
--(axis cs:-9.49494949494949,9.49494949494949)
--(axis cs:-9.29292929292929,9.29292929292929)
--(axis cs:-9.09090909090909,9.09090909090909)
--(axis cs:-8.88888888888889,8.88888888888889)
--(axis cs:-8.68686868686869,8.68686868686869)
--(axis cs:-8.48484848484848,8.48484848484848)
--(axis cs:-8.28282828282828,8.28282828282828)
--(axis cs:-8.08080808080808,8.08080808080808)
--(axis cs:-7.87878787878788,7.87878787878788)
--(axis cs:-7.67676767676768,7.67676767676768)
--(axis cs:-7.47474747474748,7.47474747474748)
--(axis cs:-7.27272727272727,7.27272727272727)
--(axis cs:-7.07070707070707,7.07070707070707)
--(axis cs:-6.86868686868687,6.86868686868687)
--(axis cs:-6.66666666666667,6.66666666666667)
--(axis cs:-6.46464646464647,6.46464646464647)
--(axis cs:-6.26262626262626,6.26262626262626)
--(axis cs:-6.06060606060606,6.06060606060606)
--(axis cs:-5.85858585858586,5.85858585858586)
--(axis cs:-5.65656565656566,5.65656565656566)
--(axis cs:-5.45454545454546,5.45454545454546)
--(axis cs:-5.25252525252525,5.25252525252525)
--(axis cs:-5.05050505050505,5.05050505050505)
--(axis cs:-4.84848484848485,4.84848484848485)
--(axis cs:-4.64646464646465,4.64646464646465)
--(axis cs:-4.44444444444444,4.44444444444444)
--(axis cs:-4.24242424242424,4.24242424242424)
--(axis cs:-4.04040404040404,4.04040404040404)
--(axis cs:-3.83838383838384,3.83838383838384)
--(axis cs:-3.63636363636364,3.63636363636364)
--(axis cs:-3.43434343434344,3.43434343434344)
--(axis cs:-3.23232323232323,3.23232323232323)
--(axis cs:-3.03030303030303,3.03030303030303)
--(axis cs:-2.82828282828283,2.82828282828283)
--(axis cs:-2.62626262626263,2.62626262626263)
--(axis cs:-2.42424242424243,2.42424242424243)
--(axis cs:-2.22222222222222,2.22222222222222)
--(axis cs:-2.02020202020202,2.02020202020202)
--(axis cs:-1.81818181818182,1.81818181818182)
--(axis cs:-1.61616161616162,1.61616161616162)
--(axis cs:-1.41414141414142,1.41414141414142)
--(axis cs:-1.21212121212121,1.21212121212121)
--(axis cs:-1.01010101010101,1.01010101010101)
--(axis cs:-0.80808080808081,0.80808080808081)
--(axis cs:-0.606060606060606,0.606060606060606)
--(axis cs:-0.404040404040405,0.404040404040405)
--(axis cs:-0.202020202020201,0.202020202020201)
--(axis cs:0,0)
--(axis cs:0,0)
--(axis cs:0,0)
--(axis cs:-0.202020202020201,-0.202020202020201)
--(axis cs:-0.404040404040405,-0.404040404040405)
--(axis cs:-0.606060606060606,-0.606060606060606)
--(axis cs:-0.80808080808081,-0.80808080808081)
--(axis cs:-1.01010101010101,-1.01010101010101)
--(axis cs:-1.21212121212121,-1.21212121212121)
--(axis cs:-1.41414141414142,-1.41414141414142)
--(axis cs:-1.61616161616162,-1.61616161616162)
--(axis cs:-1.81818181818182,-1.81818181818182)
--(axis cs:-2.02020202020202,-2.02020202020202)
--(axis cs:-2.22222222222222,-2.22222222222222)
--(axis cs:-2.42424242424243,-2.42424242424243)
--(axis cs:-2.62626262626263,-2.62626262626263)
--(axis cs:-2.82828282828283,-2.82828282828283)
--(axis cs:-3.03030303030303,-3.03030303030303)
--(axis cs:-3.23232323232323,-3.23232323232323)
--(axis cs:-3.43434343434344,-3.43434343434344)
--(axis cs:-3.63636363636364,-3.63636363636364)
--(axis cs:-3.83838383838384,-3.83838383838384)
--(axis cs:-4.04040404040404,-4.04040404040404)
--(axis cs:-4.24242424242424,-4.24242424242424)
--(axis cs:-4.44444444444444,-4.44444444444444)
--(axis cs:-4.64646464646465,-4.64646464646465)
--(axis cs:-4.84848484848485,-4.84848484848485)
--(axis cs:-5.05050505050505,-5.05050505050505)
--(axis cs:-5.25252525252525,-5.25252525252525)
--(axis cs:-5.45454545454546,-5.45454545454546)
--(axis cs:-5.65656565656566,-5.65656565656566)
--(axis cs:-5.85858585858586,-5.85858585858586)
--(axis cs:-6.06060606060606,-6.06060606060606)
--(axis cs:-6.26262626262626,-6.26262626262626)
--(axis cs:-6.46464646464647,-6.46464646464647)
--(axis cs:-6.66666666666667,-6.66666666666667)
--(axis cs:-6.86868686868687,-6.86868686868687)
--(axis cs:-7.07070707070707,-7.07070707070707)
--(axis cs:-7.27272727272727,-7.27272727272727)
--(axis cs:-7.47474747474748,-7.47474747474748)
--(axis cs:-7.67676767676768,-7.67676767676768)
--(axis cs:-7.87878787878788,-7.87878787878788)
--(axis cs:-8.08080808080808,-8.08080808080808)
--(axis cs:-8.28282828282828,-8.28282828282828)
--(axis cs:-8.48484848484848,-8.48484848484848)
--(axis cs:-8.68686868686869,-8.68686868686869)
--(axis cs:-8.88888888888889,-8.88888888888889)
--(axis cs:-9.09090909090909,-9.09090909090909)
--(axis cs:-9.29292929292929,-9.29292929292929)
--(axis cs:-9.49494949494949,-9.49494949494949)
--(axis cs:-9.6969696969697,-9.6969696969697)
--(axis cs:-9.8989898989899,-9.8989898989899)
--(axis cs:-10.1010101010101,-10.1010101010101)
--(axis cs:-10.3030303030303,-10.3030303030303)
--(axis cs:-10.5050505050505,-10.5050505050505)
--(axis cs:-10.7070707070707,-10.7070707070707)
--(axis cs:-10.9090909090909,-10.9090909090909)
--(axis cs:-11.1111111111111,-11.1111111111111)
--(axis cs:-11.3131313131313,-11.3131313131313)
--(axis cs:-11.5151515151515,-11.5151515151515)
--(axis cs:-11.7171717171717,-11.7171717171717)
--(axis cs:-11.9191919191919,-11.9191919191919)
--(axis cs:-12.1212121212121,-12.1212121212121)
--(axis cs:-12.3232323232323,-12.3232323232323)
--(axis cs:-12.5252525252525,-12.5252525252525)
--(axis cs:-12.7272727272727,-12.7272727272727)
--(axis cs:-12.9292929292929,-12.9292929292929)
--(axis cs:-13.1313131313131,-13.1313131313131)
--(axis cs:-13.3333333333333,-13.3333333333333)
--(axis cs:-13.5353535353535,-13.5353535353535)
--(axis cs:-13.7373737373737,-13.7373737373737)
--(axis cs:-13.9393939393939,-13.9393939393939)
--(axis cs:-14.1414141414141,-14.1414141414141)
--(axis cs:-14.3434343434343,-14.3434343434343)
--(axis cs:-14.5454545454545,-14.5454545454545)
--(axis cs:-14.7474747474747,-14.7474747474747)
--(axis cs:-14.9494949494949,-14.9494949494949)
--(axis cs:-15.1515151515152,-15.1515151515152)
--(axis cs:-15.3535353535354,-15.3535353535354)
--(axis cs:-15.5555555555556,-15.5555555555556)
--(axis cs:-15.7575757575758,-15.7575757575758)
--(axis cs:-15.959595959596,-15.959595959596)
--(axis cs:-16.1616161616162,-16.1616161616162)
--(axis cs:-16.3636363636364,-16.3636363636364)
--(axis cs:-16.5656565656566,-16.5656565656566)
--(axis cs:-16.7676767676768,-16.7676767676768)
--(axis cs:-16.969696969697,-16.969696969697)
--(axis cs:-17.1717171717172,-17.1717171717172)
--(axis cs:-17.3737373737374,-17.3737373737374)
--(axis cs:-17.5757575757576,-17.5757575757576)
--(axis cs:-17.7777777777778,-17.7777777777778)
--(axis cs:-17.979797979798,-17.979797979798)
--(axis cs:-18.1818181818182,-18.1818181818182)
--(axis cs:-18.3838383838384,-18.3838383838384)
--(axis cs:-18.5858585858586,-18.5858585858586)
--(axis cs:-18.7878787878788,-18.7878787878788)
--(axis cs:-18.989898989899,-18.989898989899)
--(axis cs:-19.1919191919192,-19.1919191919192)
--(axis cs:-19.3939393939394,-19.3939393939394)
--(axis cs:-19.5959595959596,-19.5959595959596)
--(axis cs:-19.7979797979798,-19.7979797979798)
--(axis cs:-20,-20)
--cycle;

\path [draw=none, fill=red, fill opacity=0.3, very thin]
(axis cs:5.885,0.1257)
--(axis cs:-0.1265,5.8549)
--(axis cs:-5.855,-0.1257)
--(axis cs:0.1265,-5.8549)
--cycle;
\addplot [thick, steelblue52138189, mark=*, mark size=2, mark options={solid}, only marks, forget plot]
table {%
-11 -18
};
\addplot [steelblue52138189, forget plot]
table {%
-11 -18
-8.9654 -10.9547
-5.72264031 -6.88564013
-2.537603913309 -4.335059518802
-0.024529932711375 -2.62455563250456
-0.378055228485392 -1.44119534914221
-0.418788308210596 -0.826919036925514
-0.357223447497051 -0.494136512857354
-0.275426510494952 -0.305701174008036
-0.201921326505686 -0.194427152444756
-0.232951701703781 -0.0409943871955294
-0.144635443205647 -0.0413412354218682
-0.100635015750828 -0.0261102400009674
-0.0682101059369273 -0.0180927042389798
-0.0465016736301783 -0.012274190495198
-0.0316607587451773 -0.00836606304138731
-0.021562371541953 -0.00569626820083243
-0.0146838714401479 -0.00387934186361246
-0.00999971720890057 -0.0026418001135835
-0.00680973757192909 -0.00179905225202304
-0.00463735421865751 -0.00122513349308485
-0.00315796313955616 -0.000834296183024798
-0.0021505069427537 -0.00056813825937129
-0.00146444025911026 -0.000386887635019868
-0.000997239616179268 -0.000263458801883882
-0.000679085516577723 -0.000179406286842897
-0.000462430558001448 -0.00012216863311408
-0.000314894970264269 -8.31914920556897e-05
-0.000214428295657233 -5.66493959546303e-05
-0.0001460143918706 -3.85752592716627e-05
-9.94274872292851e-05 -2.6267555200983e-05
};
\addplot [thick, steelblue52138189, mark=x, mark size=2, mark options={solid}, only marks, forget plot]
table {%
-9.94274872292851e-05 -2.6267555200983e-05
};
\addplot [thick, firebrick166640, mark=*, mark size=2, mark options={solid}, only marks, forget plot]
table {%
2 -19
};
\addplot [firebrick166640, forget plot]
table {%
2 -19
-1.0187 -10.1941
-0.47770253 -5.67275074
0.101496511197928 -3.147600062061
0.0744694250831489 -1.71710264305224
-0.189295610007559 -0.934457851156723
-0.235067464878968 -0.531315756128562
-0.207311754170588 -0.314314497693179
-0.162025302344464 -0.192547803251903
-0.119478181452376 -0.121336325006385
-0.0853212720532494 -0.0781412886815004
-0.0954246987924298 -0.0170004188726247
-0.0591610917663742 -0.0168419552036239
-0.0409375637571131 -0.0106361819982201
-0.0276401095467556 -0.00732862907674207
-0.0187618436230562 -0.00495274947097998
-0.0127208113060433 -0.00336126342182207
-0.00862722975432617 -0.00227912588365792
-0.00585075085488605 -0.00154571112365088
-0.0039679364351301 -0.00104827957813018
-0.00269106625749184 -0.000710947855266868
-0.00182512285291495 -0.000482175629543246
-0.00123784881585172 -0.000327024886214093
-0.000839558624087747 -0.000221801365286684
-0.000569432717922839 -0.000150437326549629
-0.000386226297565634 -0.000102036377184657
-0.000261968674034489 -6.92089964070569e-05
-0.000177690841066666 -4.69437990082389e-05
-0.000120528276726646 -3.18421318935871e-05
-8.17562828238121e-05 -2.15990339487155e-05
-5.54576772934144e-05 -1.46512563096755e-05
};
\addplot [thick, firebrick166640, mark=x, mark size=2, mark options={solid}, only marks, forget plot]
table {%
-5.54576772934144e-05 -1.46512563096755e-05
};
\addplot [thick, slategray122104166, mark=*, mark size=2, mark options={solid}, only marks, forget plot]
table {%
15 19
};
\addplot [slategray122104166, forget plot]
table {%
15 19
11.7555 11.8176
7.707449 7.58729729
7.39741942562 1.547020042362
2.93612937557486 1.30696332397565
2.35540605890467 0.538977671535833
1.54415626045415 0.417502616541035
1.05537964558716 0.275982296921904
0.714838317376402 0.188468959340924
0.48515387881046 0.127762230493865
0.32913791494801 0.0867323286771662
0.223139526127143 0.0590245103166421
0.151456026885109 0.0399957838301924
0.102750492801751 0.0271489172485288
0.0697180942497135 0.0184180345276851
0.04730253403925 0.0124968828525409
0.0320941195393218 0.00847886257614763
0.0217751890908341 0.00575274855865172
0.0147739139104646 0.0039030910927887
0.0100236499503875 0.00264812875799505
0.00680069061977899 0.00179666126150654
0.00461399270454311 0.0012189617952359
0.00313038347244453 0.000827009943740233
0.00212380637109671 0.0005610842901747
0.00144088401434153 0.000380664355889273
0.000977551937113274 0.00025825755216878
0.000663204338696486 0.000175210669219074
0.00044993685073151 0.000118867944799777
0.00030524773061808 8.0642806505756e-05
0.000207085639822549 5.47095539367891e-05
0.000140489598569576 3.7115674834267e-05
};
\addplot [thick, slategray122104166, mark=x, mark size=2, mark options={solid}, only marks, forget plot]
table {%
0.000140489598569576 3.7115674834267e-05
};
\addplot [thick, steelblue52138189, mark=*, mark size=2, mark options={solid}, only marks, forget plot]
table {%
-11 -18
};
\addplot [steelblue52138189, dash pattern=on 3.7pt off 1.6pt, forget plot]
table {%
-11 -18
-8.96540000824382 -10.9547
-5.72265314734765 -6.88564013080377
-2.53780071504365 -4.33506201577826
-0.263676307386041 -2.62461595420948
-0.517666084626102 -1.46464555714077
-0.503544306261984 -0.853545097173046
-0.410398590937705 -0.517144032162476
-0.30966285952383 -0.323616134776432
-0.224394153909606 -0.207673794255495
-0.253531169158711 -0.0451261585408006
-0.157941305448319 -0.0449452053542352
-0.109805759680001 -0.0285351378429113
-0.0745004570928327 -0.0197520089823309
-0.0508117736665566 -0.0134135659303858
-0.0346156648607191 -0.00914656334666747
-0.0235875778446493 -0.0062313327135151
-0.0160718837914416 -0.00424603283245593
-0.0109509318256436 -0.00289310059494423
-0.00746156963344074 -0.00197125834159431
-0.00508399953176897 -0.00134313187971125
-0.00346399095713763 -0.000915145074775144
-0.00236017411198698 -0.00062352982328176
-0.00160807910198104 -0.000424835302487732
-0.00109563750449887 -0.000289454349362507
-0.000746487480056718 -0.000197213080973794
-0.000508597665781685 -0.00013436543186589
-0.00034651531420374 -9.15452094535901e-05
-0.0002360841215856 -6.23706066464549e-05
-0.000160844953063519 -4.24933164976091e-05
-0.000109583319258233 -2.89506047868299e-05
};
\addplot [thick, steelblue52138189, mark=x, mark size=2, mark options={solid}, only marks, forget plot]
table {%
-0.000109583319258233 -2.89506047868299e-05
};
\addplot [thick, firebrick166640, mark=*, mark size=2, mark options={solid}, only marks, forget plot]
table {%
2 -19
};
\addplot [firebrick166640, dash pattern=on 3.7pt off 1.6pt, forget plot]
table {%
2 -19
-1.02508407836211 -10.1941
-0.571411515363787 -5.67337318764031
-0.317615098975712 -3.16748403179802
-0.613473054057463 -1.76732590673121
-0.59348528633659 -1.02899936085735
-0.481638470188099 -0.622324328996891
-0.362298350450482 -0.388100485173806
-0.261731605670694 -0.248000671345233
-0.297968410382758 -0.0529263670756659
-0.184495165018715 -0.0528286763331789
-0.127970048853908 -0.0331644099471962
-0.0863850832274469 -0.0229227708550928
-0.0586802703956542 -0.0154868952273723
-0.0398038622507905 -0.0105181491833517
-0.0270090394611916 -0.0071350874679833
-0.0183259763974369 -0.00484156258109962
-0.0124348232978874 -0.00328512288866818
-0.0084375715654831 -0.00222910706245026
-0.00572536018163158 -0.00151257160174585
-0.00388503757177053 -0.00102638058073959
-0.00263630013996601 -0.000696479020973914
-0.00178896450712718 -0.000472623067600472
-0.00121399218160924 -0.000320722242121171
-0.000823829604299487 -0.000217645946929734
-0.000559070123246348 -0.00014769965260681
-0.000379404611631379 -0.000100234169215122
-0.000257481721376061 -6.80235970751956e-05
-0.000174742170324219 -4.61647953985515e-05
-0.000118592333735388 -3.13306788657346e-05
-8.04865143116423e-05 -2.1263576265791e-05
};
\addplot [thick, firebrick166640, mark=x, mark size=2, mark options={solid}, only marks, forget plot]
table {%
-8.04865143116423e-05 -2.1263576265791e-05
};
\addplot [thick, slategray122104166, mark=*, mark size=2, mark options={solid}, only marks, forget plot]
table {%
15 19
};
\addplot [slategray122104166, dash pattern=on 3.7pt off 1.6pt, forget plot]
table {%
15 19
11.8571465304943 11.8176
13.0976552973823 2.38383376602062
7.25705489548536 2.30435322040969
3.37879739350691 1.30701610648456
2.57678441730835 0.614763994046174
1.7060791813162 0.457765696139576
1.16358768258345 0.304930614784226
0.788538551462218 0.207868716744099
0.535140529296053 0.140960792933241
0.363070107341771 0.0956832642422503
0.246157096710526 0.0651100778627331
0.167080406224887 0.0441227050221685
0.113353389987428 0.0299501972897198
0.0769136912750098 0.0203190126216961
0.0521856391891517 0.0137869414157908
0.035407858335745 0.00935431303951159
0.0240238954275821 0.00634682983630898
0.0162998606154415 0.00430622802033177
0.0110591215465009 0.00292168798722025
0.00750332459103922 0.00198228877007137
0.0050907663064595 0.00134491968911476
0.00345389318393872 0.000912477347630165
0.00234331645191952 0.000619076233098384
0.00158982478754397 0.000420012730176175
0.00107860853799347 0.000284955499748405
0.000731770159008208 0.000193324940434533
0.000496457153551412 0.000131158053464292
0.000336810072184798 8.89812003692589e-05
0.000228499140978435 6.0366745316487e-05
0.00015501733284477 4.09537288035771e-05
};
\addplot [thick, slategray122104166, mark=x, mark size=2, mark options={solid}, only marks, forget plot]
table {%
0.00015501733284477 4.09537288035771e-05
};
\addplot [thick, black]
table {%
-100 100
};
\addlegendentry{SwA}
\addplot [thick, black, dash pattern=on 7.4pt off 3.2pt]
table {%
-100 100
};
\addlegendentry{RCS}
\draw (axis cs:17,0.5) node[
  anchor=base west,
  text=black,
  rotate=0.0
]{$P_1$};
\draw (axis cs:-3.5,-19) node[
  anchor=base west,
  text=black,
  rotate=0.0
]{$P_2$};
\draw (axis cs:-19,0.5) node[
  anchor=base west,
  text=black,
  rotate=0.0
]{$P_3$};
\draw (axis cs:-3.5,18) node[
  anchor=base west,
  text=black,
  rotate=0.0
]{$P_4$};
\draw (axis cs:-3.5,2) node[
  anchor=base west,
  text=black,
  rotate=0.0
]{$\mathcal{X}_{T}$};
\end{axis}

\end{tikzpicture}

%% file: media/tikz/ex_1_cost.tex
\begin{tikzpicture}
	
\newcommand{\myH}{\axisdefaultheight}
\newcommand{\myW}{1.5*\axisdefaultwidth}

\definecolor{darkgray178}{RGB}{178,178,178}
\definecolor{firebrick166640}{RGB}{166,6,40}
\definecolor{lightgray204}{RGB}{204,204,204}
\definecolor{silver188}{RGB}{188,188,188}
\definecolor{slategray122104166}{RGB}{122,104,166}
\definecolor{whitesmoke238}{RGB}{238,238,238}

\begin{axis}[
height=\myH,
width=\myW,
axis background/.style={fill=whitesmoke238},
axis line style={silver188},
legend cell align={left},
legend columns=2,
legend style={
  fill opacity=0.8,
  draw opacity=1,
  text opacity=1,
  at={(0.03,0.97)},
  anchor=north west,
  draw=lightgray204,
  fill=whitesmoke238
},
tick pos=left,
tick scale binop=\times,
x grid style={darkgray178},
xmajorgrids,
xmin=-4.39, xmax=105.39,
xtick style={color=black},
xtick={-20,0,20,40,60,80,100,120},
xticklabels={
  \(\displaystyle {\ensuremath{-}20}\),
  \(\displaystyle {0}\),
  \(\displaystyle {20}\),
  \(\displaystyle {40}\),
  \(\displaystyle {60}\),
  \(\displaystyle {80}\),
  \(\displaystyle {100}\),
  \(\displaystyle {120}\)
},
y grid style={darkgray178},
ylabel={\(\displaystyle J/J_{MLD}\)},
xlabel={Different initial conditions},
ymajorgrids,
ymin=0.988113040838246, ymax=1.18,
ytick style={color=black},
ytick={1,1.08,1.16},
yticklabels={
  \(\displaystyle {1.00}\),
  \(\displaystyle {1.08}\),
  \(\displaystyle {1.16}\)
}
]
\draw[draw=none,fill=firebrick166640,very thin] (axis cs:0.6,1) rectangle (axis cs:1.4,1.03471772053499);
\draw[draw=none,fill=firebrick166640,very thin] (axis cs:1.6,1) rectangle (axis cs:2.4,1.00935919293208);
\draw[draw=none,fill=firebrick166640,very thin] (axis cs:2.6,1) rectangle (axis cs:3.4,1.04816168896399);
\draw[draw=none,fill=firebrick166640,very thin] (axis cs:3.6,1) rectangle (axis cs:4.4,1.01547597298183);
\draw[draw=none,fill=firebrick166640,very thin] (axis cs:4.6,1) rectangle (axis cs:5.4,1.01120055177286);
\draw[draw=none,fill=firebrick166640,very thin] (axis cs:5.6,1) rectangle (axis cs:6.4,1.00495092949063);
\draw[draw=none,fill=firebrick166640,very thin] (axis cs:6.6,1) rectangle (axis cs:7.4,1.07944450840356);
\draw[draw=none,fill=firebrick166640,very thin] (axis cs:7.6,1) rectangle (axis cs:8.4,1.03632259046219);
\draw[draw=none,fill=firebrick166640,very thin] (axis cs:8.6,1) rectangle (axis cs:9.4,1.03636437636497);
\draw[draw=none,fill=firebrick166640,very thin] (axis cs:9.6,1) rectangle (axis cs:10.4,1.03711064412632);
\draw[draw=none,fill=firebrick166640,very thin] (axis cs:10.6,1) rectangle (axis cs:11.4,1.05319489784405);
\draw[draw=none,fill=firebrick166640,very thin] (axis cs:11.6,1) rectangle (axis cs:12.4,1.00958975876181);
\draw[draw=none,fill=firebrick166640,very thin] (axis cs:12.6,1) rectangle (axis cs:13.4,1.0138454972309);
\draw[draw=none,fill=firebrick166640,very thin] (axis cs:13.6,1) rectangle (axis cs:14.4,1.00936386930264);
\draw[draw=none,fill=firebrick166640,very thin] (axis cs:14.6,1) rectangle (axis cs:15.4,1.01447580662424);
\draw[draw=none,fill=firebrick166640,very thin] (axis cs:15.6,1) rectangle (axis cs:16.4,1.02110304772308);
\draw[draw=none,fill=firebrick166640,very thin] (axis cs:16.6,1) rectangle (axis cs:17.4,1.01845239860417);
\draw[draw=none,fill=firebrick166640,very thin] (axis cs:17.6,1) rectangle (axis cs:18.4,1.12131296033961);
\draw[draw=none,fill=firebrick166640,very thin] (axis cs:18.6,1) rectangle (axis cs:19.4,1.00646887480789);
\draw[draw=none,fill=firebrick166640,very thin] (axis cs:19.6,1) rectangle (axis cs:20.4,1.00516789313985);
\draw[draw=none,fill=firebrick166640,very thin] (axis cs:20.6,1) rectangle (axis cs:21.4,1.00290351233617);
\draw[draw=none,fill=firebrick166640,very thin] (axis cs:21.6,1) rectangle (axis cs:22.4,1.01248206138633);
\draw[draw=none,fill=firebrick166640,very thin] (axis cs:22.6,1) rectangle (axis cs:23.4,1.01228255503408);
\draw[draw=none,fill=firebrick166640,very thin] (axis cs:23.6,1) rectangle (axis cs:24.4,1.00604666608194);
\draw[draw=none,fill=firebrick166640,very thin] (axis cs:24.6,1) rectangle (axis cs:25.4,1.00493274772784);
\draw[draw=none,fill=firebrick166640,very thin] (axis cs:25.6,1) rectangle (axis cs:26.4,1.01507601734014);
\draw[draw=none,fill=firebrick166640,very thin] (axis cs:26.6,1) rectangle (axis cs:27.4,1.07297162710117);
\draw[draw=none,fill=firebrick166640,very thin] (axis cs:27.6,1) rectangle (axis cs:28.4,1.01078602647989);
\draw[draw=none,fill=firebrick166640,very thin] (axis cs:28.6,1) rectangle (axis cs:29.4,1.00836966592526);
\draw[draw=none,fill=firebrick166640,very thin] (axis cs:29.6,1) rectangle (axis cs:30.4,1.01727158745667);
\draw[draw=none,fill=firebrick166640,very thin] (axis cs:30.6,1) rectangle (axis cs:31.4,1.03859974233274);
\draw[draw=none,fill=firebrick166640,very thin] (axis cs:31.6,1) rectangle (axis cs:32.4,1.01082354386311);
\draw[draw=none,fill=firebrick166640,very thin] (axis cs:32.6,1) rectangle (axis cs:33.4,1.00830531227712);
\draw[draw=none,fill=firebrick166640,very thin] (axis cs:33.6,1) rectangle (axis cs:34.4,1.01763543466317);
\draw[draw=none,fill=firebrick166640,very thin] (axis cs:34.6,1) rectangle (axis cs:35.4,1.00275475859105);
\draw[draw=none,fill=firebrick166640,very thin] (axis cs:35.6,1) rectangle (axis cs:36.4,1.0062411466818);
\draw[draw=none,fill=firebrick166640,very thin] (axis cs:36.6,1) rectangle (axis cs:37.4,1.00643758161871);
\draw[draw=none,fill=firebrick166640,very thin] (axis cs:37.6,1) rectangle (axis cs:38.4,1.02522516666627);
\draw[draw=none,fill=firebrick166640,very thin] (axis cs:38.6,1) rectangle (axis cs:39.4,1.0067438054408);
\draw[draw=none,fill=firebrick166640,very thin] (axis cs:39.6,1) rectangle (axis cs:40.4,1.00882884698653);
\draw[draw=none,fill=firebrick166640,very thin] (axis cs:40.6,1) rectangle (axis cs:41.4,1.01225870359835);
\draw[draw=none,fill=firebrick166640,very thin] (axis cs:41.6,1) rectangle (axis cs:42.4,1.01871271137302);
\draw[draw=none,fill=firebrick166640,very thin] (axis cs:42.6,1) rectangle (axis cs:43.4,1.01145783667223);
\draw[draw=none,fill=firebrick166640,very thin] (axis cs:43.6,1) rectangle (axis cs:44.4,0.9959928069781);
\draw[draw=none,fill=firebrick166640,very thin] (axis cs:44.6,1) rectangle (axis cs:45.4,1.0190036325114);
\draw[draw=none,fill=firebrick166640,very thin] (axis cs:45.6,1) rectangle (axis cs:46.4,1.00770517865485);
\draw[draw=none,fill=firebrick166640,very thin] (axis cs:46.6,1) rectangle (axis cs:47.4,1.00658617703783);
\draw[draw=none,fill=firebrick166640,very thin] (axis cs:47.6,1) rectangle (axis cs:48.4,1.00948422148253);
\draw[draw=none,fill=firebrick166640,very thin] (axis cs:48.6,1) rectangle (axis cs:49.4,1.02810275456639);
\draw[draw=none,fill=firebrick166640,very thin] (axis cs:49.6,1) rectangle (axis cs:50.4,1.09400891032357);
\draw[draw=none,fill=firebrick166640,very thin] (axis cs:50.6,1) rectangle (axis cs:51.4,1.00400638138338);
\draw[draw=none,fill=firebrick166640,very thin] (axis cs:51.6,1) rectangle (axis cs:52.4,1.03105408725475);
\draw[draw=none,fill=firebrick166640,very thin] (axis cs:52.6,1) rectangle (axis cs:53.4,1.03024498513613);
\draw[draw=none,fill=firebrick166640,very thin] (axis cs:53.6,1) rectangle (axis cs:54.4,1.03521129631179);
\draw[draw=none,fill=firebrick166640,very thin] (axis cs:54.6,1) rectangle (axis cs:55.4,1.05108319414781);
\draw[draw=none,fill=firebrick166640,very thin] (axis cs:55.6,1) rectangle (axis cs:56.4,1.15358812977517);
\draw[draw=none,fill=firebrick166640,very thin] (axis cs:56.6,1) rectangle (axis cs:57.4,1.04069031039561);
\draw[draw=none,fill=firebrick166640,very thin] (axis cs:57.6,1) rectangle (axis cs:58.4,1.01060498603695);
\draw[draw=none,fill=firebrick166640,very thin] (axis cs:58.6,1) rectangle (axis cs:59.4,1.00922259989345);
\draw[draw=none,fill=firebrick166640,very thin] (axis cs:59.6,1) rectangle (axis cs:60.4,1.02521770405954);
\draw[draw=none,fill=firebrick166640,very thin] (axis cs:60.6,1) rectangle (axis cs:61.4,1.00642765463947);
\draw[draw=none,fill=firebrick166640,very thin] (axis cs:61.6,1) rectangle (axis cs:62.4,1.0047401150415);
\draw[draw=none,fill=firebrick166640,very thin] (axis cs:62.6,1) rectangle (axis cs:63.4,1.01090571787172);
\draw[draw=none,fill=firebrick166640,very thin] (axis cs:63.6,1) rectangle (axis cs:64.4,1.00381644137682);
\draw[draw=none,fill=firebrick166640,very thin] (axis cs:64.6,1) rectangle (axis cs:65.4,1.00694626228475);
\draw[draw=none,fill=firebrick166640,very thin] (axis cs:65.6,1) rectangle (axis cs:66.4,1.00833014144404);
\draw[draw=none,fill=firebrick166640,very thin] (axis cs:66.6,1) rectangle (axis cs:67.4,1.03167148243266);
\draw[draw=none,fill=firebrick166640,very thin] (axis cs:67.6,1) rectangle (axis cs:68.4,1.00885811182463);
\draw[draw=none,fill=firebrick166640,very thin] (axis cs:68.6,1) rectangle (axis cs:69.4,1.00976438820212);
\draw[draw=none,fill=firebrick166640,very thin] (axis cs:69.6,1) rectangle (axis cs:70.4,1.04588172791209);
\draw[draw=none,fill=firebrick166640,very thin] (axis cs:70.6,1) rectangle (axis cs:71.4,1.01168970967327);
\draw[draw=none,fill=firebrick166640,very thin] (axis cs:71.6,1) rectangle (axis cs:72.4,1.00772223465947);
\draw[draw=none,fill=firebrick166640,very thin] (axis cs:72.6,1) rectangle (axis cs:73.4,1.00939426057416);
\draw[draw=none,fill=firebrick166640,very thin] (axis cs:73.6,1) rectangle (axis cs:74.4,1.01582322667382);
\draw[draw=none,fill=firebrick166640,very thin] (axis cs:74.6,1) rectangle (axis cs:75.4,1.0210524688057);
\draw[draw=none,fill=firebrick166640,very thin] (axis cs:75.6,1) rectangle (axis cs:76.4,1.01022579206917);
\draw[draw=none,fill=firebrick166640,very thin] (axis cs:76.6,1) rectangle (axis cs:77.4,1.09484444707538);
\draw[draw=none,fill=firebrick166640,very thin] (axis cs:77.6,1) rectangle (axis cs:78.4,1.02981938161399);
\draw[draw=none,fill=firebrick166640,very thin] (axis cs:78.6,1) rectangle (axis cs:79.4,1.00570930184814);
\draw[draw=none,fill=firebrick166640,very thin] (axis cs:79.6,1) rectangle (axis cs:80.4,1.01258915076924);
\draw[draw=none,fill=firebrick166640,very thin] (axis cs:80.6,1) rectangle (axis cs:81.4,1.01369561459894);
\draw[draw=none,fill=firebrick166640,very thin] (axis cs:81.6,1) rectangle (axis cs:82.4,1.00915242310911);
\draw[draw=none,fill=firebrick166640,very thin] (axis cs:82.6,1) rectangle (axis cs:83.4,1.00622496345846);
\draw[draw=none,fill=firebrick166640,very thin] (axis cs:83.6,1) rectangle (axis cs:84.4,1.02124929770939);
\draw[draw=none,fill=firebrick166640,very thin] (axis cs:84.6,1) rectangle (axis cs:85.4,1.01434288120488);
\draw[draw=none,fill=firebrick166640,very thin] (axis cs:85.6,1) rectangle (axis cs:86.4,1.01887571663366);
\draw[draw=none,fill=firebrick166640,very thin] (axis cs:86.6,1) rectangle (axis cs:87.4,1.01975879773189);
\draw[draw=none,fill=firebrick166640,very thin] (axis cs:87.6,1) rectangle (axis cs:88.4,1.01961160561845);
\draw[draw=none,fill=firebrick166640,very thin] (axis cs:88.6,1) rectangle (axis cs:89.4,1.01867725766476);
\draw[draw=none,fill=firebrick166640,very thin] (axis cs:89.6,1) rectangle (axis cs:90.4,1.0225241073196);
\draw[draw=none,fill=firebrick166640,very thin] (axis cs:90.6,1) rectangle (axis cs:91.4,1.03433968289135);
\draw[draw=none,fill=firebrick166640,very thin] (axis cs:91.6,1) rectangle (axis cs:92.4,1.01663283914256);
\draw[draw=none,fill=firebrick166640,very thin] (axis cs:92.6,1) rectangle (axis cs:93.4,1.05512931272734);
\draw[draw=none,fill=firebrick166640,very thin] (axis cs:93.6,1) rectangle (axis cs:94.4,1.00811065864766);
\draw[draw=none,fill=firebrick166640,very thin] (axis cs:94.6,1) rectangle (axis cs:95.4,1.0144039155399);
\draw[draw=none,fill=firebrick166640,very thin] (axis cs:95.6,1) rectangle (axis cs:96.4,1.01702466033748);
\draw[draw=none,fill=firebrick166640,very thin] (axis cs:96.6,1) rectangle (axis cs:97.4,1.0056888369622);
\draw[draw=none,fill=firebrick166640,very thin] (axis cs:97.6,1) rectangle (axis cs:98.4,1.02239134841419);
\draw[draw=none,fill=firebrick166640,very thin] (axis cs:98.6,1) rectangle (axis cs:99.4,1.01159219359397);
\draw[draw=none,fill=firebrick166640,very thin] (axis cs:99.6,1) rectangle (axis cs:100.4,1.02290942636087);
\addlegendimage{very thick, firebrick166640}
\addlegendentry{RC-Sets}

\draw[draw=none,fill=slategray122104166,very thin] (axis cs:0.6,1) rectangle (axis cs:1.4,1.00449230523496);
\addlegendimage{very thick, slategray122104166}
\addlegendentry{Sw-ADMM}
\draw[draw=none,fill=slategray122104166,very thin] (axis cs:1.6,1) rectangle (axis cs:2.4,1.00145142341558);
\draw[draw=none,fill=slategray122104166,very thin] (axis cs:2.6,1) rectangle (axis cs:3.4,1.02645829651932);
\draw[draw=none,fill=slategray122104166,very thin] (axis cs:3.6,1) rectangle (axis cs:4.4,1.01493536950045);
\draw[draw=none,fill=slategray122104166,very thin] (axis cs:4.6,1) rectangle (axis cs:5.4,1.00176347092092);
\draw[draw=none,fill=slategray122104166,very thin] (axis cs:5.6,1) rectangle (axis cs:6.4,1.00427314157566);
\draw[draw=none,fill=slategray122104166,very thin] (axis cs:6.6,1) rectangle (axis cs:7.4,1.01428343693066);
\draw[draw=none,fill=slategray122104166,very thin] (axis cs:7.6,1) rectangle (axis cs:8.4,1.02804427538474);
\draw[draw=none,fill=slategray122104166,very thin] (axis cs:8.6,1) rectangle (axis cs:9.4,1.01132950744366);
\draw[draw=none,fill=slategray122104166,very thin] (axis cs:9.6,1) rectangle (axis cs:10.4,1.03446444728569);
\draw[draw=none,fill=slategray122104166,very thin] (axis cs:10.6,1) rectangle (axis cs:11.4,1.05289249731807);
\draw[draw=none,fill=slategray122104166,very thin] (axis cs:11.6,1) rectangle (axis cs:12.4,1.00264592240343);
\draw[draw=none,fill=slategray122104166,very thin] (axis cs:12.6,1) rectangle (axis cs:13.4,1.00391436043914);
\draw[draw=none,fill=slategray122104166,very thin] (axis cs:13.6,1) rectangle (axis cs:14.4,1.00303955105748);
\draw[draw=none,fill=slategray122104166,very thin] (axis cs:14.6,1) rectangle (axis cs:15.4,1.00607233154618);
\draw[draw=none,fill=slategray122104166,very thin] (axis cs:15.6,1) rectangle (axis cs:16.4,1.01318324220312);
\draw[draw=none,fill=slategray122104166,very thin] (axis cs:16.6,1) rectangle (axis cs:17.4,1.00258890572962);
\draw[draw=none,fill=slategray122104166,very thin] (axis cs:17.6,1) rectangle (axis cs:18.4,1.00068055858915);
\draw[draw=none,fill=slategray122104166,very thin] (axis cs:18.6,1) rectangle (axis cs:19.4,1.00375373604607);
\draw[draw=none,fill=slategray122104166,very thin] (axis cs:19.6,1) rectangle (axis cs:20.4,1.00161844608593);
\draw[draw=none,fill=slategray122104166,very thin] (axis cs:20.6,1) rectangle (axis cs:21.4,1.00209577606134);
\draw[draw=none,fill=slategray122104166,very thin] (axis cs:21.6,1) rectangle (axis cs:22.4,1.00551716666973);
\draw[draw=none,fill=slategray122104166,very thin] (axis cs:22.6,1) rectangle (axis cs:23.4,1.01031438278237);
\draw[draw=none,fill=slategray122104166,very thin] (axis cs:23.6,1) rectangle (axis cs:24.4,1.00331256109111);
\draw[draw=none,fill=slategray122104166,very thin] (axis cs:24.6,1) rectangle (axis cs:25.4,1.00200738239499);
\draw[draw=none,fill=slategray122104166,very thin] (axis cs:25.6,1) rectangle (axis cs:26.4,1.00221539038807);
\draw[draw=none,fill=slategray122104166,very thin] (axis cs:26.6,1) rectangle (axis cs:27.4,1.00089864960469);
\draw[draw=none,fill=slategray122104166,very thin] (axis cs:27.6,1) rectangle (axis cs:28.4,1.00260405622077);
\draw[draw=none,fill=slategray122104166,very thin] (axis cs:28.6,1) rectangle (axis cs:29.4,1.00471229519612);
\draw[draw=none,fill=slategray122104166,very thin] (axis cs:29.6,1) rectangle (axis cs:30.4,1.01376744610704);
\draw[draw=none,fill=slategray122104166,very thin] (axis cs:30.6,1) rectangle (axis cs:31.4,1.03500747590659);
\draw[draw=none,fill=slategray122104166,very thin] (axis cs:31.6,1) rectangle (axis cs:32.4,1.00442303168013);
\draw[draw=none,fill=slategray122104166,very thin] (axis cs:32.6,1) rectangle (axis cs:33.4,1.00276580614382);
\draw[draw=none,fill=slategray122104166,very thin] (axis cs:33.6,1) rectangle (axis cs:34.4,1.0130433765599);
\draw[draw=none,fill=slategray122104166,very thin] (axis cs:34.6,1) rectangle (axis cs:35.4,1.00193685027209);
\draw[draw=none,fill=slategray122104166,very thin] (axis cs:35.6,1) rectangle (axis cs:36.4,1.00336926680209);
\draw[draw=none,fill=slategray122104166,very thin] (axis cs:36.6,1) rectangle (axis cs:37.4,1.00372180484182);
\draw[draw=none,fill=slategray122104166,very thin] (axis cs:37.6,1) rectangle (axis cs:38.4,1.00312599419411);
\draw[draw=none,fill=slategray122104166,very thin] (axis cs:38.6,1) rectangle (axis cs:39.4,1.00536432055734);
\draw[draw=none,fill=slategray122104166,very thin] (axis cs:39.6,1) rectangle (axis cs:40.4,1.00583364799325);
\draw[draw=none,fill=slategray122104166,very thin] (axis cs:40.6,1) rectangle (axis cs:41.4,1.00636732388674);
\draw[draw=none,fill=slategray122104166,very thin] (axis cs:41.6,1) rectangle (axis cs:42.4,1.01764733640148);
\draw[draw=none,fill=slategray122104166,very thin] (axis cs:42.6,1) rectangle (axis cs:43.4,1.00874515799061);
\draw[draw=none,fill=slategray122104166,very thin] (axis cs:43.6,1) rectangle (axis cs:44.4,1.00192413981083);
\draw[draw=none,fill=slategray122104166,very thin] (axis cs:44.6,1) rectangle (axis cs:45.4,1.01499948613335);
\draw[draw=none,fill=slategray122104166,very thin] (axis cs:45.6,1) rectangle (axis cs:46.4,1.00406260590974);
\draw[draw=none,fill=slategray122104166,very thin] (axis cs:46.6,1) rectangle (axis cs:47.4,1.00487057922253);
\draw[draw=none,fill=slategray122104166,very thin] (axis cs:47.6,1) rectangle (axis cs:48.4,1.00134249693525);
\draw[draw=none,fill=slategray122104166,very thin] (axis cs:48.6,1) rectangle (axis cs:49.4,1.00337524818534);
\draw[draw=none,fill=slategray122104166,very thin] (axis cs:49.6,1) rectangle (axis cs:50.4,1.02583888993305);
\draw[draw=none,fill=slategray122104166,very thin] (axis cs:50.6,1) rectangle (axis cs:51.4,1.00127670913137);
\draw[draw=none,fill=slategray122104166,very thin] (axis cs:51.6,1) rectangle (axis cs:52.4,1.03068926917883);
\draw[draw=none,fill=slategray122104166,very thin] (axis cs:52.6,1) rectangle (axis cs:53.4,1.01588229651279);
\draw[draw=none,fill=slategray122104166,very thin] (axis cs:53.6,1) rectangle (axis cs:54.4,1.01091633910132);
\draw[draw=none,fill=slategray122104166,very thin] (axis cs:54.6,1) rectangle (axis cs:55.4,1.00468106390214);
\draw[draw=none,fill=slategray122104166,very thin] (axis cs:55.6,1) rectangle (axis cs:56.4,1.00061793819833);
\draw[draw=none,fill=slategray122104166,very thin] (axis cs:56.6,1) rectangle (axis cs:57.4,1.02129606440349);
\draw[draw=none,fill=slategray122104166,very thin] (axis cs:57.6,1) rectangle (axis cs:58.4,1.01021830629838);
\draw[draw=none,fill=slategray122104166,very thin] (axis cs:58.6,1) rectangle (axis cs:59.4,1.00587297038904);
\draw[draw=none,fill=slategray122104166,very thin] (axis cs:59.6,1) rectangle (axis cs:60.4,1.00194572286088);
\draw[draw=none,fill=slategray122104166,very thin] (axis cs:60.6,1) rectangle (axis cs:61.4,1.0006414575104);
\draw[draw=none,fill=slategray122104166,very thin] (axis cs:61.6,1) rectangle (axis cs:62.4,1.00166749796333);
\draw[draw=none,fill=slategray122104166,very thin] (axis cs:62.6,1) rectangle (axis cs:63.4,1.00643040095101);
\draw[draw=none,fill=slategray122104166,very thin] (axis cs:63.6,1) rectangle (axis cs:64.4,1.00104515874012);
\draw[draw=none,fill=slategray122104166,very thin] (axis cs:64.6,1) rectangle (axis cs:65.4,1.00253129217063);
\draw[draw=none,fill=slategray122104166,very thin] (axis cs:65.6,1) rectangle (axis cs:66.4,1.00411854720571);
\draw[draw=none,fill=slategray122104166,very thin] (axis cs:66.6,1) rectangle (axis cs:67.4,1.0266870568425);
\draw[draw=none,fill=slategray122104166,very thin] (axis cs:67.6,1) rectangle (axis cs:68.4,1.00283070928422);
\draw[draw=none,fill=slategray122104166,very thin] (axis cs:68.6,1) rectangle (axis cs:69.4,1.00269276384646);
\draw[draw=none,fill=slategray122104166,very thin] (axis cs:69.6,1) rectangle (axis cs:70.4,1.03687415418774);
\draw[draw=none,fill=slategray122104166,very thin] (axis cs:70.6,1) rectangle (axis cs:71.4,1.00333090676935);
\draw[draw=none,fill=slategray122104166,very thin] (axis cs:71.6,1) rectangle (axis cs:72.4,1.00567748115721);
\draw[draw=none,fill=slategray122104166,very thin] (axis cs:72.6,1) rectangle (axis cs:73.4,1.00551900962653);
\draw[draw=none,fill=slategray122104166,very thin] (axis cs:73.6,1) rectangle (axis cs:74.4,1.00424170595976);
\draw[draw=none,fill=slategray122104166,very thin] (axis cs:74.6,1) rectangle (axis cs:75.4,1.00290757246131);
\draw[draw=none,fill=slategray122104166,very thin] (axis cs:75.6,1) rectangle (axis cs:76.4,1.00170176878584);
\draw[draw=none,fill=slategray122104166,very thin] (axis cs:76.6,1) rectangle (axis cs:77.4,1.00720705503139);
\draw[draw=none,fill=slategray122104166,very thin] (axis cs:77.6,1) rectangle (axis cs:78.4,1.02182366335809);
\draw[draw=none,fill=slategray122104166,very thin] (axis cs:78.6,1) rectangle (axis cs:79.4,1.00182608579777);
\draw[draw=none,fill=slategray122104166,very thin] (axis cs:79.6,1) rectangle (axis cs:80.4,1.01001506374962);
\draw[draw=none,fill=slategray122104166,very thin] (axis cs:80.6,1) rectangle (axis cs:81.4,1.0087996806311);
\draw[draw=none,fill=slategray122104166,very thin] (axis cs:81.6,1) rectangle (axis cs:82.4,1.00875322259764);
\draw[draw=none,fill=slategray122104166,very thin] (axis cs:82.6,1) rectangle (axis cs:83.4,1.00479803368453);
\draw[draw=none,fill=slategray122104166,very thin] (axis cs:83.6,1) rectangle (axis cs:84.4,1.01222825098976);
\draw[draw=none,fill=slategray122104166,very thin] (axis cs:84.6,1) rectangle (axis cs:85.4,1.00367870902315);
\draw[draw=none,fill=slategray122104166,very thin] (axis cs:85.6,1) rectangle (axis cs:86.4,1.01187042744106);
\draw[draw=none,fill=slategray122104166,very thin] (axis cs:86.6,1) rectangle (axis cs:87.4,1.01483541701309);
\draw[draw=none,fill=slategray122104166,very thin] (axis cs:87.6,1) rectangle (axis cs:88.4,1.00097889177086);
\draw[draw=none,fill=slategray122104166,very thin] (axis cs:88.6,1) rectangle (axis cs:89.4,1.00448642707476);
\draw[draw=none,fill=slategray122104166,very thin] (axis cs:89.6,1) rectangle (axis cs:90.4,1.00236976195675);
\draw[draw=none,fill=slategray122104166,very thin] (axis cs:90.6,1) rectangle (axis cs:91.4,1.03397251533821);
\draw[draw=none,fill=slategray122104166,very thin] (axis cs:91.6,1) rectangle (axis cs:92.4,1.01027890295455);
\draw[draw=none,fill=slategray122104166,very thin] (axis cs:92.6,1) rectangle (axis cs:93.4,1.00432066249548);
\draw[draw=none,fill=slategray122104166,very thin] (axis cs:93.6,1) rectangle (axis cs:94.4,1.00730567384411);
\draw[draw=none,fill=slategray122104166,very thin] (axis cs:94.6,1) rectangle (axis cs:95.4,1.00059011843783);
\draw[draw=none,fill=slategray122104166,very thin] (axis cs:95.6,1) rectangle (axis cs:96.4,1.00249578761476);
\draw[draw=none,fill=slategray122104166,very thin] (axis cs:96.6,1) rectangle (axis cs:97.4,1.00410728614765);
\draw[draw=none,fill=slategray122104166,very thin] (axis cs:97.6,1) rectangle (axis cs:98.4,1.00718420797648);
\draw[draw=none,fill=slategray122104166,very thin] (axis cs:98.6,1) rectangle (axis cs:99.4,1.00562013124722);
\draw[draw=none,fill=slategray122104166,very thin] (axis cs:99.6,1) rectangle (axis cs:100.4,1.00192858338813);
\end{axis}

\end{tikzpicture}

%% file: media/tikz/ex_1_time.tex
\begin{tikzpicture}
	
\newcommand{\myH}{\axisdefaultheight}
\newcommand{\myW}{0.6*\axisdefaultwidth}

\definecolor{darkgray178}{RGB}{178,178,178}
\definecolor{firebrick166640}{RGB}{166,6,40}
\definecolor{silver188}{RGB}{188,188,188}
\definecolor{whitesmoke238}{RGB}{238,238,238}

\begin{groupplot}[group style={group size=2 by 1, horizontal sep=0.1cm}]
\nextgroupplot[
height=\myH,
width=\myW,
axis background/.style={fill=whitesmoke238},
axis line style={silver188},
log basis y={10},
scaled y ticks=manual:{}{\pgfmathparse{#1}},
tick pos=left,
tick scale binop=\times,
title={Sw-ADMM},
x grid style={darkgray178},
xmajorgrids,
xmin=0.5, xmax=4.5,
xtick style={color=black},
xtick={1,2,3,4},
xticklabels={gurobi,{\raisebox{-0.5cm}{osqp}},qpoases,\raisebox{-0.5cm}{qrqp}},
y grid style={darkgray178},
ylabel={Time (s)},
ymajorgrids,
ymin=0.00139591318955465, ymax=0.311899173070285,
ymode=log,
ytick style={color=black},
ytick={0.0001,0.001,0.01,0.1,1,10},
yticklabels={
  \(\displaystyle {10^{-4}}\),
  \(\displaystyle {10^{-3}}\),
  \(\displaystyle {10^{-2}}\),
  \(\displaystyle {10^{-1}}\),
  \(\displaystyle {10^{0}}\),
  \(\displaystyle {10^{1}}\)
}
]
\addplot [black]
table {%
0.775 0.08476425
1.225 0.08476425
1.225 0.1156595
0.775 0.1156595
0.775 0.08476425
};
\addplot [black]
table {%
1 0.08476425
1 0.074862
};
\addplot [black]
table {%
1 0.1156595
1 0.161125
};
\addplot [black]
table {%
0.8875 0.074862
1.1125 0.074862
};
\addplot [black]
table {%
0.8875 0.161125
1.1125 0.161125
};
\addplot [black]
table {%
1.775 0.003999
2.225 0.003999
2.225 0.006026
1.775 0.006026
1.775 0.003999
};
\addplot [black]
table {%
2 0.003999
2 0.003083
};
\addplot [black]
table {%
2 0.006026
2 0.008829
};
\addplot [black]
table {%
1.8875 0.003083
2.1125 0.003083
};
\addplot [black]
table {%
1.8875 0.008829
2.1125 0.008829
};
\addplot [black, mark=o, mark size=1, mark options={solid,fill opacity=0}, only marks]
table {%
2 0.009468
2 0.010171
2 0.009928
2 0.009681
2 0.009332
2 0.013571
2 0.010038
2 0.010674
2 0.009824
2 0.016375
2 0.010781
2 0.009333
};
\addplot [black]
table {%
2.775 0.0053435
3.225 0.0053435
3.225 0.01149875
2.775 0.01149875
2.775 0.0053435
};
\addplot [black]
table {%
3 0.0053435
3 0.002317
};
\addplot [black]
table {%
3 0.01149875
3 0.020647
};
\addplot [black]
table {%
2.8875 0.002317
3.1125 0.002317
};
\addplot [black]
table {%
2.8875 0.020647
3.1125 0.020647
};
\addplot [black, mark=o, mark size=1, mark options={solid,fill opacity=0}, only marks]
table {%
3 0.021752
3 0.039624
};
\addplot [black]
table {%
3.775 0.00273025
4.225 0.00273025
4.225 0.004035
3.775 0.004035
3.775 0.00273025
};
\addplot [black]
table {%
4 0.00273025
4 0.001785
};
\addplot [black]
table {%
4 0.004035
4 0.00589
};
\addplot [black]
table {%
3.8875 0.001785
4.1125 0.001785
};
\addplot [black]
table {%
3.8875 0.00589
4.1125 0.00589
};
\addplot [black, mark=o, mark size=1, mark options={solid,fill opacity=0}, only marks]
table {%
4 0.006456
4 0.007118
4 0.006901
4 0.006034
4 0.006002
};
\addplot [firebrick166640]
table {%
0.775 0.092507
1.225 0.092507
};
\addplot [firebrick166640]
table {%
1.775 0.004781
2.225 0.004781
};
\addplot [firebrick166640]
table {%
2.775 0.0084735
3.225 0.0084735
};
\addplot [firebrick166640]
table {%
3.775 0.0033085
4.225 0.0033085
};

\nextgroupplot[
height=\myH,
width=\myW,
axis background/.style={fill=whitesmoke238},
axis line style={silver188},
log basis y={10},
scaled y ticks=manual:{}{\pgfmathparse{#1}},
tick pos=left,
tick scale binop=\times,
title={RC-Sets},
x grid style={darkgray178},
xmajorgrids,
xmin=0.5, xmax=3.5,
xtick style={color=black},
xtick={1,2,3},
xticklabels={cplex,mosek,gurobi},
y grid style={darkgray178},
ymajorgrids,
ymin=0.00139591318955465, ymax=0.311899173070285,
ymode=log,
ytick style={color=black},
yticklabels={}
]
\addplot [black]
table {%
0.85 0.06716955
1.15 0.06716955
1.15 0.08182665
0.85 0.08182665
0.85 0.06716955
};
\addplot [black]
table {%
1 0.06716955
1 0.0560936
};
\addplot [black]
table {%
1 0.08182665
1 0.1029527
};
\addplot [black]
table {%
0.925 0.0560936
1.075 0.0560936
};
\addplot [black]
table {%
0.925 0.1029527
1.075 0.1029527
};
\addplot [black, mark=o, mark size=1, mark options={solid,fill opacity=0}, only marks]
table {%
1 0.1100158
1 0.1219042
1 0.1299543
1 0.1083038
1 0.1082172
1 0.1260135
1 0.1048585
1 0.1230711
1 0.1178341
};
\addplot [black]
table {%
1.85 0.0341633
2.15 0.0341633
2.15 0.0497338
1.85 0.0497338
1.85 0.0341633
};
\addplot [black]
table {%
2 0.0341633
2 0.022084
};
\addplot [black]
table {%
2 0.0497338
2 0.072783
};
\addplot [black]
table {%
1.925 0.022084
2.075 0.022084
};
\addplot [black]
table {%
1.925 0.072783
2.075 0.072783
};
\addplot [black, mark=o, mark size=1, mark options={solid,fill opacity=0}, only marks]
table {%
2 0.1174396
2 0.1389995
2 0.1090815
2 0.1297832
2 0.094654
2 0.1192172
2 0.1354561
2 0.0969842
2 0.0905622
2 0.2439127
2 0.1154303
2 0.1040888
2 0.0806144
2 0.1370983
2 0.09039
2 0.1172262
2 0.1109376
2 0.1821959
2 0.073662
2 0.1402196
2 0.1124699
2 0.1118732
2 0.0764175
2 0.1102901
};
\addplot [black]
table {%
2.85 0.0118859
3.15 0.0118859
3.15 0.0148233
2.85 0.0148233
2.85 0.0118859
};
\addplot [black]
table {%
3 0.0118859
3 0.0107376
};
\addplot [black]
table {%
3 0.0148233
3 0.019152
};
\addplot [black]
table {%
2.925 0.0107376
3.075 0.0107376
};
\addplot [black]
table {%
2.925 0.019152
3.075 0.019152
};
\addplot [black, mark=o, mark size=1, mark options={solid,fill opacity=0}, only marks]
table {%
3 0.0198125
3 0.0220378
3 0.0208711
3 0.0334278
3 0.025432
3 0.0228379
3 0.0234078
3 0.023032
3 0.0237699
3 0.2357852
3 0.020405
3 0.0237513
3 0.0196736
3 0.0215784
3 0.0222165
3 0.0214435
3 0.0201438
3 0.0202859
3 0.0210419
3 0.0262732
3 0.0195017
};
\addplot [firebrick166640]
table {%
0.85 0.073148
1.15 0.073148
};
\addplot [firebrick166640]
table {%
1.85 0.0410984
2.15 0.0410984
};
\addplot [firebrick166640]
table {%
2.85 0.0127671
3.15 0.0127671
};
\end{groupplot}

\end{tikzpicture}

%% file: media/tikz/ex_1_unstab.tex
\begin{tikzpicture}

\definecolor{darkgray178}{RGB}{178,178,178}
\definecolor{firebrick166640}{RGB}{166,6,40}
\definecolor{gray}{RGB}{128,128,128}
\definecolor{lightgray204}{RGB}{204,204,204}
\definecolor{silver188}{RGB}{188,188,188}
\definecolor{slategray122104166}{RGB}{122,104,166}
\definecolor{steelblue52138189}{RGB}{52,138,189}
\definecolor{whitesmoke238}{RGB}{238,238,238}

\newcommand{\ms}{2}

\begin{axis}[
width=\axisdefaultwidth,
height=2*\axisdefaultheight,
axis background/.style={fill=whitesmoke238},
axis line style={silver188},
legend cell align={left},
legend style={
  fill opacity=0.8,
  draw opacity=1,
  text opacity=1,
  at={(0.03,0.97)},
  anchor=north west,
  draw=lightgray204,
  fill=whitesmoke238
},
tick pos=left,
tick scale binop=\times,
x grid style={darkgray178},
xmajorgrids,
xmin=-20, xmax=20,
xtick style={color=black},
xtick={-20,0,20},
xticklabels={
	\(\displaystyle {\ensuremath{-}20}\),
	\(\displaystyle {0}\),
	\(\displaystyle {20}\)
},
y grid style={darkgray178},
ymajorgrids,
ymin=-20, ymax=20,
ytick style={color=black},
ytick={-20,0,20},
yticklabels={
	\(\displaystyle {\ensuremath{-}20}\),
	\(\displaystyle {0}\),
	\(\displaystyle {20}\)
}
]
\path [draw=gray, fill=gray, opacity=0.3, very thin]
(axis cs:0,0)
--(axis cs:0,0)
--(axis cs:0.202020202020202,0.202020202020202)
--(axis cs:0.404040404040404,0.404040404040404)
--(axis cs:0.606060606060606,0.606060606060606)
--(axis cs:0.808080808080808,0.808080808080808)
--(axis cs:1.01010101010101,1.01010101010101)
--(axis cs:1.21212121212121,1.21212121212121)
--(axis cs:1.41414141414141,1.41414141414141)
--(axis cs:1.61616161616162,1.61616161616162)
--(axis cs:1.81818181818182,1.81818181818182)
--(axis cs:2.02020202020202,2.02020202020202)
--(axis cs:2.22222222222222,2.22222222222222)
--(axis cs:2.42424242424242,2.42424242424242)
--(axis cs:2.62626262626263,2.62626262626263)
--(axis cs:2.82828282828283,2.82828282828283)
--(axis cs:3.03030303030303,3.03030303030303)
--(axis cs:3.23232323232323,3.23232323232323)
--(axis cs:3.43434343434343,3.43434343434343)
--(axis cs:3.63636363636364,3.63636363636364)
--(axis cs:3.83838383838384,3.83838383838384)
--(axis cs:4.04040404040404,4.04040404040404)
--(axis cs:4.24242424242424,4.24242424242424)
--(axis cs:4.44444444444444,4.44444444444444)
--(axis cs:4.64646464646465,4.64646464646465)
--(axis cs:4.84848484848485,4.84848484848485)
--(axis cs:5.05050505050505,5.05050505050505)
--(axis cs:5.25252525252525,5.25252525252525)
--(axis cs:5.45454545454545,5.45454545454545)
--(axis cs:5.65656565656566,5.65656565656566)
--(axis cs:5.85858585858586,5.85858585858586)
--(axis cs:6.06060606060606,6.06060606060606)
--(axis cs:6.26262626262626,6.26262626262626)
--(axis cs:6.46464646464646,6.46464646464646)
--(axis cs:6.66666666666667,6.66666666666667)
--(axis cs:6.86868686868687,6.86868686868687)
--(axis cs:7.07070707070707,7.07070707070707)
--(axis cs:7.27272727272727,7.27272727272727)
--(axis cs:7.47474747474747,7.47474747474747)
--(axis cs:7.67676767676768,7.67676767676768)
--(axis cs:7.87878787878788,7.87878787878788)
--(axis cs:8.08080808080808,8.08080808080808)
--(axis cs:8.28282828282828,8.28282828282828)
--(axis cs:8.48484848484848,8.48484848484848)
--(axis cs:8.68686868686869,8.68686868686869)
--(axis cs:8.88888888888889,8.88888888888889)
--(axis cs:9.09090909090909,9.09090909090909)
--(axis cs:9.29292929292929,9.29292929292929)
--(axis cs:9.49494949494949,9.49494949494949)
--(axis cs:9.6969696969697,9.6969696969697)
--(axis cs:9.8989898989899,9.8989898989899)
--(axis cs:10.1010101010101,10.1010101010101)
--(axis cs:10.3030303030303,10.3030303030303)
--(axis cs:10.5050505050505,10.5050505050505)
--(axis cs:10.7070707070707,10.7070707070707)
--(axis cs:10.9090909090909,10.9090909090909)
--(axis cs:11.1111111111111,11.1111111111111)
--(axis cs:11.3131313131313,11.3131313131313)
--(axis cs:11.5151515151515,11.5151515151515)
--(axis cs:11.7171717171717,11.7171717171717)
--(axis cs:11.9191919191919,11.9191919191919)
--(axis cs:12.1212121212121,12.1212121212121)
--(axis cs:12.3232323232323,12.3232323232323)
--(axis cs:12.5252525252525,12.5252525252525)
--(axis cs:12.7272727272727,12.7272727272727)
--(axis cs:12.9292929292929,12.9292929292929)
--(axis cs:13.1313131313131,13.1313131313131)
--(axis cs:13.3333333333333,13.3333333333333)
--(axis cs:13.5353535353535,13.5353535353535)
--(axis cs:13.7373737373737,13.7373737373737)
--(axis cs:13.9393939393939,13.9393939393939)
--(axis cs:14.1414141414141,14.1414141414141)
--(axis cs:14.3434343434343,14.3434343434343)
--(axis cs:14.5454545454545,14.5454545454545)
--(axis cs:14.7474747474747,14.7474747474747)
--(axis cs:14.9494949494949,14.9494949494949)
--(axis cs:15.1515151515152,15.1515151515152)
--(axis cs:15.3535353535354,15.3535353535354)
--(axis cs:15.5555555555556,15.5555555555556)
--(axis cs:15.7575757575758,15.7575757575758)
--(axis cs:15.959595959596,15.959595959596)
--(axis cs:16.1616161616162,16.1616161616162)
--(axis cs:16.3636363636364,16.3636363636364)
--(axis cs:16.5656565656566,16.5656565656566)
--(axis cs:16.7676767676768,16.7676767676768)
--(axis cs:16.969696969697,16.969696969697)
--(axis cs:17.1717171717172,17.1717171717172)
--(axis cs:17.3737373737374,17.3737373737374)
--(axis cs:17.5757575757576,17.5757575757576)
--(axis cs:17.7777777777778,17.7777777777778)
--(axis cs:17.979797979798,17.979797979798)
--(axis cs:18.1818181818182,18.1818181818182)
--(axis cs:18.3838383838384,18.3838383838384)
--(axis cs:18.5858585858586,18.5858585858586)
--(axis cs:18.7878787878788,18.7878787878788)
--(axis cs:18.989898989899,18.989898989899)
--(axis cs:19.1919191919192,19.1919191919192)
--(axis cs:19.3939393939394,19.3939393939394)
--(axis cs:19.5959595959596,19.5959595959596)
--(axis cs:19.7979797979798,19.7979797979798)
--(axis cs:20,20)
--(axis cs:20,-20)
--(axis cs:20,-20)
--(axis cs:19.7979797979798,-19.7979797979798)
--(axis cs:19.5959595959596,-19.5959595959596)
--(axis cs:19.3939393939394,-19.3939393939394)
--(axis cs:19.1919191919192,-19.1919191919192)
--(axis cs:18.989898989899,-18.989898989899)
--(axis cs:18.7878787878788,-18.7878787878788)
--(axis cs:18.5858585858586,-18.5858585858586)
--(axis cs:18.3838383838384,-18.3838383838384)
--(axis cs:18.1818181818182,-18.1818181818182)
--(axis cs:17.979797979798,-17.979797979798)
--(axis cs:17.7777777777778,-17.7777777777778)
--(axis cs:17.5757575757576,-17.5757575757576)
--(axis cs:17.3737373737374,-17.3737373737374)
--(axis cs:17.1717171717172,-17.1717171717172)
--(axis cs:16.969696969697,-16.969696969697)
--(axis cs:16.7676767676768,-16.7676767676768)
--(axis cs:16.5656565656566,-16.5656565656566)
--(axis cs:16.3636363636364,-16.3636363636364)
--(axis cs:16.1616161616162,-16.1616161616162)
--(axis cs:15.959595959596,-15.959595959596)
--(axis cs:15.7575757575758,-15.7575757575758)
--(axis cs:15.5555555555556,-15.5555555555556)
--(axis cs:15.3535353535354,-15.3535353535354)
--(axis cs:15.1515151515152,-15.1515151515152)
--(axis cs:14.9494949494949,-14.9494949494949)
--(axis cs:14.7474747474747,-14.7474747474747)
--(axis cs:14.5454545454545,-14.5454545454545)
--(axis cs:14.3434343434343,-14.3434343434343)
--(axis cs:14.1414141414141,-14.1414141414141)
--(axis cs:13.9393939393939,-13.9393939393939)
--(axis cs:13.7373737373737,-13.7373737373737)
--(axis cs:13.5353535353535,-13.5353535353535)
--(axis cs:13.3333333333333,-13.3333333333333)
--(axis cs:13.1313131313131,-13.1313131313131)
--(axis cs:12.9292929292929,-12.9292929292929)
--(axis cs:12.7272727272727,-12.7272727272727)
--(axis cs:12.5252525252525,-12.5252525252525)
--(axis cs:12.3232323232323,-12.3232323232323)
--(axis cs:12.1212121212121,-12.1212121212121)
--(axis cs:11.9191919191919,-11.9191919191919)
--(axis cs:11.7171717171717,-11.7171717171717)
--(axis cs:11.5151515151515,-11.5151515151515)
--(axis cs:11.3131313131313,-11.3131313131313)
--(axis cs:11.1111111111111,-11.1111111111111)
--(axis cs:10.9090909090909,-10.9090909090909)
--(axis cs:10.7070707070707,-10.7070707070707)
--(axis cs:10.5050505050505,-10.5050505050505)
--(axis cs:10.3030303030303,-10.3030303030303)
--(axis cs:10.1010101010101,-10.1010101010101)
--(axis cs:9.8989898989899,-9.8989898989899)
--(axis cs:9.6969696969697,-9.6969696969697)
--(axis cs:9.49494949494949,-9.49494949494949)
--(axis cs:9.29292929292929,-9.29292929292929)
--(axis cs:9.09090909090909,-9.09090909090909)
--(axis cs:8.88888888888889,-8.88888888888889)
--(axis cs:8.68686868686869,-8.68686868686869)
--(axis cs:8.48484848484848,-8.48484848484848)
--(axis cs:8.28282828282828,-8.28282828282828)
--(axis cs:8.08080808080808,-8.08080808080808)
--(axis cs:7.87878787878788,-7.87878787878788)
--(axis cs:7.67676767676768,-7.67676767676768)
--(axis cs:7.47474747474747,-7.47474747474747)
--(axis cs:7.27272727272727,-7.27272727272727)
--(axis cs:7.07070707070707,-7.07070707070707)
--(axis cs:6.86868686868687,-6.86868686868687)
--(axis cs:6.66666666666667,-6.66666666666667)
--(axis cs:6.46464646464646,-6.46464646464646)
--(axis cs:6.26262626262626,-6.26262626262626)
--(axis cs:6.06060606060606,-6.06060606060606)
--(axis cs:5.85858585858586,-5.85858585858586)
--(axis cs:5.65656565656566,-5.65656565656566)
--(axis cs:5.45454545454545,-5.45454545454545)
--(axis cs:5.25252525252525,-5.25252525252525)
--(axis cs:5.05050505050505,-5.05050505050505)
--(axis cs:4.84848484848485,-4.84848484848485)
--(axis cs:4.64646464646465,-4.64646464646465)
--(axis cs:4.44444444444444,-4.44444444444444)
--(axis cs:4.24242424242424,-4.24242424242424)
--(axis cs:4.04040404040404,-4.04040404040404)
--(axis cs:3.83838383838384,-3.83838383838384)
--(axis cs:3.63636363636364,-3.63636363636364)
--(axis cs:3.43434343434343,-3.43434343434343)
--(axis cs:3.23232323232323,-3.23232323232323)
--(axis cs:3.03030303030303,-3.03030303030303)
--(axis cs:2.82828282828283,-2.82828282828283)
--(axis cs:2.62626262626263,-2.62626262626263)
--(axis cs:2.42424242424242,-2.42424242424242)
--(axis cs:2.22222222222222,-2.22222222222222)
--(axis cs:2.02020202020202,-2.02020202020202)
--(axis cs:1.81818181818182,-1.81818181818182)
--(axis cs:1.61616161616162,-1.61616161616162)
--(axis cs:1.41414141414141,-1.41414141414141)
--(axis cs:1.21212121212121,-1.21212121212121)
--(axis cs:1.01010101010101,-1.01010101010101)
--(axis cs:0.808080808080808,-0.808080808080808)
--(axis cs:0.606060606060606,-0.606060606060606)
--(axis cs:0.404040404040404,-0.404040404040404)
--(axis cs:0.202020202020202,-0.202020202020202)
--(axis cs:0,0)
--cycle;

\path [draw=gray, fill=gray, opacity=0.3, very thin]
(axis cs:-20,-20)
--(axis cs:-20,20)
--(axis cs:-19.7979797979798,19.7979797979798)
--(axis cs:-19.5959595959596,19.5959595959596)
--(axis cs:-19.3939393939394,19.3939393939394)
--(axis cs:-19.1919191919192,19.1919191919192)
--(axis cs:-18.989898989899,18.989898989899)
--(axis cs:-18.7878787878788,18.7878787878788)
--(axis cs:-18.5858585858586,18.5858585858586)
--(axis cs:-18.3838383838384,18.3838383838384)
--(axis cs:-18.1818181818182,18.1818181818182)
--(axis cs:-17.979797979798,17.979797979798)
--(axis cs:-17.7777777777778,17.7777777777778)
--(axis cs:-17.5757575757576,17.5757575757576)
--(axis cs:-17.3737373737374,17.3737373737374)
--(axis cs:-17.1717171717172,17.1717171717172)
--(axis cs:-16.969696969697,16.969696969697)
--(axis cs:-16.7676767676768,16.7676767676768)
--(axis cs:-16.5656565656566,16.5656565656566)
--(axis cs:-16.3636363636364,16.3636363636364)
--(axis cs:-16.1616161616162,16.1616161616162)
--(axis cs:-15.959595959596,15.959595959596)
--(axis cs:-15.7575757575758,15.7575757575758)
--(axis cs:-15.5555555555556,15.5555555555556)
--(axis cs:-15.3535353535354,15.3535353535354)
--(axis cs:-15.1515151515152,15.1515151515152)
--(axis cs:-14.9494949494949,14.9494949494949)
--(axis cs:-14.7474747474747,14.7474747474747)
--(axis cs:-14.5454545454545,14.5454545454545)
--(axis cs:-14.3434343434343,14.3434343434343)
--(axis cs:-14.1414141414141,14.1414141414141)
--(axis cs:-13.9393939393939,13.9393939393939)
--(axis cs:-13.7373737373737,13.7373737373737)
--(axis cs:-13.5353535353535,13.5353535353535)
--(axis cs:-13.3333333333333,13.3333333333333)
--(axis cs:-13.1313131313131,13.1313131313131)
--(axis cs:-12.9292929292929,12.9292929292929)
--(axis cs:-12.7272727272727,12.7272727272727)
--(axis cs:-12.5252525252525,12.5252525252525)
--(axis cs:-12.3232323232323,12.3232323232323)
--(axis cs:-12.1212121212121,12.1212121212121)
--(axis cs:-11.9191919191919,11.9191919191919)
--(axis cs:-11.7171717171717,11.7171717171717)
--(axis cs:-11.5151515151515,11.5151515151515)
--(axis cs:-11.3131313131313,11.3131313131313)
--(axis cs:-11.1111111111111,11.1111111111111)
--(axis cs:-10.9090909090909,10.9090909090909)
--(axis cs:-10.7070707070707,10.7070707070707)
--(axis cs:-10.5050505050505,10.5050505050505)
--(axis cs:-10.3030303030303,10.3030303030303)
--(axis cs:-10.1010101010101,10.1010101010101)
--(axis cs:-9.8989898989899,9.8989898989899)
--(axis cs:-9.6969696969697,9.6969696969697)
--(axis cs:-9.49494949494949,9.49494949494949)
--(axis cs:-9.29292929292929,9.29292929292929)
--(axis cs:-9.09090909090909,9.09090909090909)
--(axis cs:-8.88888888888889,8.88888888888889)
--(axis cs:-8.68686868686869,8.68686868686869)
--(axis cs:-8.48484848484848,8.48484848484848)
--(axis cs:-8.28282828282828,8.28282828282828)
--(axis cs:-8.08080808080808,8.08080808080808)
--(axis cs:-7.87878787878788,7.87878787878788)
--(axis cs:-7.67676767676768,7.67676767676768)
--(axis cs:-7.47474747474748,7.47474747474748)
--(axis cs:-7.27272727272727,7.27272727272727)
--(axis cs:-7.07070707070707,7.07070707070707)
--(axis cs:-6.86868686868687,6.86868686868687)
--(axis cs:-6.66666666666667,6.66666666666667)
--(axis cs:-6.46464646464647,6.46464646464647)
--(axis cs:-6.26262626262626,6.26262626262626)
--(axis cs:-6.06060606060606,6.06060606060606)
--(axis cs:-5.85858585858586,5.85858585858586)
--(axis cs:-5.65656565656566,5.65656565656566)
--(axis cs:-5.45454545454546,5.45454545454546)
--(axis cs:-5.25252525252525,5.25252525252525)
--(axis cs:-5.05050505050505,5.05050505050505)
--(axis cs:-4.84848484848485,4.84848484848485)
--(axis cs:-4.64646464646465,4.64646464646465)
--(axis cs:-4.44444444444444,4.44444444444444)
--(axis cs:-4.24242424242424,4.24242424242424)
--(axis cs:-4.04040404040404,4.04040404040404)
--(axis cs:-3.83838383838384,3.83838383838384)
--(axis cs:-3.63636363636364,3.63636363636364)
--(axis cs:-3.43434343434344,3.43434343434344)
--(axis cs:-3.23232323232323,3.23232323232323)
--(axis cs:-3.03030303030303,3.03030303030303)
--(axis cs:-2.82828282828283,2.82828282828283)
--(axis cs:-2.62626262626263,2.62626262626263)
--(axis cs:-2.42424242424243,2.42424242424243)
--(axis cs:-2.22222222222222,2.22222222222222)
--(axis cs:-2.02020202020202,2.02020202020202)
--(axis cs:-1.81818181818182,1.81818181818182)
--(axis cs:-1.61616161616162,1.61616161616162)
--(axis cs:-1.41414141414142,1.41414141414142)
--(axis cs:-1.21212121212121,1.21212121212121)
--(axis cs:-1.01010101010101,1.01010101010101)
--(axis cs:-0.80808080808081,0.80808080808081)
--(axis cs:-0.606060606060606,0.606060606060606)
--(axis cs:-0.404040404040405,0.404040404040405)
--(axis cs:-0.202020202020201,0.202020202020201)
--(axis cs:0,0)
--(axis cs:0,0)
--(axis cs:0,0)
--(axis cs:-0.202020202020201,-0.202020202020201)
--(axis cs:-0.404040404040405,-0.404040404040405)
--(axis cs:-0.606060606060606,-0.606060606060606)
--(axis cs:-0.80808080808081,-0.80808080808081)
--(axis cs:-1.01010101010101,-1.01010101010101)
--(axis cs:-1.21212121212121,-1.21212121212121)
--(axis cs:-1.41414141414142,-1.41414141414142)
--(axis cs:-1.61616161616162,-1.61616161616162)
--(axis cs:-1.81818181818182,-1.81818181818182)
--(axis cs:-2.02020202020202,-2.02020202020202)
--(axis cs:-2.22222222222222,-2.22222222222222)
--(axis cs:-2.42424242424243,-2.42424242424243)
--(axis cs:-2.62626262626263,-2.62626262626263)
--(axis cs:-2.82828282828283,-2.82828282828283)
--(axis cs:-3.03030303030303,-3.03030303030303)
--(axis cs:-3.23232323232323,-3.23232323232323)
--(axis cs:-3.43434343434344,-3.43434343434344)
--(axis cs:-3.63636363636364,-3.63636363636364)
--(axis cs:-3.83838383838384,-3.83838383838384)
--(axis cs:-4.04040404040404,-4.04040404040404)
--(axis cs:-4.24242424242424,-4.24242424242424)
--(axis cs:-4.44444444444444,-4.44444444444444)
--(axis cs:-4.64646464646465,-4.64646464646465)
--(axis cs:-4.84848484848485,-4.84848484848485)
--(axis cs:-5.05050505050505,-5.05050505050505)
--(axis cs:-5.25252525252525,-5.25252525252525)
--(axis cs:-5.45454545454546,-5.45454545454546)
--(axis cs:-5.65656565656566,-5.65656565656566)
--(axis cs:-5.85858585858586,-5.85858585858586)
--(axis cs:-6.06060606060606,-6.06060606060606)
--(axis cs:-6.26262626262626,-6.26262626262626)
--(axis cs:-6.46464646464647,-6.46464646464647)
--(axis cs:-6.66666666666667,-6.66666666666667)
--(axis cs:-6.86868686868687,-6.86868686868687)
--(axis cs:-7.07070707070707,-7.07070707070707)
--(axis cs:-7.27272727272727,-7.27272727272727)
--(axis cs:-7.47474747474748,-7.47474747474748)
--(axis cs:-7.67676767676768,-7.67676767676768)
--(axis cs:-7.87878787878788,-7.87878787878788)
--(axis cs:-8.08080808080808,-8.08080808080808)
--(axis cs:-8.28282828282828,-8.28282828282828)
--(axis cs:-8.48484848484848,-8.48484848484848)
--(axis cs:-8.68686868686869,-8.68686868686869)
--(axis cs:-8.88888888888889,-8.88888888888889)
--(axis cs:-9.09090909090909,-9.09090909090909)
--(axis cs:-9.29292929292929,-9.29292929292929)
--(axis cs:-9.49494949494949,-9.49494949494949)
--(axis cs:-9.6969696969697,-9.6969696969697)
--(axis cs:-9.8989898989899,-9.8989898989899)
--(axis cs:-10.1010101010101,-10.1010101010101)
--(axis cs:-10.3030303030303,-10.3030303030303)
--(axis cs:-10.5050505050505,-10.5050505050505)
--(axis cs:-10.7070707070707,-10.7070707070707)
--(axis cs:-10.9090909090909,-10.9090909090909)
--(axis cs:-11.1111111111111,-11.1111111111111)
--(axis cs:-11.3131313131313,-11.3131313131313)
--(axis cs:-11.5151515151515,-11.5151515151515)
--(axis cs:-11.7171717171717,-11.7171717171717)
--(axis cs:-11.9191919191919,-11.9191919191919)
--(axis cs:-12.1212121212121,-12.1212121212121)
--(axis cs:-12.3232323232323,-12.3232323232323)
--(axis cs:-12.5252525252525,-12.5252525252525)
--(axis cs:-12.7272727272727,-12.7272727272727)
--(axis cs:-12.9292929292929,-12.9292929292929)
--(axis cs:-13.1313131313131,-13.1313131313131)
--(axis cs:-13.3333333333333,-13.3333333333333)
--(axis cs:-13.5353535353535,-13.5353535353535)
--(axis cs:-13.7373737373737,-13.7373737373737)
--(axis cs:-13.9393939393939,-13.9393939393939)
--(axis cs:-14.1414141414141,-14.1414141414141)
--(axis cs:-14.3434343434343,-14.3434343434343)
--(axis cs:-14.5454545454545,-14.5454545454545)
--(axis cs:-14.7474747474747,-14.7474747474747)
--(axis cs:-14.9494949494949,-14.9494949494949)
--(axis cs:-15.1515151515152,-15.1515151515152)
--(axis cs:-15.3535353535354,-15.3535353535354)
--(axis cs:-15.5555555555556,-15.5555555555556)
--(axis cs:-15.7575757575758,-15.7575757575758)
--(axis cs:-15.959595959596,-15.959595959596)
--(axis cs:-16.1616161616162,-16.1616161616162)
--(axis cs:-16.3636363636364,-16.3636363636364)
--(axis cs:-16.5656565656566,-16.5656565656566)
--(axis cs:-16.7676767676768,-16.7676767676768)
--(axis cs:-16.969696969697,-16.969696969697)
--(axis cs:-17.1717171717172,-17.1717171717172)
--(axis cs:-17.3737373737374,-17.3737373737374)
--(axis cs:-17.5757575757576,-17.5757575757576)
--(axis cs:-17.7777777777778,-17.7777777777778)
--(axis cs:-17.979797979798,-17.979797979798)
--(axis cs:-18.1818181818182,-18.1818181818182)
--(axis cs:-18.3838383838384,-18.3838383838384)
--(axis cs:-18.5858585858586,-18.5858585858586)
--(axis cs:-18.7878787878788,-18.7878787878788)
--(axis cs:-18.989898989899,-18.989898989899)
--(axis cs:-19.1919191919192,-19.1919191919192)
--(axis cs:-19.3939393939394,-19.3939393939394)
--(axis cs:-19.5959595959596,-19.5959595959596)
--(axis cs:-19.7979797979798,-19.7979797979798)
--(axis cs:-20,-20)
--cycle;

\path [draw=none, fill=red, fill opacity=0.3, very thin]
(axis cs:5.885,0.1257)
--(axis cs:-0.1265,5.8549)
--(axis cs:-5.855,-0.1257)
--(axis cs:0.1265,-5.8549)
--cycle;
\addplot [thick, steelblue52138189, mark=*, mark size=\ms, mark options={solid}, only marks, forget plot]
table {%
-18 15
};
\addplot [steelblue52138189, forget plot]
table {%
-18 15
-1.16900000992364 -2.6046
-2.09205292600452 -0.173185240967556
-0.949834438097819 0.782243339450156
-0.713873197698611 0.754177349536279
-0.269992824159336 1.00320009449609
0.0672203332857554 0.973933165637622
0.278274381430276 0.858957361034826
0.390635681890986 0.732253939579054
0.437345674767476 0.618337146238709
0.444360461005934 0.52273121648247
0.429161003102879 0.444407947303512
0.402616912893589 0.380554215266122
0.544855857299057 0.161915106888808
0.497846183629141 0.13242429973311
0.443841611173539 0.11744231680031
0.395386709386412 0.10442404526136
0.352108663040532 0.0930341704101267
0.313686975259584 0.0828695231009415
0.279515760732955 0.0738453821403773
0.249123299836118 0.0658152768713765
0.222072618173569 0.0586689508993017
0.197984592081037 0.0523051665958109
0.176525656626645 0.0466359696044922
0.157402747770971 0.0415839276928661
0.140357580071535 0.0370807968619911
0.125161883640346 0.0330662756046025
0.111613443454859 0.0294869393438369
0.0995327723866191 0.0262953702876301
0.0887603245338149 0.0234494181561153
0.0791541274157061 0.0209115755535213
};
\addplot [thick, steelblue52138189, mark=x, mark size=\ms, mark options={solid}, only marks, forget plot]
table {%
0.0791541274157061 0.0209115755535213
};
\addplot [thick, firebrick166640, mark=*, mark size=\ms, mark options={solid}, only marks, forget plot]
table {%
19 0
};
\addplot [firebrick166640, forget plot]
table {%
19 0
8.17449999005757 8.5328
5.23201858059516 7.16193806903061
2.5998820617337 5.57446041965069
0.23763428581112 4.12714461756445
0.525827078395275 2.82254583945779
0.652103116982323 1.99849581407656
0.691598146554726 1.47099004149785
0.687466432592946 1.12362804811595
0.661320617543897 0.887008924466822
0.62386148072082 0.720006875382314
0.580765442557047 0.597901944270629
0.535372472628426 0.505534170974767
0.720520898289409 0.212475478527798
0.654911561172043 0.174999733946798
0.584549691949643 0.154398098463754
0.520859477173063 0.137645507024135
0.464423174547019 0.122687708558503
0.414162406578694 0.109418257133291
0.369384694467932 0.0975868467614406
0.329455611388773 0.087038226017803
0.293840361443027 0.0776291357868634
0.262069577079817 0.069235644640449
0.233728173737004 0.0617481965310109
0.208447100495486 0.0550692346435478
0.185897132433514 0.0491118034759022
0.165784300726055 0.0437982331273965
0.147846008783021 0.0390591506064008
0.131847714320615 0.034832592177372
0.117579983807496 0.031063228099798
0.104855858805837 0.0277016661699351
};
\addplot [thick, firebrick166640, mark=x, mark size=\ms, mark options={solid}, only marks, forget plot]
table {%
0.104855858805837 0.0277016661699351
};
\addplot [thick, slategray122104166, mark=*, mark size=\ms, mark options={solid}, only marks, forget plot]
table {%
10 18
};
\addplot [slategray122104166, forget plot]
table {%
10 18
5.45699999004024 13.2192
3.94551398219568 7.34490197902892
1.20596976564041 4.37390486704019
-0.43587991017478 2.63482955824625
-0.000707439984046543 1.51915837008864
0.167099513755021 0.991270752322404
0.244828025623241 0.714247483538257
0.285100625805834 0.551925859010911
0.303841987463143 0.446806193641785
0.307567995774608 0.372916844478542
0.300595541647847 0.317573096470534
0.286457372049379 0.274094063338962
0.393387129045726 0.118646367751612
0.364567495316628 0.0970272480893388
0.327979094616263 0.0866873096447454
0.294008947074646 0.0776974481330566
0.263152080542944 0.0695129578316282
0.235228409619651 0.062147615513618
0.210105733996903 0.0555065215963679
0.187556666883064 0.0495504701894946
0.167365266832075 0.0442158506218267
0.14931037706299 0.0394460298753136
0.13318220388988 0.0351851469585485
0.11878445934694 0.031381436115948
0.105936873754546 0.0279872575395419
0.0944755599753486 0.0249593150744543
0.0842525811339708 0.0222585261838837
0.0751350123224076 0.0198497733177972
0.0670037699061073 0.0177015961471229
0.0597523708188894 0.0157858630703825
};
\addplot [thick, slategray122104166, mark=x, mark size=\ms, mark options={solid}, only marks, forget plot]
table {%
0.0597523708188894 0.0157858630703825
};
\addplot [thick, black]
table {%
-100 100
};
\addlegendentry{Sw-ADMM}
\draw (axis cs:17,1.5) node[
  anchor=base west,
  text=black,
  rotate=0.0
]{$P_1$};
\draw (axis cs:-2,-19) node[
  anchor=base west,
  text=black,
  rotate=0.0
]{$P_2$};
\draw (axis cs:-19,0) node[
  anchor=base west,
  text=black,
  rotate=0.0
]{$P_3$};
\draw (axis cs:-2,18) node[
  anchor=base west,
  text=black,
  rotate=0.0
]{$P_4$};
\draw (axis cs:-4,1.5) node[
  anchor=base west,
  text=black,
  rotate=0.0
]{$\mathcal{X}_{T}$};
\end{axis}

\end{tikzpicture}

%% file: media/tikz/ex_1_u.tex
\begin{tikzpicture}

\definecolor{darkgray178}{RGB}{178,178,178}
\definecolor{firebrick166640}{RGB}{166,6,40}
\definecolor{silver188}{RGB}{188,188,188}
\definecolor{steelblue52138189}{RGB}{52,138,189}
\definecolor{whitesmoke238}{RGB}{238,238,238}

\newcommand{\ms}{1}

\begin{groupplot}[group style={group size=1 by 3, vertical sep = 0.15cm}]
\nextgroupplot[
height=0.7*\axisdefaultheight,
width=\axisdefaultwidth,
axis background/.style={fill=whitesmoke238},
axis line style={silver188},
scaled x ticks=manual:{}{\pgfmathparse{#1}},
tick pos=left,
tick scale binop=\times,
x grid style={darkgray178},
xmajorgrids,
xmin=-3.7, xmax=77.7,
xtick style={color=black},
xticklabels={},
y grid style={darkgray178},
ylabel={\(\displaystyle u_1(0)\)},
ymajorgrids,
ymin=-1.93770991191327, ymax=-0.267168887700133,
ytick style={color=black},
ytick={-2,-1.5,-1,-0.5,0},
yticklabels={
  \(\displaystyle {\ensuremath{-}2.0}\),
  \(\displaystyle {\ensuremath{-}1.5}\),
  \(\displaystyle {\ensuremath{-}1.0}\),
  \(\displaystyle {\ensuremath{-}0.5}\),
  \(\displaystyle {0.0}\)
}
]
\addplot [thick, steelblue52138189]
table {%
0 -1.86177622899449
1 -1.04872554275676
2 -0.545626564376786
3 -0.35485144833847
4 -0.343102570618912
5 -0.376587360196417
6 -0.408182965537409
7 -0.636224147388266
8 -0.586656154672762
9 -0.622157400488667
10 -0.402968337575442
11 -0.550134765118195
12 -0.810682127264677
13 -0.477612806246249
14 -0.556342117185935
15 -0.800902113709557
16 -0.460138497826109
17 -0.523330394285193
18 -0.765187152953821
19 -0.433576709870014
20 -0.506267330257792
21 -0.750738947604487
22 -0.427491271536545
23 -0.50393027690959
24 -0.750818497131783
25 -0.424162379388705
26 -0.502392610200042
27 -0.751895712966066
28 -0.423865481358576
29 -0.502184536770164
30 -0.751710797144408
31 -0.425169024227244
32 -0.502653504311454
33 -0.751243075014677
34 -0.425112352945034
35 -0.502523916957569
36 -0.751051307951836
37 -0.424487092132897
38 -0.502228917827501
39 -0.751070014707474
40 -0.424446931688714
41 -0.50224546313934
42 -0.751124340337084
43 -0.424632424588764
44 -0.502333898876565
45 -0.751127279559605
46 -0.42463746024125
47 -0.502321830614931
48 -0.751109383508003
49 -0.424583329543916
50 -0.491616335422313
51 -0.495068632720185
52 -0.477501652838087
53 -0.456723562461612
54 -0.438581515502822
55 -0.424093710893114
56 -0.41287986323968
57 -0.404399888821877
58 -0.398182889835381
59 -0.393782189011357
60 -0.390755050960861
61 -0.388699243202356
62 -0.387295262184295
63 -0.386319689251695
64 -0.385630382200115
65 -0.385140525088578
66 -0.384795745818483
67 -0.384559378572008
68 -0.384404581182108
69 -0.384310305195043
70 -0.384259127247019
71 -0.384236198539108
72 -0.384229065431508
73 -0.384228023798954
74 -0.384226567967704
};
\addplot [thick, firebrick166640, mark=*, mark size=\ms, mark options={solid}, only marks]
table {%
1 -1.04872554275676
7 -0.636224147388266
10 -0.402968337575442
11 -0.550134765118195
12 -0.810682127264677
13 -0.477612806246249
14 -0.556342117185935
15 -0.800902113709557
16 -0.460138497826109
17 -0.523330394285193
18 -0.765187152953821
19 -0.433576709870014
20 -0.506267330257792
21 -0.750738947604487
22 -0.427491271536545
23 -0.50393027690959
24 -0.750818497131783
25 -0.424162379388705
26 -0.502392610200042
27 -0.751895712966066
28 -0.423865481358576
29 -0.502184536770164
30 -0.751710797144408
31 -0.425169024227244
32 -0.502653504311454
33 -0.751243075014677
34 -0.425112352945034
35 -0.502523916957569
36 -0.751051307951836
37 -0.424487092132897
38 -0.502228917827501
39 -0.751070014707474
40 -0.424446931688714
41 -0.50224546313934
42 -0.751124340337084
43 -0.424632424588764
44 -0.502333898876565
45 -0.751127279559605
46 -0.42463746024125
47 -0.502321830614931
48 -0.751109383508003
49 -0.424583329543916
};

\nextgroupplot[
height=0.7*\axisdefaultheight,
width=\axisdefaultwidth,
axis background/.style={fill=whitesmoke238},
axis line style={silver188},
scaled x ticks=manual:{}{\pgfmathparse{#1}},
tick pos=left,
tick scale binop=\times,
x grid style={darkgray178},
xmajorgrids,
xmin=-3.7, xmax=77.7,
xtick style={color=black},
xticklabels={},
y grid style={darkgray178},
ylabel={\(\displaystyle u_2(0)\)},
ymajorgrids,
ymin=-3.00000000988304, ymax=-3.00000000943521,
ytick style={color=black},
ytick={-3},
yticklabels={
  \(\displaystyle {\ensuremath{-}3.0}\),
}
]
\addplot [thick, steelblue52138189]
table {%
0 -3.00000000972447
1 -3.00000000945557
2 -3.00000000958765
3 -3.00000000973376
4 -3.00000000979906
5 -3.0000000098364
6 -3.00000000985566
7 -3.00000000986151
8 -3.00000000986268
9 -3.00000000985954
10 -3.00000000985544
11 -3.00000000985273
12 -3.00000000984664
13 -3.00000000984694
14 -3.0000000098477
15 -3.0000000098442
16 -3.0000000098468
17 -3.00000000984931
18 -3.00000000984703
19 -3.0000000098496
20 -3.00000000985124
21 -3.00000000984784
22 -3.00000000984922
23 -3.00000000985016
24 -3.00000000984653
25 -3.00000000984822
26 -3.00000000984973
27 -3.00000000984662
28 -3.00000000984869
29 -3.00000000985032
30 -3.00000000984714
31 -3.00000000984899
32 -3.00000000985036
33 -3.00000000984699
34 -3.00000000984875
35 -3.00000000985013
36 -3.00000000984682
37 -3.00000000984868
38 -3.00000000985015
39 -3.0000000098469
40 -3.00000000984877
41 -3.00000000985022
42 -3.00000000984694
43 -3.00000000984878
44 -3.00000000985021
45 -3.00000000984691
46 -3.00000000984875
47 -3.00000000985018
48 -3.0000000098469
49 -3.00000000984875
50 -3.00000000985019
51 -3.00000000985103
52 -3.00000000985145
53 -3.00000000985154
54 -3.00000000985144
55 -3.00000000985127
56 -3.00000000985111
57 -3.00000000985099
58 -3.00000000985091
59 -3.00000000985086
60 -3.00000000985083
61 -3.00000000985083
62 -3.00000000985084
63 -3.00000000985087
64 -3.00000000985089
65 -3.00000000985092
66 -3.00000000985093
67 -3.00000000985094
68 -3.00000000985093
69 -3.00000000985092
70 -3.00000000985091
71 -3.00000000985091
72 -3.0000000098509
73 -3.0000000098509
74 -3.0000000098509
};

\nextgroupplot[
height=0.7*\axisdefaultheight,
width=\axisdefaultwidth,
axis background/.style={fill=whitesmoke238},
axis line style={silver188},
tick pos=left,
tick scale binop=\times,
x grid style={darkgray178},
xlabel={\(\displaystyle \tau\)},
xmajorgrids,
xmin=-3.7, xmax=77.7,
xtick style={color=black},
xtick={-20,0,20,40,60,80},
xticklabels={
  \(\displaystyle {\ensuremath{-}20}\),
  \(\displaystyle {0}\),
  \(\displaystyle {20}\),
  \(\displaystyle {40}\),
  \(\displaystyle {60}\),
  \(\displaystyle {80}\)
},
y grid style={darkgray178},
ylabel={\(\displaystyle u_3(0)\)},
ymajorgrids,
ymin=-3.02180121132501, ymax=-2.54217477601347,
ytick style={color=black},
ytick={-3.25,-3,-2.75,-2.5},
yticklabels={
  \(\displaystyle {\ensuremath{-}3.25}\),
  \(\displaystyle {\ensuremath{-}3.00}\),
  \(\displaystyle {\ensuremath{-}2.75}\),
  \(\displaystyle {\ensuremath{-}2.50}\)
}
]
\addplot [thick, steelblue52138189]
table {%
0 -2.84150188869545
1 -2.56397597761854
2 -2.9600679667378
3 -3.0000000092834
4 -3.00000000961257
5 -3.00000000968067
6 -3.0000000096983
7 -3.00000000970765
8 -3.00000000971618
9 -3.00000000971994
10 -3.00000000971652
11 -3.00000000970657
12 -3.0000000096953
13 -3.00000000968554
14 -3.00000000967938
15 -3.00000000968007
16 -3.00000000968409
17 -3.00000000968797
18 -3.0000000096925
19 -3.00000000969514
20 -3.00000000969475
21 -3.00000000969424
22 -3.00000000969288
23 -3.0000000096904
24 -3.00000000968981
25 -3.00000000968977
26 -3.00000000968919
27 -3.00000000969036
28 -3.00000000969141
29 -3.00000000969111
30 -3.00000000969189
31 -3.00000000969223
32 -3.00000000969119
33 -3.00000000969143
34 -3.00000000969151
35 -3.0000000096905
36 -3.00000000969095
37 -3.00000000969132
38 -3.00000000969056
39 -3.00000000969116
40 -3.00000000969158
41 -3.00000000969078
42 -3.0000000096913
43 -3.00000000969162
44 -3.00000000969074
45 -3.00000000969122
46 -3.00000000969153
47 -3.00000000969067
48 -3.00000000969118
49 -3.00000000969153
50 -3.00000000969069
51 -3.00000000969151
52 -3.00000000969319
53 -3.00000000969456
54 -3.00000000969494
55 -3.00000000969442
56 -3.00000000969351
57 -3.00000000969276
58 -3.0000000096924
59 -3.0000000096924
60 -3.00000000969257
61 -3.00000000969275
62 -3.00000000969286
63 -3.00000000969289
64 -3.00000000969289
65 -3.00000000969289
66 -3.00000000969288
67 -3.00000000969288
68 -3.00000000969287
69 -3.00000000969286
70 -3.00000000969285
71 -3.00000000969284
72 -3.00000000969284
73 -3.00000000969285
74 -3.00000000969285
};
\end{groupplot}

\end{tikzpicture}

%% file: media/tikz/ex_1_r.tex
\begin{tikzpicture}

\definecolor{darkgray178}{RGB}{178,178,178}
\definecolor{silver188}{RGB}{188,188,188}
\definecolor{steelblue52138189}{RGB}{52,138,189}
\definecolor{whitesmoke238}{RGB}{238,238,238}

\begin{axis}[
height=0.8*\axisdefaultheight,
width=\axisdefaultwidth,
axis background/.style={fill=whitesmoke238},
axis line style={silver188},
tick pos=left,
tick scale binop=\times,
x grid style={darkgray178},
xlabel={\(\displaystyle \tau\)},
xmajorgrids,
xmin=-3.7, xmax=77.7,
xtick style={color=black},
xtick={-20,0,20,40,60,80},
xticklabels={
  \(\displaystyle {\ensuremath{-}20}\),
  \(\displaystyle {0}\),
  \(\displaystyle {20}\),
  \(\displaystyle {40}\),
  \(\displaystyle {60}\),
  \(\displaystyle {80}\)
},
y grid style={darkgray178},
ylabel={\(\displaystyle r\)},
ymajorgrids,
ymin=-2.29284321166362, ymax=48.2624895929328,
ytick style={color=black},
ytick={-10,0,20,40},
yticklabels={
  \(\displaystyle {\ensuremath{-}10}\),
  \(\displaystyle {0}\),
  \(\displaystyle {20}\),
  \(\displaystyle {40}\),
}
]
\addplot [thick, steelblue52138189]
table {%
0 45.9645199199966
1 25.3912810495892
2 23.4492892854164
3 21.1685433657287
4 16.1436505313849
5 12.0070304389595
6 10.6838080407492
7 11.6547392481871
8 10.8406254743039
9 9.77460891481191
10 8.12425066120217
11 7.9421703943902
12 5.24886872377169
13 4.23678766667037
14 5.91202113959367
15 6.24529318701859
16 4.31069712689752
17 5.04533499420789
18 5.13847931885558
19 3.89866739495401
20 4.74353058616838
21 4.87413504454224
22 3.67047154281593
23 4.33651103484469
24 4.83081960396905
25 3.49056996123361
26 4.32215690217367
27 4.89277586876675
28 3.4958843767031
29 4.39668911157256
30 4.87181829181918
31 3.51416586233542
32 4.39103206033719
33 4.86398573771796
34 3.51771149467371
35 4.36392236314295
36 4.86528868101003
37 3.5100037372687
38 4.3644138707788
39 4.86496996906211
40 3.50710406938908
41 4.3686568178958
42 4.86403217886881
43 3.50861384100063
44 4.36859386401943
45 4.86472309080807
46 3.50969052174031
47 4.36800119679452
48 4.86496832013998
49 3.50930525599329
50 2.73713047059605
51 1.90654037892572
52 1.39096412090613
53 1.07888350908415
54 0.940093812818137
55 0.717388889699901
56 0.45549876390277
57 0.312166217423195
58 0.291548068339198
59 0.272405614845321
60 0.204796421923192
61 0.131097971458349
62 0.0741554138984825
63 0.0732508941584534
64 0.0848544481059898
65 0.0736875916763407
66 0.0507116618441125
67 0.0314725384113523
68 0.0278271376406054
69 0.03656488538087
70 0.0365668964252059
71 0.0292052346646132
72 0.0187144793182665
73 0.00918506732967265
74 0.00512646127258337
};
\end{axis}

\end{tikzpicture}

%% file: media/tikz/ex_2_u.tex
\begin{tikzpicture}

\definecolor{darkgray178}{RGB}{178,178,178}
\definecolor{firebrick166640}{RGB}{166,6,40}
\definecolor{silver188}{RGB}{188,188,188}
\definecolor{steelblue52138189}{RGB}{52,138,189}
\definecolor{whitesmoke238}{RGB}{238,238,238}

\newcommand{\ms}{1}

\begin{groupplot}[group style={group size=1 by 5, vertical sep = 0.15cm}]
\nextgroupplot[
height=0.7*\axisdefaultheight,
width=\axisdefaultwidth,
xmin=-1, xmax=21,
ymin=-1.1, ymax=1.1,
ytick={-1,0,1},
yticklabels={
	\(-1\),
	\(0\),
	\(1\)
},
axis background/.style={fill=whitesmoke238},
axis line style={silver188},
scaled x ticks=manual:{}{\pgfmathparse{#1}},
tick pos=left,
tick scale binop=\times,
x grid style={darkgray178},
xmajorgrids,
xtick style={color=black},
xticklabels={},
y grid style={darkgray178},
ylabel={\(\displaystyle u_1(0)\)},
ymajorgrids,
ytick style={color=black},
]
\addplot [thick, steelblue52138189]
table {%
0 -0.909869639794168
1 -0.909869639794169
2 -0.909869639794169
3 -0.909869639794169
4 -0.909869639794169
5 -0.909869639794169
6 -0.909869639794169
7 -0.909869639794169
8 -0.909869639794169
9 -0.909869639794169
10 -0.909869639794169
11 -0.909869639794169
12 -0.909869639794169
13 -0.909869639794169
14 -0.909869639794169
15 -0.909869639794169
16 -0.909869639794169
17 -0.909869639794169
18 -0.909869639794169
19 -0.909869639794169
20 -0.909869639794169
21 -0.909869639794169
22 -0.909869639794169
23 -0.909869639794169
24 -0.909869639794169
25 -0.909869639794169
26 -0.909869639794169
27 -0.909869639794169
28 -0.909869639794169
29 -0.909869639794169
30 -0.909869639794169
31 -0.909869639794169
32 -0.909869639794169
33 -0.909869639794169
34 -0.909869639794169
35 -0.909869639794169
36 -0.909869639794169
37 -0.909869639794169
38 -0.909869639794169
39 -0.909869639794169
40 -0.909869639794169
41 -0.909869639794169
42 -0.909869639794169
43 -0.909869639794169
44 -0.909869639794169
45 -0.909869639794169
46 -0.909869639794169
47 -0.909869639794169
48 -0.909869639794169
49 -0.909869639794169
50 -0.909869639794169
51 -0.909869639794169
52 -0.909869639794169
53 -0.909869639794169
54 -0.909869639794169
55 -0.909869639794169
56 -0.909869639794169
57 -0.909869639794169
58 -0.909869639794169
59 -0.909869639794169
60 -0.909869639794169
61 -0.909869639794169
62 -0.909869639794169
63 -0.909869639794169
64 -0.909869639794169
65 -0.909869639794169
66 -0.909869639794169
67 -0.909869639794169
68 -0.909869639794169
69 -0.909869639794169
70 -0.909869639794169
71 -0.909869639794169
72 -0.909869639794169
73 -0.909869639794169
74 -0.909869639794169
75 -0.909869639794169
76 -0.909869639794169
77 -0.909869639794169
78 -0.909869639794169
79 -0.909869639794169
80 -0.909869639794169
81 -0.909869639794169
82 -0.909869639794169
83 -0.909869639794169
84 -0.909869639794169
85 -0.909869639794169
86 -0.909869639794169
87 -0.909869639794169
88 -0.909869639794169
89 -0.909869639794169
90 -0.909869639794169
91 -0.909869639794169
92 -0.909869639794169
93 -0.909869639794169
94 -0.909869639794169
95 -0.909869639794169
96 -0.909869639794169
97 -0.909869639794169
98 -0.909869639794169
99 -0.909869639794169
};
\addplot [thick, firebrick166640, mark=*, mark size=\ms, mark options={solid}, only marks]
table {%
1 -0.909869639794169
2 -0.909869639794169
};

\nextgroupplot[
height=0.7*\axisdefaultheight,
width=\axisdefaultwidth,
xmin=-1, xmax=21,
ymin=-1.1, ymax=1.1,
ytick={-1,0,1},
yticklabels={
	\(-1\),
	\(0\),
	\(1\)
},
axis background/.style={fill=whitesmoke238},
axis line style={silver188},
scaled x ticks=manual:{}{\pgfmathparse{#1}},
tick pos=left,
tick scale binop=\times,
x grid style={darkgray178},
xmajorgrids,
xtick style={color=black},
xticklabels={},
y grid style={darkgray178},
ylabel={\(\displaystyle u_2(0)\)},
ymajorgrids,
ytick style={color=black},
]
\addplot [thick, steelblue52138189]
table {%
0 0.0350354442344052
1 -0.666287807183365
2 -0.666287807183365
3 0.0350354442344055
4 0.791179111531191
5 0.791179111531191
6 1
7 1
8 1
9 1
10 1
11 1
12 1
13 1
14 1
15 1
16 1
17 1
18 1
19 1
20 1
21 1
22 1
23 1
24 1
25 1
26 1
27 1
28 1
29 1
30 1
31 1
32 1
33 1
34 1
35 1
36 1
37 1
38 1
39 1
40 1
41 1
42 1
43 1
44 1
45 1
46 1
47 1
48 1
49 1
50 1
51 1
52 1
53 1
54 1
55 1
56 1
57 1
58 1
59 1
60 1
61 1
62 1
63 1
64 1
65 1
66 1
67 1
68 1
69 1
70 1
71 1
72 1
73 1
74 1
75 1
76 1
77 1
78 1
79 1
80 1
81 1
82 1
83 1
84 1
85 1
86 1
87 1
88 1
89 1
90 1
91 1
92 1
93 1
94 1
95 1
96 1
97 1
98 1
99 1
};
\addplot [thick, firebrick166640, mark=*, mark size=\ms, mark options={solid}, only marks]
table {%
1 -0.666287807183365
2 -0.666287807183365
3 0.0350354442344055
4 0.791179111531191
6 1
};

\nextgroupplot[
height=0.7*\axisdefaultheight,
width=\axisdefaultwidth,
xmin=-1, xmax=21,
ymin=-1.1, ymax=1.1,
ytick={-1,0,1},
yticklabels={
	\(-1\),
	\(0\),
	\(1\)
},
axis background/.style={fill=whitesmoke238},
axis line style={silver188},
scaled x ticks=manual:{}{\pgfmathparse{#1}},
tick pos=left,
tick scale binop=\times,
x grid style={darkgray178},
xmajorgrids,
xtick style={color=black},
xticklabels={},
y grid style={darkgray178},
ylabel={\(\displaystyle u_3(0)\)},
ymajorgrids,
ytick style={color=black},
]
\addplot [thick, steelblue52138189]
table {%
0 -0.179104017747104
1 -0.179104017747104
2 -0.179104017747104
3 0.522892531427163
4 0.522892531427163
5 0.522892531427163
6 0.522892531427163
7 0.522892531427163
8 0.522892531427163
9 0.522892531427163
10 0.522892531427163
11 0.522892531427163
12 0.522892531427163
13 0.522892531427163
14 0.522892531427163
15 0.522892531427163
16 0.522892531427163
17 0.522892531427163
18 0.522892531427163
19 0.522892531427163
20 0.522892531427163
21 0.522892531427163
22 0.522892531427163
23 0.522892531427163
24 0.522892531427163
25 0.522892531427163
26 0.522892531427163
27 0.522892531427163
28 0.522892531427163
29 0.522892531427163
30 0.522892531427163
31 0.522892531427163
32 0.522892531427163
33 0.522892531427163
34 0.522892531427163
35 0.522892531427163
36 0.522892531427163
37 0.522892531427163
38 0.522892531427163
39 0.522892531427163
40 0.522892531427163
41 0.522892531427163
42 0.522892531427163
43 0.522892531427163
44 0.522892531427163
45 0.522892531427163
46 0.522892531427163
47 0.522892531427163
48 0.522892531427163
49 0.522892531427163
50 0.522892531427163
51 0.522892531427163
52 0.522892531427163
53 0.522892531427163
54 0.522892531427163
55 0.522892531427163
56 0.522892531427163
57 0.522892531427163
58 0.522892531427163
59 0.522892531427163
60 0.522892531427163
61 0.522892531427163
62 0.522892531427163
63 0.522892531427163
64 0.522892531427163
65 0.522892531427163
66 0.522892531427163
67 0.522892531427163
68 0.522892531427163
69 0.522892531427163
70 0.522892531427163
71 0.522892531427163
72 0.522892531427163
73 0.522892531427163
74 0.522892531427163
75 0.522892531427163
76 0.522892531427163
77 0.522892531427163
78 0.522892531427163
79 0.522892531427163
80 0.522892531427163
81 0.522892531427163
82 0.522892531427163
83 0.522892531427163
84 0.522892531427163
85 0.522892531427163
86 0.522892531427163
87 0.522892531427163
88 0.522892531427163
89 0.522892531427163
90 0.522892531427163
91 0.522892531427163
92 0.522892531427163
93 0.522892531427163
94 0.522892531427163
95 0.522892531427163
96 0.522892531427163
97 0.522892531427163
98 0.522892531427163
99 0.522892531427163
};
\addplot [thick, firebrick166640, mark=*, mark size=\ms, mark options={solid}, only marks]
table {%
4 0.522892531427163
5 0.522892531427163
};

\nextgroupplot[
height=0.7*\axisdefaultheight,
width=\axisdefaultwidth,
xmin=-1, xmax=21,
ymin=-1.1, ymax=1.1,
ytick={-1,0,1},
yticklabels={
	\(-1\),
	\(0\),
	\(1\)
},
axis background/.style={fill=whitesmoke238},
axis line style={silver188},
scaled x ticks=manual:{}{\pgfmathparse{#1}},
tick pos=left,
tick scale binop=\times,
x grid style={darkgray178},
xmajorgrids,
xtick style={color=black},
xticklabels={},
y grid style={darkgray178},
ylabel={\(\displaystyle u_4(0)\)},
ymajorgrids,
ytick style={color=black},
]
\addplot [thick, steelblue52138189]
table {%
0 -0.658163988657845
1 -0.658163988657845
2 -0.658163988657845
3 0.0431592627599252
4 0.799302930056711
5 0.799302930056711
6 -0.658163988657844
7 -0.658163988657844
8 -0.658163988657845
9 -0.658163988657845
10 -0.658163988657845
11 -0.658163988657845
12 -0.658163988657845
13 -0.658163988657845
14 -0.658163988657845
15 -0.658163988657845
16 -0.658163988657845
17 -0.658163988657845
18 -0.658163988657845
19 -0.658163988657845
20 -0.658163988657845
21 -0.658163988657845
22 -0.658163988657845
23 -0.658163988657845
24 -0.658163988657845
25 -0.658163988657845
26 -0.658163988657845
27 -0.658163988657845
28 -0.658163988657845
29 -0.658163988657845
30 -0.658163988657845
31 -0.658163988657845
32 -0.658163988657845
33 -0.658163988657845
34 -0.658163988657845
35 -0.658163988657845
36 -0.658163988657845
37 -0.658163988657845
38 -0.658163988657845
39 -0.658163988657845
40 -0.658163988657845
41 -0.658163988657845
42 -0.658163988657845
43 -0.658163988657845
44 -0.658163988657845
45 -0.658163988657845
46 -0.658163988657845
47 -0.658163988657845
48 -0.658163988657845
49 -0.658163988657845
50 -0.658163988657845
51 -0.658163988657845
52 -0.658163988657845
53 -0.658163988657845
54 -0.658163988657845
55 -0.658163988657845
56 -0.658163988657845
57 -0.658163988657845
58 -0.658163988657845
59 -0.658163988657845
60 -0.658163988657845
61 -0.658163988657845
62 -0.658163988657845
63 -0.658163988657845
64 -0.658163988657845
65 -0.658163988657845
66 -0.658163988657845
67 -0.658163988657845
68 -0.658163988657845
69 -0.658163988657845
70 -0.658163988657845
71 -0.658163988657845
72 -0.658163988657845
73 -0.658163988657845
74 -0.658163988657845
75 -0.658163988657845
76 -0.658163988657845
77 -0.658163988657845
78 -0.658163988657845
79 -0.658163988657845
80 -0.658163988657845
81 -0.658163988657845
82 -0.658163988657845
83 -0.658163988657845
84 -0.658163988657845
85 -0.658163988657845
86 -0.658163988657845
87 -0.658163988657845
88 -0.658163988657845
89 -0.658163988657845
90 -0.658163988657845
91 -0.658163988657845
92 -0.658163988657845
93 -0.658163988657845
94 -0.658163988657845
95 -0.658163988657845
96 -0.658163988657845
97 -0.658163988657845
98 -0.658163988657845
99 -0.658163988657845
};
\addplot [thick, firebrick166640, mark=*, mark size=\ms, mark options={solid}, only marks]
table {%
1 -0.658163988657845
2 -0.658163988657845
3 0.0431592627599252
4 0.799302930056711
7 -0.658163988657844
8 -0.658163988657845
};

\nextgroupplot[
height=0.7*\axisdefaultheight,
width=\axisdefaultwidth,
xmin=-1, xmax=21,
ymin=-1.1, ymax=1.1,
ytick={-1,0,1},
yticklabels={
	\(-1\),
	\(0\),
	\(1\)
},
axis background/.style={fill=whitesmoke238},
axis line style={silver188},
tick pos=left,
tick scale binop=\times,
x grid style={darkgray178},
xlabel={\(\displaystyle \tau\)},
xmajorgrids,
xtick style={color=black},
xtick={0,5,10,15,20},
xticklabels={
  \(\displaystyle {0}\),
  \(\displaystyle {5}\),
  \(\displaystyle {10}\),
  \(\displaystyle {15}\),
  \(\displaystyle {20}\),
},
y grid style={darkgray178},
ylabel={\(\displaystyle u_5(0)\)},
ymajorgrids,
ytick style={color=black},
]
\addplot [thick, steelblue52138189]
table {%
0 0.0350354442344051
1 -0.666287807183364
2 -0.666287807183365
3 0.035035444234405
4 0.791179111531191
5 0.791179111531191
6 0.791179111531191
7 0.791179111531191
8 0.0350354442344047
9 -0.666287807183365
10 -0.577440926275993
11 -0.577440926275992
12 0.0350354442344048
13 0.791179111531191
14 0.791179111531191
15 0.791179111531191
16 0.791179111531191
17 0.791179111531191
18 0.791179111531191
19 0.791179111531191
20 0.791179111531191
21 0.791179111531191
22 0.791179111531191
23 0.791179111531191
24 0.791179111531191
25 0.791179111531191
26 0.791179111531191
27 0.791179111531191
28 0.791179111531191
29 0.791179111531191
30 0.791179111531191
31 0.791179111531191
32 0.791179111531191
33 0.791179111531191
34 0.791179111531191
35 0.791179111531191
36 0.791179111531191
37 0.791179111531191
38 0.791179111531191
39 0.791179111531191
40 0.791179111531191
41 0.791179111531191
42 0.791179111531191
43 0.791179111531191
44 0.791179111531191
45 0.791179111531191
46 0.791179111531191
47 0.791179111531191
48 0.791179111531191
49 0.791179111531191
50 0.791179111531191
51 0.791179111531191
52 0.791179111531191
53 0.791179111531191
54 0.791179111531191
55 0.791179111531191
56 0.791179111531191
57 0.791179111531191
58 0.791179111531191
59 0.791179111531191
60 0.791179111531191
61 0.791179111531191
62 0.791179111531191
63 0.791179111531191
64 0.791179111531191
65 0.791179111531191
66 0.791179111531191
67 0.791179111531191
68 0.791179111531191
69 0.791179111531191
70 0.791179111531191
71 0.791179111531191
72 0.791179111531191
73 0.791179111531191
74 0.791179111531191
75 0.791179111531191
76 0.791179111531191
77 0.791179111531191
78 0.791179111531191
79 0.791179111531191
80 0.791179111531191
81 0.791179111531191
82 0.791179111531191
83 0.791179111531191
84 0.791179111531191
85 0.791179111531191
86 0.791179111531191
87 0.791179111531191
88 0.791179111531191
89 0.791179111531191
90 0.791179111531191
91 0.791179111531191
92 0.791179111531191
93 0.791179111531191
94 0.791179111531191
95 0.791179111531191
96 0.791179111531191
97 0.791179111531191
98 0.791179111531191
99 0.791179111531191
};
\addplot [thick, firebrick166640, mark=*, mark size=\ms, mark options={solid}, only marks]
table {%
1 -0.666287807183364
2 -0.666287807183365
3 0.035035444234405
4 0.791179111531191
9 -0.666287807183365
10 -0.577440926275993
11 -0.577440926275992
12 0.0350354442344048
13 0.791179111531191
};
\end{groupplot}

\end{tikzpicture}

%% file: media/tikz/ex_2_r.tex
\begin{tikzpicture}[spy using outlines={circle,yellow,magnification=5,size=1.5cm, connect spies}]

\definecolor{darkgray178}{RGB}{178,178,178}
\definecolor{silver188}{RGB}{188,188,188}
\definecolor{steelblue52138189}{RGB}{52,138,189}
\definecolor{whitesmoke238}{RGB}{238,238,238}

\begin{axis}[
height=0.8*\axisdefaultheight,
width=\axisdefaultwidth,
axis background/.style={fill=whitesmoke238},
axis line style={silver188},
tick pos=left,
tick scale binop=\times,
x grid style={darkgray178},
xlabel={\(\displaystyle \tau\)},
xmajorgrids,
xmin=-1, xmax=105,
xtick style={color=black},
xtick={0,25,50,75,100},
xticklabels={
  \(\displaystyle {0}\),
  \(\displaystyle {25}\),
  \(\displaystyle {50}\),
  \(\displaystyle {75}\),
  \(\displaystyle {100}\),
},
y grid style={darkgray178},
ymode=log,
ylabel={\(\displaystyle r\)},
ymajorgrids,
ymin=-0.0001, ymax=1000000,
ytick style={color=black},
ytick={0.001, 0.1, 10, 1000, 100000},
yticklabels={
  \(\displaystyle {10^{-3}}\),
	\(\displaystyle {10^{-1}}\),
	\(\displaystyle {10^{1}}\),
	\(\displaystyle {10^{3}}\),
	\(\displaystyle {10^{5}}\),
}
]
\addplot [thick, steelblue52138189]
table {%
0 70098.4483934242
1 18060.2060026133
2 37721.2327820656
3 22829.9474229998
4 6070.87083659955
5 4157.23660826791
6 5371.58155377254
7 3649.29445520271
8 1613.75296864981
9 698.999447800056
10 902.218355049472
11 817.932650417223
12 574.739125330053
13 312.90688051885
14 358.519458205454
15 311.860854530968
16 227.730882569314
17 159.843082221581
18 169.989201616609
19 153.898245713322
20 125.322258775374
21 99.9349323652241
22 94.8679295146564
23 84.5359221542355
24 72.1445290112954
25 61.2645693959732
26 55.9632625998163
27 49.7515366657378
28 43.4054382501298
29 37.7988939580427
30 33.9191002695391
31 30.1067059594786
32 26.5190019101399
33 23.3428916031171
34 20.7983128492944
35 18.4486208495789
36 16.3132803351646
37 14.4212753016452
38 12.8118194654644
39 11.3614343962438
40 10.0624789018989
41 8.91106474353908
42 7.90716759675147
43 7.01121776720203
44 6.21356763521969
45 5.50640682727337
46 4.88363695589797
47 4.32999625401594
48 3.83829455359789
49 3.40234371857958
50 3.01686843628078
51 2.67471972289093
52 2.37114806993686
53 2.10198850013094
54 1.86360155017024
55 1.65214287714358
56 1.46460054704783
57 1.29831586409595
58 1.15094410373717
59 1.0202529476207
60 0.904361781228277
61 0.801606158647823
62 0.710513567140505
63 0.629739810524811
64 0.558117814389686
65 0.494613559999201
66 0.438311043746274
67 0.388388449743935
68 0.344123298235429
69 0.304875179530534
70 0.270076459741443
71 0.239221464678928
72 0.21186336367146
73 0.187606020571828
74 0.166098248233139
75 0.147028061220584
76 0.130119221741106
77 0.115126813500497
78 0.101833691816995
79 0.0900471732571637
80 0.0795965096816027
81 0.0703303084678081
82 0.0621143387453949
83 0.0548295467785804
84 0.0483703959813897
85 0.0426433121611834
86 0.0375653275672407
87 0.0330628686308004
88 0.0290707066368648
89 0.0255310060627573
90 0.0223924866420935
91 0.0196096778331973
92 0.0171422642002852
93 0.0149544996737696
94 0.0130146896139714
95 0.0112947310991074
96 0.00976970669431651
97 0.00841752309995927
98 0.00721859103203786
99 0.00622641148424015
};
\end{axis}
\end{tikzpicture}

%% file: media/tikz/ex_2_J.tex
\begin{tikzpicture}

\definecolor{darkgray178}{RGB}{178,178,178}
\definecolor{darkolivegreen7012033}{RGB}{70,120,33}
\definecolor{firebrick166640}{RGB}{166,6,40}
\definecolor{lightgray204}{RGB}{204,204,204}
\definecolor{silver188}{RGB}{188,188,188}
\definecolor{slategray122104166}{RGB}{122,104,166}
\definecolor{steelblue52138189}{RGB}{52,138,189}
\definecolor{whitesmoke238}{RGB}{238,238,238}

\begin{axis}[
axis background/.style={fill=whitesmoke238},
axis line style={silver188},
legend cell align={left},
legend columns=4,
legend style={
  fill opacity=0.8,
  draw opacity=1,
  text opacity=1,
  at={(0.05,0.97)},
  anchor=north west,
  draw=lightgray204,
  fill=whitesmoke238,
},
tick pos=left,
tick scale binop=\times,
x grid style={darkgray178},
xlabel={Time step \(\displaystyle t\)},
xmajorgrids,
xmin=-4.95, xmax=103.95,
xtick style={color=black},
xtick={-20,0,20,40,60,80,100,120},
xticklabels={
  \(\displaystyle {\ensuremath{-}20}\),
  \(\displaystyle {0}\),
  \(\displaystyle {20}\),
  \(\displaystyle {40}\),
  \(\displaystyle {60}\),
  \(\displaystyle {80}\),
  \(\displaystyle {100}\),
  \(\displaystyle {120}\)
},
y grid style={darkgray178},
ylabel={\(J(t)\)},
ymajorgrids,
ymin=-19116.8971478869, ymax=600000.840105626,
ytick style={color=black},
ytick={-50000,0,100000,200000,300000,400000},
yticklabels={
  \(\displaystyle {\ensuremath{-}50000}\),
  \(\displaystyle {0}\),
  \(\displaystyle {1}\),
  \(\displaystyle {2}\),
  \(\displaystyle {3}\),
  \(\displaystyle {4}\),
}
]
\addplot [thick, steelblue52138189]
table {%
0 0
1 8669.93014433033
2 19185.045091873
3 31454.221048256
4 44883.3285964761
5 58743.2981850602
6 72458.3296232018
7 85558.5362225148
8 97650.334104394
9 108545.269381531
10 118209.03678289
11 126659.983010788
12 134100.368559446
13 140846.010125882
14 147103.775936555
15 152953.537648709
16 158408.431181484
17 163450.039696881
18 168059.400063012
19 172221.175654458
20 175920.704490848
21 179154.428034525
22 181926.768394096
23 184249.080683794
24 186146.106191951
25 187654.098828333
26 188811.895029335
27 189661.055282155
28 190245.268887744
29 190608.848087664
30 190805.252049325
31 190892.946321078
32 190922.338156544
33 190929.76735707
34 190937.118609834
35 190954.606492439
36 190983.197142984
37 191019.331598449
38 191055.52829095
39 191081.949237984
40 191097.879441497
41 191106.52148339
42 191110.755332156
43 191112.997108807
44 191114.222171336
45 191114.968654529
46 191115.491648958
47 191115.905950114
48 191116.26421551
49 191116.59267587
50 191116.905605323
51 191117.211062109
52 191117.513481397
53 191117.815123001
54 191118.116985508
55 191118.419395388
56 191118.722361672
57 191119.02577416
58 191119.329500897
59 191119.633429031
60 191119.937475753
61 191120.2415863
62 191120.545727415
63 191120.849880579
64 191121.154036485
65 191121.458191132
66 191121.762343294
67 191122.066493025
68 191122.370640837
69 191122.674787311
70 191122.978932933
71 191123.283078057
72 191123.587222915
73 191123.891367647
74 191124.195512331
75 191124.499657005
76 191124.803801684
77 191125.107946375
78 191125.412091077
79 191125.716235787
80 191126.020380502
81 191126.324525222
82 191126.628669943
83 191126.932814665
84 191127.236959388
85 191127.541104111
86 191127.845248834
87 191128.149393557
88 191128.45353828
89 191128.757683003
90 191129.061827726
91 191129.365972449
92 191129.670117172
93 191129.974261895
94 191130.278406618
95 191130.58255134
96 191130.886696063
97 191131.190840786
98 191131.494985509
99 191131.799130232
};
\addlegendentry{C}
\addplot [thick, firebrick166640, dashed]
table {%
0 0
1 8669.61093189811
2 19241.1055177235
3 31734.7898990688
4 45560.6683130185
5 59939.494371905
6 74270.0958102035
7 88077.6349805796
8 100921.316118859
9 112525.192022741
10 122844.639694221
11 131984.457168839
12 140197.788579889
13 147798.515682289
14 154966.193357613
15 161761.375476229
16 168184.211321482
17 174207.013020231
18 179794.790296929
19 184923.314981813
20 189572.484430516
21 193729.971565014
22 197390.492913825
23 200556.324765897
24 203237.407195891
25 205457.84299968
26 207253.726260048
27 208663.743297196
28 209729.440824413
29 210494.612491484
30 211004.635571105
31 211306.807563536
32 211461.333805369
33 211525.524557629
34 211544.997317481
35 211550.605362457
36 211559.513390414
37 211578.852907437
38 211607.100132982
39 211640.006171005
40 211669.562470597
41 211688.010830333
42 211698.634909327
43 211704.091223817
44 211706.956487305
45 211708.474198164
46 211709.352247242
47 211709.935606277
48 211710.379127402
49 211710.752326768
50 211711.088471172
51 211711.405175586
52 211711.712251193
53 211712.015175879
54 211712.316798901
55 211712.618450455
56 211712.920647676
57 211713.223506884
58 211713.52685626
59 211713.830513236
60 211714.134450773
61 211714.438556466
62 211714.742744317
63 211715.046917578
64 211715.351158086
65 211715.655355543
66 211715.959563451
67 211716.263768303
68 211716.567985784
69 211716.872194553
70 211717.176380601
71 211717.480593799
72 211717.78483412
73 211718.089074395
74 211718.393263759
75 211718.697478894
76 211719.001664141
77 211719.305847621
78 211719.610102557
79 211719.91428966
80 211720.218500798
81 211720.522683224
82 211720.826928069
83 211721.131127498
84 211721.435350495
85 211721.739549492
86 211722.043760493
87 211722.34796609
88 211722.652198136
89 211722.956394618
90 211723.2606305
91 211723.564822586
92 211723.868977585
93 211724.173205849
94 211724.477420872
95 211724.781663245
96 211725.08585058
97 211725.390048348
98 211725.694254762
99 211725.998485489
};
\addlegendentry{SwA}
\addplot [thick, slategray122104166]
table {%
0 0
1 8671.02804261011
2 19429.123099329
3 32759.0088218614
4 48597.9434918087
5 66474.0302973831
6 85671.753528458
7 105308.891186085
8 124570.381824187
9 142938.859709524
10 160252.156658119
11 176812.63780057
12 193070.664799114
13 209313.923320428
14 225535.092061488
15 241417.847063167
16 256532.630009156
17 270443.674271289
18 282799.368051647
19 293422.619689678
20 302439.351336808
21 310160.014281024
22 316928.513686994
23 323071.159794503
24 328752.035310501
25 333918.014202223
26 338419.322976549
27 342097.95409553
28 344840.462685784
29 346631.42492173
30 347614.332612168
31 348063.405312472
32 348299.018175525
33 348553.767077119
34 348969.597783513
35 349617.576239615
36 350471.738968569
37 351474.924107443
38 352571.691782823
39 353729.162794815
40 354897.708401047
41 356011.048866839
42 357002.19652194
43 357784.064951226
44 358324.2361675
45 358611.825060909
46 358721.060508691
47 358749.458997183
48 358754.385279057
49 358755.359363665
50 358755.692159037
51 358756.045714098
52 358756.349698355
53 358756.660671355
54 358756.965305021
55 358757.270673025
56 358757.575034308
57 358757.879400541
58 358758.18358044
59 358758.487733775
60 358758.791847719
61 358759.095950577
62 358759.400044211
63 358759.704134303
64 358760.008222086
65 358760.312308834
66 358760.616394981
67 358760.920480837
68 358761.224566533
69 358761.528652149
70 358761.832737721
71 358762.136823272
72 358762.440908811
73 358762.744994345
74 358763.049079875
75 358763.353165403
76 358763.657250931
77 358763.961336458
78 358764.265421985
79 358764.569507511
80 358764.873593038
81 358765.177678565
82 358765.481764091
83 358765.785849618
84 358766.089935145
85 358766.394020671
86 358766.698106198
87 358767.002191725
88 358767.306277251
89 358767.610362778
90 358767.914448305
91 358768.218533831
92 358768.522619358
93 358768.826704884
94 358769.130790411
95 358769.434875938
96 358769.738961464
97 358770.043046991
98 358770.347132517
99 358770.651218044
};
\addlegendentry{S}
\addplot [thick, darkolivegreen7012033, dashed]
table {%
0 0
1 8671.02804260986
2 19429.123099324
3 32759.0088218401
4 48597.9434921073
5 66474.030298405
6 85671.7535413328
7 105308.038380432
8 124551.838904868
9 142862.673318626
10 160078.417286171
11 176501.063259837
12 192579.874675065
13 208601.461209217
14 224584.491206218
15 240249.444002186
16 255168.944618075
17 268897.684287214
18 281067.869679734
19 291421.538929385
20 299985.605767411
21 307075.15330186
22 313082.311114365
23 318377.931358476
24 323258.521669245
25 327846.611405925
26 332052.872872036
27 335702.904570813
28 338638.230818906
29 340768.465058097
30 342151.873821483
31 342975.817025238
32 343456.975536078
33 343811.82283257
34 344201.164917207
35 344737.871460491
36 345499.52021877
37 346459.250926255
38 347538.477325595
39 348705.215114581
40 349915.845417463
41 351092.083895723
42 352137.402080817
43 352966.268939127
44 353545.339169281
45 353858.716995109
46 353980.170640323
47 354012.026360808
48 354017.446897113
49 354018.531182877
50 354018.854591927
51 354019.254031401
52 354019.593635634
53 354019.959176836
54 354020.318756923
55 354020.68301036
56 354021.046848975
57 354021.411533324
58 354021.776259356
59 354022.141149992
60 354022.506071704
61 354022.871028564
62 354023.235995949
63 354023.60097133
64 354023.965949747
65 354024.330930059
66 354024.695911197
67 354025.060892792
68 354025.425874602
69 354025.790856525
70 354026.155838018
71 354026.520820625
72 354026.885801426
73 354027.250783465
74 354027.615764761
75 354027.980746017
76 354028.345727187
77 354028.710708313
78 354029.075689424
79 354029.440670521
80 354029.805651615
81 354030.170632705
82 354030.535613794
83 354030.900594882
84 354031.265576454
85 354031.630557425
86 354031.995539131
87 354032.360520941
88 354032.725502867
89 354033.090484847
90 354033.45546685
91 354033.82044887
92 354034.185430896
93 354034.550412926
94 354034.915394958
95 354035.280376991
96 354035.645359024
97 354036.010341058
98 354036.375323092
99 354036.740305126
};
\addlegendentry{NCA}
\end{axis}

\end{tikzpicture}

%% file: media/tikz/ex_2_n.tex
\begin{tikzpicture}

\definecolor{darkgray178}{RGB}{178,178,178}
\definecolor{darkolivegreen7012033}{RGB}{70,120,33}
\definecolor{firebrick166640}{RGB}{166,6,40}
\definecolor{lightgray204}{RGB}{204,204,204}
\definecolor{silver188}{RGB}{188,188,188}
\definecolor{slategray122104166}{RGB}{122,104,166}
\definecolor{steelblue52138189}{RGB}{52,138,189}
\definecolor{whitesmoke238}{RGB}{238,238,238}

\newcommand{\ms}{2}
\newcommand{\myP}{on 4pt off 1.6pt}

\begin{groupplot}[group style={group size=1 by 2, vertical sep = 0.3cm}]
\nextgroupplot[
axis background/.style={fill=whitesmoke238},
axis line style={silver188},
legend cell align={left},
legend columns=4,
legend style={
  fill opacity=0.8,
  draw opacity=1,
  text opacity=1,
  at={(0.03,0.97)},
  anchor=north west,
  draw=lightgray204,
  fill=whitesmoke238
},
scaled x ticks=manual:{}{\pgfmathparse{#1}},
tick pos=left,
tick scale binop=\times,
x grid style={darkgray178},
xmajorgrids,
xmin=4.5, xmax=15.5,
xtick style={color=black},
xtick={5, 6, 7, 8, 9, 10, 11, 12, 13, 14, 15},
xticklabels={},
y grid style={darkgray178},
ylabel={\( J(100)\)},
ymajorgrids,
ymin=43206.7613561477, ymax=698546.522685339,
ytick style={color=black},
ytick={0,100000,200000,300000,400000, 500000, 600000},
yticklabels={
  \(\displaystyle {0}\),
  \(\displaystyle {1}\),
  \(\displaystyle {2}\),
  \(\displaystyle {3}\),
  \(\displaystyle {4}\),
  \(\displaystyle {5}\),
  \(\displaystyle {6}\),
}
]
\addplot [steelblue52138189, dash pattern=on 4pt off 1.6pt, mark=o, mark size=2.5, mark options={solid}, line width=0.6]
table {%
5 115195.34705984
10 191132.103274955
15 279363.176291191
};
\addlegendentry{C}
\addplot [firebrick166640, dash pattern=on 4pt off 1.6pt, mark=o, mark size=2.5, mark options={solid} , line width=0.6]
table {%
5 128860.407229092
10 211726.302713371
15 306579.083084007
};
\addlegendentry{SwA}
\addplot [slategray122104166, dash pattern=on 4pt off 1.6pt, mark=o, mark size=2.5, mark options={solid} , line width=0.6]
table {%
5 207756.280917537
10 358770.955303571
15 593234.567476313
};
\addlegendentry{S}
\addplot [darkolivegreen7012033, dash pattern=on 4pt off 1.6pt, mark=o, mark size=2.5, mark options={solid} , line width=0.6]
table {%
5 208841.780338245
10 354037.10528716
15 578054.71592232
};
\addlegendentry{NCA}

\nextgroupplot[
axis background/.style={fill=whitesmoke238},
axis line style={silver188},
log basis y={10},
tick pos=left,
tick scale binop=\times,
x grid style={darkgray178},
xlabel={\(\displaystyle M\)},
xmajorgrids,
xmin=4.5, xmax=15.5,
xtick style={color=black},
xtick={5, 6, 7, 8, 9, 10, 11, 12, 13, 14, 15},
xticklabels={
  \(\displaystyle {5}\),
    \(\),
     \(\),
      \(\),
       \(\),
  \(\displaystyle {10}\),
   \(\),
    \(\),
     \(\),
      \(\),
  \(\displaystyle {15}\),
},
y grid style={darkgray178},
ylabel={Time (s)},
ymajorgrids,
ymin=0.0125231824274463, ymax=131.738807955243,
ymode=log,
ytick style={color=black},
ytick={0.01,0.1,1,10,100,10000,1000000},
yticklabels={
  \( {10^{-2}}\),
  \( {10^{-1}}\),
  \( {10^{0}}\),
   \( {10^{1}}\),
  \( {10^{2}}\),
  \( {10^{4}}\),
  \( {10^{6}}\)
}
]
\addplot [steelblue52138189, dash pattern=on 4pt off 1.6pt, mark=o, mark size=\ms, mark options={solid} , line width=0.6]
table {%
5 0.216540009975433
10 2.074129986763
15 4.3185800075531
};
\addplot [thick, steelblue52138189, mark=triangle, mark size=\ms, mark options={solid}, only marks]
table {%
5 0.0810000896453857
10 0.388000011444092
15 0.131999969482422
};
\addplot [thick, steelblue52138189, mark=triangle, mark size=\ms, mark options={solid,rotate=180}, only marks]
table {%
5 1.56100010871887
10 38.0660002231598
15 59.8619999885559
};
\addplot [steelblue52138189]
table {%
5 1.56100010871887
5 0.0810000896453857
};
\addplot [steelblue52138189]
table {%
10 38.0660002231598
10 0.388000011444092
};
\addplot [steelblue52138189]
table {%
15 59.8619999885559
15 0.131999969482422
};
\addplot [firebrick166640, dash pattern=on 4pt off 1.6pt, mark=o, mark size=\ms, mark options={solid}, line width=0.6]
table {%
5 0.03072056
10 0.030439805
15 0.03467306
};
\addplot [thick, firebrick166640, mark=triangle, mark size=\ms, mark options={solid}, only marks]
table {%
5 0.021632
10 0.021632
15 0.025043
};
\addplot [thick, firebrick166640, mark=triangle, mark size=\ms, mark options={solid,rotate=180}, only marks]
table {%
5 0.049008
10 0.050883
15 0.104262
};
\addplot [firebrick166640]
table {%
5 0.049008
5 0.021632
};
\addplot [firebrick166640]
table {%
10 0.050883
10 0.021632
};
\addplot [firebrick166640]
table {%
15 0.104262
15 0.025043
};
\addplot [slategray122104166, dash pattern=on 4pt off 1.6pt, mark=o, mark size=\ms, mark options={solid}, line width=0.6]
table {%
5 0.0634800052642822
10 0.173880002498627
15 0.188099982738495
};
\addplot [thick, slategray122104166, mark=triangle, mark size=\ms, mark options={solid}, only marks]
table {%
5 0.0380001068115234
10 0.108999967575073
15 0.0999996662139893
};
\addplot [thick, slategray122104166, mark=triangle, mark size=\ms, mark options={solid,rotate=180}, only marks]
table {%
5 0.0969998836517334
10 0.331000089645386
15 0.371999740600586
};
\addplot [slategray122104166]
table {%
5 0.0969998836517334
5 0.0380001068115234
};
\addplot [slategray122104166]
table {%
10 0.331000089645386
10 0.108999967575073
};
\addplot [slategray122104166]
table {%
15 0.371999740600586
15 0.0999996662139893
};
\addplot [darkolivegreen7012033, dash pattern=on 4pt off 1.6pt, mark=o, mark size=\ms, mark options={solid}, line width=0.6]
table {%
5 2.38537132740021
10 3.23639132738113
15 3.71080122709274
};
\addplot [thick, darkolivegreen7012033, mark=triangle, mark size=\ms, mark options={solid}, only marks]
table {%
5 1.35300278663635
10 2.11699986457825
15 1.94700217247009
};
\addplot [thick, darkolivegreen7012033, mark=triangle, mark size=\ms, mark options={solid,rotate=180}, only marks]
table {%
5 4.66299962997437
10 6.48799848556519
15 5.6150016784668
};
\addplot [darkolivegreen7012033]
table {%
5 4.66299962997437
5 1.35300278663635
};
\addplot [darkolivegreen7012033]
table {%
10 6.48799848556519
10 2.11699986457825
};
\addplot [darkolivegreen7012033]
table {%
15 5.6150016784668
15 1.94700217247009
};
\end{groupplot}

\end{tikzpicture}